\documentclass[twoside,11pt]{article}
\usepackage[round]{natbib}
\usepackage[T1]{fontenc}
\usepackage[utf8]{inputenc}
\usepackage{amsmath}
\usepackage{amsthm}
\usepackage{amssymb}
\usepackage{adjustbox}
\usepackage[title]{appendix}
\usepackage[nottoc, notlof, notlot]{tocbibind}
\usepackage{mwe}
\usepackage{bbold}
\usepackage{multirow}
\usepackage[table,xcdraw]{xcolor}
\usepackage{algorithm}
\usepackage{algpseudocode}
\usepackage{quiver}
\usepackage{textcomp}
\usepackage{siunitx}
\usepackage{xfrac}
\usepackage{bbm}
\usepackage{amsthm}
\usepackage{graphicx}
\usepackage{float}
\usepackage{subcaption}
\usepackage[colorlinks=true]{hyperref}
\usepackage{setspace}
\usepackage{tabto}
\usepackage{geometry}
\allowdisplaybreaks

\usepackage{bigfoot}

\DeclareNewFootnote{AAffil}[arabic]
\DeclareNewFootnote{ANote}[fnsymbol]

\usepackage{etoolbox}
\makeatletter
\patchcmd\maketitle{\def\@makefnmark{\rlap{\@textsuperscript{\normalfont\@thefnmark}}}}{}{}{}
\makeatother

\makeatletter
\def\thanksAAffil#1{
  \footnotemarkAAffil\protected@xdef\@thanks{\@thanks%
        \protect\footnotetextAAffil[\the \c@footnoteAAffil]{#1}}%
}
\def\thanksANote#1{%
  \footnotemarkANote%
  \protected@xdef\@thanks{\@thanks%
        \protect\footnotetextANote[\the \c@footnoteANote]{#1}}%
}
\makeatother

\theoremstyle{definition}
\newtheorem{defi}{Definition}
\theoremstyle{plain}
\newtheorem{thm}{Theorem}

\newtheorem{lmm}{Lemma}
\theoremstyle{remark}

\theoremstyle{plain}
\newtheorem{prp}{Proposition}
\theoremstyle{plain}

\theoremstyle{plain}

\newtheorem{innercustomthm}{Theorem}
\newenvironment{customthm}[1]
  {\renewcommand\theinnercustomthm{#1}\innercustomthm}
  {\endinnercustomthm}

\begin{document}
\title{Persistence Diagrams Estimation of Multivariate Piecewise Hölder-continuous Signals}
\author{Hugo Henneuse \thanksAAffil{Laboratoire de Mathématiques d'Orsay, Université Paris-Saclay, Orsay, France}$^{\text{ ,}}$\thanksAAffil{DataShape, Inria Saclay, Palaiseau, France}\\ \href{mailto:hugo.henneuse@universite-paris-saclay.fr}{hugo.henneuse@universite-paris-saclay.fr}}
\maketitle
\begin{abstract}
To our knowledge, the analysis of convergence rates for persistence diagrams estimation from noisy signals has predominantly relied on lifting signal estimation results through sup-norm (or other functional norm) stability theorems. We believe that moving forward from this approach can lead to considerable gains. We illustrate it in the setting of nonparametric regression. From a minimax perspective, we examine the inference of persistence diagrams (for the sublevel sets filtration). We show that for piecewise Hölder-continuous functions, with control over the reach of the set of discontinuities, taking the persistence diagram coming from a simple histogram estimator of the signal permits achieving the minimax rates known for Hölder-continuous functions. 
\end{abstract}
\section*{Introduction}
Inferring information from noisy real-valued signals is a central subject in statistics. Specifically, in the nonparametric regression setting, where the signal is observed discretely and up to additive noise, the recovery of the whole signal structure has been extensively studied by the nonparametric statistics community. When the signal is regular (e.g., belonging to a Hölder, Sobolev, or Besov space), rigorous minimax studies and tractable optimal procedures have been provided, forming a nearly exhaustive benchmark. For an overview, see \citet{TsybakovBook}. Nonetheless, in many practical scenarios, signals are not so regular. For instance, a photograph of a natural scene often exhibits smooth variations within homogeneous regions (such as sky, sea, or foliage), but the boundaries of the objects introduce sharp discontinuities. This has motivated the study of broader classes of irregular signals, typically signals that are only piecewise continuous. This setting has been explored in several subsequent works. We refer the reader to \citet{bookQiu} for an overview. However, proposed methods suffer from certain limitations: strong additional knowledge assumptions (e.g., assuming knowledge of the number of jumps, their locations, or their magnitudes), limited to low-dimensional cases (only univariate or bivariate signals), high computational costs, or lack of rigorous and general statistical guarantees. These challenges motivate the exploration of looser descriptors that can be inferred more easily. \\\\
Over the past two decades, Topological Data Analysis (TDA) has emerged as a powerful framework, introducing geometric and topological tools to better understand complex data. Among these tools, persistent homology has received particular attention. By encoding multiscale topological features in the form of persistence diagrams (or barcodes), it provides a rich and compact summary that has proven to be a versatile descriptor, valuable from both practical and theoretical perspectives. Its applications encompass image analysis, time-series analysis, network analysis, clustering and classification tasks, and much more, as highlighted in \cite{Harer08} and \cite{chazal2021introduction}.\\\\
A growing body of work has focused on the probabilistic and statistical analysis of persistence diagrams or related objects. Importantly, this has led to the establishment of several limit theorems \citep{Bubenik15,Hiraoka16,Divol18,Divol21} and descriptions of persistence diagrams (or related quantities) arising from important random objects such as Brownian motion, fractional Brownian motion, and Lévy processes \citep{Divol18,Perez2020,perez2022,PerezThese}, point processes \citep{Kahle2009,Adler12,Divol18}, and Gaussian fields \citep{Adler10,Bobrowski12,Pranav19,Pranav2021,Masoomy21}.\\\\
Beyond these asymptotic and descriptive results, another major line of research has focused on the estimation of persistence diagrams from data. In this context, the density model, in which one observes a point cloud formed by i.i.d. samples $X_1, \dots, X_n$ from a probability distribution, has become a central framework. This line of research was initiated in a parametric setting by \citet{BubenikKim07}, and has since been extended to more general nonparametric frameworks \citep{Chazal11, Fasy, ChazalGlisseMichel, Chazal13}. These works introduce various geometric assumptions on the distribution and its support, which enable consistent recovery of persistence diagrams and quantification of convergence rates.\\
More closely related to our setting and to signal processing, \cite{BCL2009} and \cite{Chung2009} tackle the problem of estimating persistence diagrams from the superlevel sets (or, equivalently, sublevel sets) in the nonparametric regression setting. In particular, in both papers, they show that a plug-in estimator based on a kernel estimate of the signal converges at rate $O((\log(n)/n)^{\alpha/(2\alpha+d)})$ over the classes of $(L,\alpha)-$Hölder signals on $\mathbb{R}^{d}$. \\\\
The general approach followed in all of these works involves estimating the signal or the density, quantifying the estimation error in sup-norm, Hausdorff distance, or Gromov-Hausdorff distance, and bounding the bottleneck error on the diagram using stability theorems. The power and importance of stability theorems are evident as they enable the direct translation of convergence rates in sup-norm (or similar metrics) to convergence rates in bottleneck distance over diagrams. To further emphasize the importance of stability theorems, the convergence rates derived from such results are generally minimax-optimal for regular function classes. For instance, in the density model, this has been demonstrated in \cite{ChazalGlisseMichel} and \cite{Fasy}. Similarly, in the nonparametric regression setting, although it is not explicitly stated in the original papers, the rates obtained in \cite{BCL2009} and \cite{Chung2009} can also be shown to be minimax-optimal (see, e.g., Section \ref{section LB}).\\\\
However, stability-based approaches may come at the expense of generality. While persistence diagrams offer a robust summary of topological features, they encode significantly less information than the full signal. Moreover, noise does not necessarily affect persistence diagrams in the same way it impacts the raw data, meaning that sup-norm stability can sometimes obscure the diagrams’ intrinsic robustness. In this direction, \cite{Turkes21} offers compelling empirical evidence on how common types of noise influence persistence diagrams in image analysis. Building on the idea that persistence diagrams can, in some cases, be easier to infer than the full signal, \cite{Bobrowski} explores this perspective in both the density estimation and nonparametric regression settings. Departing from traditional approaches, the authors propose a novel estimation method based on image persistence and show that it achieves quasi-consistency under minimal assumptions, thereby allowing for the consideration of broader function classes. We believe that these initial results highlight the necessity of moving away from sup-norm stability and, more broadly, emphasize the value of using topological or geometric descriptors when conventional nonparametric techniques yield unsatisfactory results. Nevertheless, their framework appears too broad to quantify convergence rates or even prove proper consistency, motivating the identification of more restricted function classes for which stronger and more precise results can be obtained. The present work aims to address this by studying the problem of estimating persistence diagrams in the nonparametric regression framework, over broad classes of piecewise-continuous signals. As highlighted previously, such classes of signals are ubiquitous in practical applications, particularly when dealing with image or voxel data, and estimating the full structure of these signals can be especially challenging. In such contexts, having access to descriptors, like persistence diagrams, that are easier to infer becomes particularly valuable.

\subsection*{Framework}
\textbf{Regularity assumptions.} For a set $A\subset[0,1]^{d}$, we denote $\overline{A}$ its closure, $A^{\circ}$ its interior, $\partial A$ its boundary and $A^{c}$ its complement. Let $f:[0,1]^{d}\rightarrow \mathbb{R}$, we make the following assumptions about $f$:\\\\
\textbf{A0.} $f$ is bounded by $M$, i.e., for all $x\in[0,1]^{d}$, $|f(x)|\leq M$.\\\\
\textbf{A1.} $f$ is a piecewise $(L,\alpha)-$Hölder-continuous function, i.e., there exist $M_{1},...,M_{l}$ disjoint open sets of $[0,1]^{d}$ such that
$$\bigcup_{i=1}^{l}\overline{M_{i}}=[0,1]^{d}$$ and for all $i\in\{1,...,l\}$ and $x,y\in M_{i}$,
$$|f(x)-f(y)|\leq L\|x-y\|_{2}^{\alpha}.$$
\textbf{A2}. $f$ satisfies, for all $x_{0}\in [0,1]^{d}$,
$$\underset{x\in \bigcup_{i=1}^{l}M_{i}\rightarrow x_{0}}{\liminf}f(x)=f(x_{0})$$
In this context, two signals, differing only on a null set, are statistically indistinguishable. Persistent homology is sensitive to point-wise irregularity. As a result, two signals differing only on a null set can have very different persistence diagrams. Assumption \textbf{A2} prevents such scenarios. Furthermore, note that for any piecewise Hölder-continuous function $f$, there exists a modification $\Tilde{f}$ satisfying Assumption \textbf{A2} such that $f$ and $\Tilde{f}$ coincide except on a set of measure zero. Following this remark, note that, although this assumption is necessary for the proofs of our main results, it is more of a convention to avoid statistical indistinguishability than a true limitation, at least from a practical perspective.\\\\
\textbf{A3.} $\bigcup_{i=1}^{l}\partial M_{i}\cap]0,1[^{d}$ is a $C^{1,1}$ hypersurface, and, for $R>0$, it satisfies:
$$\operatorname{reach}\left(]0,1[^{d}\cap\bigcup_{i=1}^{l}\partial M_{i}\right)\geq R \text{ and } d_{2}\left(\bigcup_{i=1}^{l}\partial M_{i}\cap]0,1[^{d},\partial [0,1]^{d}\right)\geq R$$
where, for a set $A\subset \mathbb{R}^{d}$, $\operatorname{reach}(A)=$
$$\sup \left\{r \in \mathbb{R}: \forall x \in \mathbb{R}^d \backslash A \text { with } d_{2}(\{x\}, A)<r,\exists!y \in A \text { s.t. } ||x-y||_{2}=d_{2}(\{x\}, A)\right\}$$
and,
    $$d_{2}(A,B)=\max\left(\sup_{x\in B}\inf\limits_{y\in A}||x-y||_{2},\sup_{x\in A}\inf\limits_{y\in B}||x-y||_{2} \right).$$
is the Hausdorff distance between two sets $A$ and $B$. The reach is a curvature measure introduced by \cite{Fed59}. An intuitive way to understand it is that if $A$ has a reach $R$, we can roll a ball of radius $R$ along the boundary of $A$. Positive reach assumptions are fairly common in statistical TDA \citep{pmlr-v22-balakrishnan12a,NPS08} and geometric inference \citep{JMLR:v13:genovese12a, KRW19, AL19,AKCMRW19, Berenfeld_2021}. Here, the first part of Assumption \textbf{A3} ensures geometric control over the union of the boundary of the $M_{i}$, $i\in\{1,...,l\}$, in the interior of $[0,1]^{d}$, for example, it prevents the appearance of cusps, corners, and multiple points. The second part ensures that discontinuities do not appear too close to the boundary of the cube $[0,1]^{d}$.\\\\
\begin{figure}[h]
\centering
\begin{subfigure}{.5\textwidth}
  \centering
\begin{tikzpicture}[x=0.75pt,y=0.75pt,yscale=-1,xscale=1]

\draw   (248.6,57.2) -- (442,57.2) -- (442,250.6) -- (248.6,250.6) -- cycle ;
\draw  [dash pattern={on 4.5pt off 4.5pt}] (316,88.6) .. controls (331,70.6) and (388,54) .. (368,74) .. controls (348,94) and (346,103.6) .. (366,133.6) .. controls (386,163.6) and (309,202.6) .. (308,168.6) .. controls (307,134.6) and (258,120.6) .. (257,108.6) .. controls (256,96.6) and (255.5,92.15) .. (266.75,83.38) .. controls (278,74.6) and (301,106.6) .. (316,88.6) -- cycle ;
\draw  [dash pattern={on 4.5pt off 4.5pt}] (382,142.6) .. controls (402,132.6) and (445,131.6) .. (433,148.6) .. controls (421,165.6) and (439,209.6) .. (427,222.6) .. controls (415,235.6) and (408,206.2) .. (390,235.6) .. controls (372,265) and (362,152.6) .. (382,142.6) -- cycle ;
\draw  [dash pattern={on 4.5pt off 4.5pt}] (268,176) .. controls (278,171) and (312,170.6) .. (292,190.6) .. controls (272,210.6) and (332,180.6) .. (352,210.6) .. controls (372,240.6) and (301,256.6) .. (281,226.6) .. controls (261,196.6) and (253,208.5) .. (253,198.5) .. controls (253,188.5) and (258,181) .. (268,176) -- cycle ;
\draw  [dash pattern={on 4.5pt off 4.5pt}] (379.4,189.8) .. controls (379.4,178.31) and (388.71,169) .. (400.2,169) .. controls (411.69,169) and (421,178.31) .. (421,189.8) .. controls (421,201.29) and (411.69,210.6) .. (400.2,210.6) .. controls (388.71,210.6) and (379.4,201.29) .. (379.4,189.8) -- cycle ;
\draw  [dash pattern={on 4.5pt off 4.5pt}] (348,140.6) .. controls (346,126.6) and (349,110.6) .. (322,105.6) .. controls (295,100.6) and (270,100.6) .. (283,107.6) .. controls (296,114.6) and (350,154.6) .. (348,140.6) -- cycle ;

\draw (314,109.4) node [anchor=north west][inner sep=0.75pt]    {$M_{1}$};
\draw (317,152.4) node [anchor=north west][inner sep=0.75pt]    {$M_{2}$};
\draw (295,210.4) node [anchor=north west][inner sep=0.75pt]    {$M_{3}$};
\draw (389,179.4) node [anchor=north west][inner sep=0.75pt]    {$M_{4}$};
\draw (388,91.4) node [anchor=north west][inner sep=0.75pt]    {$M_{6}$};
\draw (396,142.4) node [anchor=north west][inner sep=0.75pt]    {$M_{5}$};

\end{tikzpicture}

\caption{Assumption \textbf{A3} verified}
\label{fig:sub1}
\end{subfigure}%
\begin{subfigure}{.5\textwidth}
\centering
\begin{tikzpicture}[x=0.75pt,y=0.75pt,yscale=-1,xscale=1]

\draw  [draw opacity=0][fill={rgb, 255:red, 208; green, 2; blue, 27 }  ,fill opacity=1 ] (295.36,200.6) .. controls (295.36,198.81) and (296.81,197.36) .. (298.6,197.36) .. controls (300.39,197.36) and (301.84,198.81) .. (301.84,200.6) .. controls (301.84,202.39) and (300.39,203.84) .. (298.6,203.84) .. controls (296.81,203.84) and (295.36,202.39) .. (295.36,200.6) -- cycle ;
\draw  [draw opacity=0][fill={rgb, 255:red, 208; green, 2; blue, 27 }  ,fill opacity=1 ] (288.47,119.47) .. controls (288.47,117.68) and (289.92,116.23) .. (291.71,116.23) .. controls (293.5,116.23) and (294.96,117.68) .. (294.96,119.47) .. controls (294.96,121.26) and (293.5,122.71) .. (291.71,122.71) .. controls (289.92,122.71) and (288.47,121.26) .. (288.47,119.47) -- cycle ;
\draw  [draw opacity=0][fill={rgb, 255:red, 208; green, 2; blue, 27 }  ,fill opacity=1 ] (336.76,57.6) .. controls (336.76,55.81) and (338.21,54.36) .. (340,54.36) .. controls (341.79,54.36) and (343.24,55.81) .. (343.24,57.6) .. controls (343.24,59.39) and (341.79,60.84) .. (340,60.84) .. controls (338.21,60.84) and (336.76,59.39) .. (336.76,57.6) -- cycle ;
\draw  [draw opacity=0][fill={rgb, 255:red, 208; green, 2; blue, 27 }  ,fill opacity=1 ] (245.33,111.59) .. controls (245.33,109.8) and (246.78,108.35) .. (248.57,108.35) .. controls (250.36,108.35) and (251.81,109.8) .. (251.81,111.59) .. controls (251.81,113.38) and (250.36,114.83) .. (248.57,114.83) .. controls (246.78,114.83) and (245.33,113.38) .. (245.33,111.59) -- cycle ;
\draw  [draw opacity=0][fill={rgb, 255:red, 208; green, 2; blue, 27 }  ,fill opacity=1 ] (367.9,57.59) .. controls (367.9,55.8) and (369.35,54.35) .. (371.14,54.35) .. controls (372.93,54.35) and (374.38,55.8) .. (374.38,57.59) .. controls (374.38,59.38) and (372.93,60.83) .. (371.14,60.83) .. controls (369.35,60.83) and (367.9,59.38) .. (367.9,57.59) -- cycle ;
\draw  [draw opacity=0][fill={rgb, 255:red, 208; green, 2; blue, 27 }  ,fill opacity=1 ] (429.62,116.5) .. controls (429.62,114.71) and (431.07,113.26) .. (432.86,113.26) .. controls (434.65,113.26) and (436.1,114.71) .. (436.1,116.5) .. controls (436.1,118.29) and (434.65,119.74) .. (432.86,119.74) .. controls (431.07,119.74) and (429.62,118.29) .. (429.62,116.5) -- cycle ;
\draw   (248.6,57.2) -- (442,57.2) -- (442,250.6) -- (248.6,250.6) -- cycle ;
\draw  [dash pattern={on 4.5pt off 4.5pt}]  (250,112) .. controls (314,143.6) and (300,87.6) .. (340,57.6) ;
\draw  [dash pattern={on 4.5pt off 4.5pt}]  (248.6,200.6) -- (298.6,200.6) ;
\draw  [dash pattern={on 4.5pt off 4.5pt}]  (298.6,200.6) -- (298.6,250.6) ;
\draw  [dash pattern={on 4.5pt off 4.5pt}] (352,197.6) .. controls (362,202.6) and (362,238.6) .. (414,149) .. controls (466,59.4) and (394,179) .. (414,209) .. controls (434,239) and (360,259.6) .. (340,229.6) .. controls (320,199.6) and (342,192.6) .. (352,197.6) -- cycle ;
\draw  [dash pattern={on 4.5pt off 4.5pt}] (295,117) .. controls (315,107) and (353,151.6) .. (333,171.6) .. controls (313,191.6) and (315,207) .. (295,177) .. controls (275,147) and (275,127) .. (295,117) -- cycle ;
\draw  [dash pattern={on 4.5pt off 4.5pt}] (365,59) .. controls (385,49) and (423,93.6) .. (403,113.6) .. controls (383,133.6) and (375,190.6) .. (355,160.6) .. controls (335,130.6) and (345,69) .. (365,59) -- cycle ;

\draw (268,72.4) node [anchor=north west][inner sep=0.75pt]    {$M_{1}$};
\draw (265,215.4) node [anchor=north west][inner sep=0.75pt]    {$M_{2}$};
\draw (295,146.4) node [anchor=north west][inner sep=0.75pt]    {$M_{3}$};
\draw (359,101.4) node [anchor=north west][inner sep=0.75pt]    {$M_{4}$};
\draw (371,212.4) node [anchor=north west][inner sep=0.75pt]    {$M_{5}$};
\draw (414,71.4) node [anchor=north west][inner sep=0.75pt]    {$M_{6}$};

\end{tikzpicture}

  \caption{Assumption \textbf{A3} not verified}
  \label{fig:sub2}
\end{subfigure}
\caption{Illustration of Assumption \textbf{A3}}
\label{fig:test}
\end{figure}
We denote $S_{d}(M,L,\alpha,R)$ the set of such functions. The combination of Assumptions \textbf{A1}, \textbf{A2} and \textbf{A3} ensures that the persistence diagrams of $f\in S_{d}(M,L,\alpha,R)$ are well-defined (see Appendix \ref{q-tame appendix}, Proposition \ref{LmmQtame1}).\\\\
\textbf{Statistical model.} We consider the classical nonparametric regression setting with fixed regular grid design, observing $n=N^{d}$ points, 
$$X_{i}=f(x_{i})+\sigma\varepsilon_{i}$$
with $(x_{i})_{1\leq i\leq n}$ the points of the regular grid $G_{1/N}$ over $[0,1]^{d}$ of step size $1/N$, $\sigma$ the level of noise and $(\varepsilon_{i})_{1\leq i\leq n}$ independent standard Gaussian variables (i.e., $\mathcal{N}(0,1)$). This provides a standard and general framework for our analysis. In terms of applications, it is particularly well-suited for modeling noisy images, where the observed points correspond to pixels (or more generally, voxels) observed with additive noise.\\\\
\textbf{Estimator.} In this context, our goal is to estimate $\operatorname{dgm}(f)$, the collection of persistence diagrams associated with the sublevel sets of $f$, by computing the diagram associated with the sublevel sets of a histogram estimator of the signal.\\\\
More formally, let $h>0$ such that $Nh$ is an integer (and so is $nh^{d}$), consider $G_{h}$ the regular orthogonal grid over $[0,1]^{d}$ of step $h$ and $C_{h}$ the collection of all the closed hypercubes of side $h$ composing $G_{h}$.  For all $\lambda\in\mathbb{R}$, we define the estimator of $\mathcal{F}_{\lambda}=f^{-1}(]-\infty,\lambda])$ as follows:
$$\widehat{\mathcal{F}}_{\lambda}=\bigcup\limits_{H\in C_{h,\lambda}}H\text{, with }C_{h,\lambda}=\left\{H\in C_{h}\text{ such that }\frac{1}{nh^{d}}\sum\limits_{x_i\in H}X_{i}\leq \lambda\right\}.$$
the set $\widehat{\mathcal{F}}_{\lambda}$ represents the sublevel set indexed by $\lambda$ of the histogram estimator of $f$ and $nh^{d}$ represents the cardinality of the set $H\cap G_h$ for all $H\in C_{h}$. Next, we consider, for all $s\in\{0,...,d\}$, $\widehat{\mathbb{V}}_{f,s}$ the persistence module induced by the collection of homology groups: $$\left(H_{s}\left(\widehat{\mathcal{F}}_{\lambda}\right)\right)_{\lambda\in\mathbb{R}}$$ 
equipped with the linear maps $\widehat{v}_{\lambda}^{\lambda'}$ induced by the inclusion $\widehat{\mathcal{F}}_{\lambda}\subset \widehat{\mathcal{F}}_{\lambda^{'}}$, $\lambda^{'}\geq\lambda$. We define $\widehat{\operatorname{dgm}(f)}$ as the associated collection of persistence diagrams.
\begin{figure}
\centering
\begin{tikzpicture}[x=0.75pt,y=0.75pt,yscale=-1,xscale=1]

\draw (153.89,114.32) node  {\includegraphics[width=149.91pt,height=126.38pt]{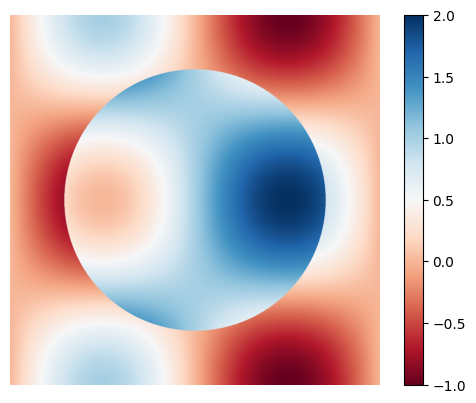}};
\draw (393.27,111.27) node  {\includegraphics[width=149.91pt,height=115.93pt]{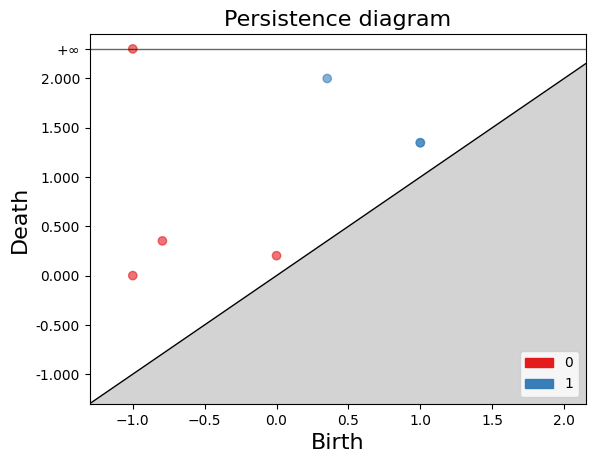}};
\draw  [draw opacity=0][fill={rgb, 255:red, 255; green, 255; blue, 255 }  ,fill opacity=1 ] (302.39,56) -- (305.04,139.5) -- (296.47,139.77) -- (293.81,56.27) -- cycle ;
\draw (154.36,307.73) node  {\includegraphics[width=156.54pt,height=127.97pt]{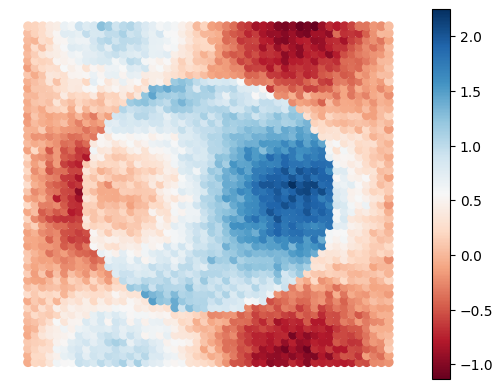}};
\draw (290.17,500.75) node  {\includegraphics[width=149.91pt,height=124.25pt]{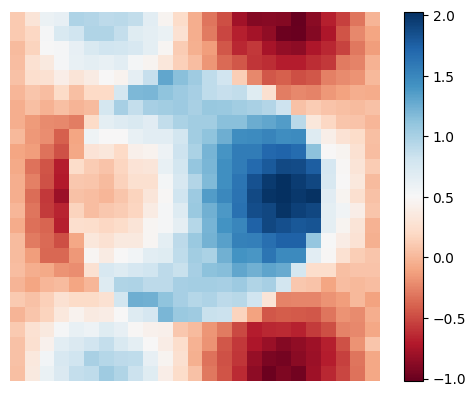}};
\draw (393.27,303.63) node  {\includegraphics[width=149.91pt,height=115.93pt]{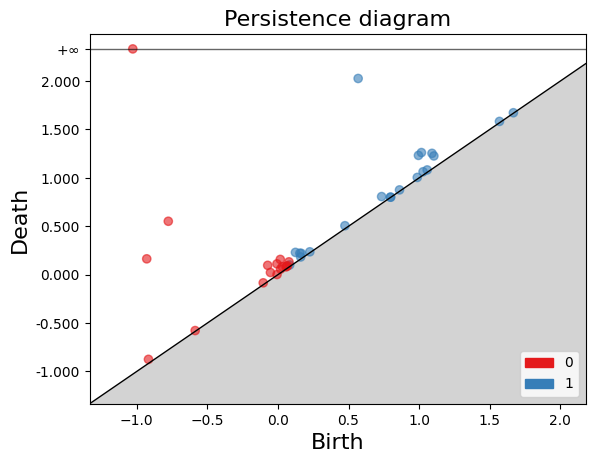}};
\draw  [draw opacity=0][fill={rgb, 255:red, 255; green, 255; blue, 255 }  ,fill opacity=1 ] (304.34,216.43) -- (304.66,283.31) -- (298.29,283.33) -- (297.97,216.46) -- cycle ;
\draw  [draw opacity=0][fill={rgb, 255:red, 255; green, 255; blue, 255 }  ,fill opacity=1 ] (367.47,177.63) -- (434.77,177.63) -- (434.77,186.54) -- (367.47,186.54) -- cycle ;
\draw  [draw opacity=0][fill={rgb, 255:red, 255; green, 255; blue, 255 }  ,fill opacity=1 ] (374.45,370.5) -- (441.75,370.5) -- (441.75,377.29) -- (374.45,377.29) -- cycle ;
\draw  [draw opacity=0][fill={rgb, 255:red, 255; green, 255; blue, 255 }  ,fill opacity=1 ] (447.6,44.45) -- (363.54,45.27) -- (363.46,36.75) -- (447.51,35.93) -- cycle ;
\draw  [draw opacity=0][fill={rgb, 255:red, 255; green, 255; blue, 255 }  ,fill opacity=1 ] (368.79,228.62) -- (446.75,228.62) -- (446.75,237.06) -- (368.79,237.06) -- cycle ;
\draw    (253.98,104.65) -- (296.86,104.65) ;
\draw [shift={(298.86,104.65)}, rotate = 180] [color={rgb, 255:red, 0; green, 0; blue, 0 }  ][line width=0.75]    (10.93,-3.29) .. controls (6.95,-1.4) and (3.31,-0.3) .. (0,0) .. controls (3.31,0.3) and (6.95,1.4) .. (10.93,3.29)   ;
\draw    (445.01,188.54) -- (445.01,230.08) ;
\draw [shift={(445.01,232.08)}, rotate = 270] [color={rgb, 255:red, 0; green, 0; blue, 0 }  ][line width=0.75]    (10.93,-3.29) .. controls (6.95,-1.4) and (3.31,-0.3) .. (0,0) .. controls (3.31,0.3) and (6.95,1.4) .. (10.93,3.29)   ;
\draw [shift={(445.01,186.54)}, rotate = 90] [color={rgb, 255:red, 0; green, 0; blue, 0 }  ][line width=0.75]    (10.93,-3.29) .. controls (6.95,-1.4) and (3.31,-0.3) .. (0,0) .. controls (3.31,0.3) and (6.95,1.4) .. (10.93,3.29)   ;
\draw    (140.85,410.85) -- (140.85,497.22) -- (186.25,497.22) ;
\draw [shift={(188.25,497.22)}, rotate = 180] [color={rgb, 255:red, 0; green, 0; blue, 0 }  ][line width=0.75]    (10.93,-3.29) .. controls (6.95,-1.4) and (3.31,-0.3) .. (0,0) .. controls (3.31,0.3) and (6.95,1.4) .. (10.93,3.29)   ;
\draw    (397.61,497.22) -- (445.01,497.22) -- (445.01,412.85) ;
\draw [shift={(445.01,410.85)}, rotate = 90] [color={rgb, 255:red, 0; green, 0; blue, 0 }  ][line width=0.75]    (10.93,-3.29) .. controls (6.95,-1.4) and (3.31,-0.3) .. (0,0) .. controls (3.31,0.3) and (6.95,1.4) .. (10.93,3.29)   ;
\draw  [draw opacity=0][fill={rgb, 255:red, 255; green, 255; blue, 255 }  ,fill opacity=1 ] (304.91,271.87) -- (304.24,349.34) -- (295.75,349.27) -- (296.41,271.8) -- cycle ;
\draw    (140.85,194.94) -- (140.85,225.68) ;
\draw [shift={(140.85,227.68)}, rotate = 270] [color={rgb, 255:red, 0; green, 0; blue, 0 }  ][line width=0.75]    (10.93,-3.29) .. controls (6.95,-1.4) and (3.31,-0.3) .. (0,0) .. controls (3.31,0.3) and (6.95,1.4) .. (10.93,3.29)   ;

\draw (85.11,12.43) node [anchor=north west][inner sep=0.75pt]   [align=left] {Original signal };
\draw (77.01,390.86) node [anchor=north west][inner sep=0.75pt]   [align=left] {Noisy observations };
\draw (208.5,579.29) node [anchor=north west][inner sep=0.75pt]   [align=left] {Histogram estimator };
\draw (306.52,390.86) node [anchor=north west][inner sep=0.75pt]   [align=left] {Estimated persistence diagram };
\draw (325.3,12.43) node [anchor=north west][inner sep=0.75pt]   [align=left] {True persistence diagram};
\draw (362.98,190.77) node [anchor=north west][inner sep=0.75pt]   [align=left] {Comparison \\ \ \ \ \ in $\displaystyle d_{b}$};

\end{tikzpicture}

\caption{Numerical illustration of the estimation procedures for $f(x,y)=\cos(2\pi x)\sin(2\pi x)+\mathbb{1}_{(x-1/2)^{2}+(y-1/2)^{2}<1/8}$, $\sigma=0.1$, $n=2500$, $h=1/4(\frac{\log(n)}{n})^{1/4}$.}
    \label{fig:Illustration Procedure}
\end{figure}
A natural question is how to calibrate the window size $h$ for signals. From the proof of Lemma \ref{lmm1} (Appendix \ref{Proof interleavish}), a good choice is taking $h$ such that
\begin{equation}
\label{calibration}
h^{\alpha}>\sqrt{\frac{\log\left(1/h^{d}\right)}{nh^{d}}}
\end{equation}
thus, we can choose:
$$h\asymp\left(\frac{\log(n)}{n}\right)^{\frac{1}{d+2\alpha}}.$$
\subsection*{Contribution}
In this framework, we study the convergence properties of the estimator $\widehat{\operatorname{dgm}(f)}$ with respect to the bottleneck distance. We provide a rigorous analysis showing that it achieves the following convergence rates over the classes $S_d(M, L, \alpha, R)$.
\begin{customthm}{2}
$$\underset{f\in S_{d}(M,L,\alpha,R)}{\sup}\quad\mathbb{E}\left(d_{b}\left(\widehat{\operatorname{dgm}(f)},\operatorname{dgm}(f)\right)\right)\lesssim \left(\frac{\log(n)}{n}\right)^{\frac{\alpha}{d+2\alpha}}.$$
\end{customthm}
Furthermore, we establish that these rates are optimal, in the minimax sense, over the classes $S_{d}(M,L,\alpha,R)$. 
\begin{customthm}{3}
    $$\underset{\widehat{\operatorname{dgm}(f)}}{\inf}\quad\underset{f\in S_{d}(M,L,\alpha,R)}{\sup}\quad\mathbb{E}\left(d_{b}\left(\widehat{\operatorname{dgm}(f)},\operatorname{dgm}(f)\right)\right)\gtrsim \left(\frac{\log(n)}{n}\right)^{\frac{\alpha}{d+2\alpha}}.$$
\end{customthm}
Interestingly, these rates coincide with the rates for Hölder spaces, obtained in Corollary 4.4 of \cite{BCL2009}. This means that, up to a multiplicative constant, there is no additional cost to consider signals in $S_{d}(M,L,\alpha,R)$. It demonstrates the gain of breaking free from the usual analysis approach in TDA and the robustness to discontinuities of persistence diagram estimation. Furthermore, as such irregularities are challenging to handle for signal estimation, these results promote the use of persistence diagrams while processing noisy (irregular) signals.
The paper is organized as follows. Section \ref{Background} provides some background on persistent homology. Section \ref{Upperbounds section} is dedicated to the proof of Theorem \ref{estimation borne sup}. Section \ref{section LB} is dedicated to the proof of Theorem \ref{lowerbound}. Section \ref{sec: illustration} provides some numerical illustrations of our results in the context of image analysis. Appendix \ref{q-tame appendix},\ref{technical appendix}, and \ref{appendix: claim} contain proofs of technical lemmas, propositions, and claims invoked in this paper.
\section{Topological background}
\label{Background}
This section aims to provide the essential topological background needed to follow this paper.
\subsection{Homology and invariance under deformation retraction}
We provide a brief introduction to (singular) homology, emphasizing key properties that are central to our work. For further details, we refer the reader to \cite{Hatcher}, which inspired this section. We begin by introducing the notion of singular simplex. 
\begin{defi}Let a topological space $X$ and $\Delta^{s}$ the standard $s-$simplex (i.e., the $s$-dimensional simplex whose vertices are the $s+1$ standard unit vectors in $\mathbb{R}^{s+1}$, denoted $v_1,...,v_s$). A singular $s-$simplex of $X$ is a continuous map $\sigma: \Delta^{s}\rightarrow X$. 
\end{defi}
Singular simplices are continuous maps from a standard simplex to the space $X$. They are not required to be homeomorphisms of any kind and can be widely singular, which allows for very general construction of homology groups. From these singular simplices, we can define the space of singular chains and the boundary operator.
\begin{defi}
	The group of $s-$chains on $X$, $C_{s}(X)$, for a field $\mathbb{K}$, is the abelian group of finite formal sums $\sum_{i}c_{i}\sigma_{i}$ with $\sigma_{i}$ a singular $s$-simplex of $X$ and $c_{i}\in\mathbb{K}$.
\end{defi}
\begin{defi}
	The boundary $\partial_{s} \sigma$ of a $s$-singular simplex $\sigma$ is defined by:
	$$\partial_{s} \sigma=\sum_{i=0}^s(-1)^i\sigma|_{\left[v_0, \cdots, \widehat{v}_i, \cdots, v_s\right]}.
	$$
    where $\left[v_0, \cdots, \hat{v}_i, \cdots, v_s\right]=\left[v_0, \cdots,v_{i-1},v_{i+1}, \cdots, v_s\right]$. The map $\sigma \mid_{\left[v_0, \cdots, \hat{v}_i, \cdots, v_s\right]}$ can then be thought of as a map from $\Delta^{s-1}$ to $X$.
	The boundary operator $\partial_s: C_{s}(X) \rightarrow C_{s-1}(X)$ is then defined by:
	\begin{align*}
		\partial_{s}:C_{s}(\mathcal{K}) & \rightarrow C_{s-1}(\mathcal{K}) \\
		c=\sum_{i}c_{i}\sigma_{i} & \rightarrow \partial_{s} c=\sum_{i} c_{i}\partial_{s} \sigma_{i}.
	\end{align*}
\end{defi}
A crucial property of the boundary operator is that applying it twice always gives zero: $\partial_{s} \circ \partial_{s+1}=0$. We can then consider two subgroups of $C_{s}(X)$: 
\begin{itemize}
    \item the $s-$cycles: $Z_{s}=\operatorname{Ker}(\partial_{s})$, i.e., the $s$-chains whose boundary is zero.
    \item the $s-$boundaries: $B_{s}=\operatorname{Im}(\partial_{s+1})$ i.e., the $s-$chains that are themselves boundaries of higher-dimensional chains.
    \end{itemize}
    As $\partial_{s} \circ \partial_{s+1}=0$, $B_{s}$ is a normal subgroup of $Z_{s}$. We can then define the $s-$th homology groups as $s-$cycles quotiented by $s-$boundaries.
\begin{defi}
	Let $X$ be a topological space and $s\in\mathbb{N}$. We can define the $s-$th (singular) homology group of $\mathcal{K}$ (on the field $\mathbb{K}$) by:
	$$H_{s}(X)= Z_{s}/B_{s}$$
\end{defi}
Intuitively, these groups identify the ``holes'' in the space $X$ at a given dimension: $H_{0}(X)$ identifies its connected components, $H_{1}(X)$ its loops, $H_{2}(X)$ its enclosed cavities, and so on.\\\\
A key property of homology, which we extensively exploit in this work, is that it is invariant under deformation retraction, meaning that if a space can be continuously "shrunk" to a subspace without tearing or gluing, its homology groups remain the same. The construction of such deformations is a classical technique that lies at the heart of many foundational results in algebraic topology.
\begin{defi}
\label{defi: def retract}
A subspace $A$ of $X$ is called a deformation retract of $X$ if there is a continuous $F: X \times [0,1] \rightarrow X$ such that for all $x \in X$ and $a \in A$,
\begin{itemize}
    \item $F(x, 0)=x$ 
    \item $F(x, 1) \in A$
    \item $F(a, 1)=a$.
\end{itemize}
The function $F$ is then called a (deformation) retraction from $X$ to $A$.
\end{defi}
 A deformation retraction \( F \) from \( X \) to \( A \) ensures that \( X \) and \( A \) are ``homotopy equivalent'' \citep[see][pages~2--3]{Hatcher}, and thus are topologically similar. In particular, \( F \) induces isomorphisms between the homology groups of \( X \) and \( A \). More precisely, for all $t\in[0,1]$, $F(.,t)$ induces a group morphism $F^{\#}(.,t):C_{s}(X)\rightarrow C_{S}(X)$ between $s-$chains of $X$ defined by composing each singular $s$-simplex $\sigma: \Delta^s \rightarrow X$ with $F(.,t)$ to get a singular $s$-simplex $F^{\#}(\sigma,t)=F(.,t)\circ\sigma: \Delta^s \rightarrow X$, then extending $F^{\#}(.,t)$ linearly via $F^{\#}\left(\sum_i n_i \sigma_i,t\right)=\sum_i n_i F^{\#}\left(\sigma_i,t\right)=$ $\sum_i n_i F(.,t)\circ \sigma_i$. The group morphism $F^{*}(.,t):H_{s}(X)\rightarrow H_{s}(X)$ defined by $[C]\mapsto [F^{\#}(C,t)]$ can be shown to be an isomorphism for all $t\in[0,1]$ \citep[see][pages 110-113]{Hatcher}. In particular, $F^{*}(.,1):H_{s}(X)\rightarrow H_{s}(A)$ is an isomorphism. 
\subsection{Filtrations, persistence modules and diagrams}
The idea behind persistent homology is to encode the evolution of the topology (in the homology sense) of a nested family of topological spaces, called a filtration. As we move along indices, topological features (connected components, cycles, cavities, ...) can appear or disappear (existing connected components merge, cycles or cavities are filled, ...). Two keys to formalize this idea, which we use throughout this paper, are the notions of filtration and of persistence module.
\begin{defi}
Let $\Lambda \subset \mathbb{R}$ be a set of indices. A filtration indexed by $\Lambda$ is a family $\left(\mathcal{K}_\lambda\right)_{\lambda \in \Lambda}$ of topological spaces satisfying,  for all $ \lambda, \lambda^{\prime} \in \Lambda, \lambda \leqslant \lambda^{\prime}$
$$
\mathcal{K}_\lambda \subset \mathcal{K}_{\lambda^{\prime}}
.$$
\end{defi}
The typical filtration that we will consider in this paper is, for a function $f:\mathbb{R}^{d}\rightarrow\mathbb{R}$, the family of sublevel sets $\left(\mathcal{F}_{\lambda}\right)_{\lambda\in\mathbb{R}}=\left(f^{-1}(]-\infty,\lambda])\right)_{\lambda\in\mathbb{R}}$.
\begin{defi}
 Let $\Lambda\subset \mathbb{R}$ be a set of indices. A persistence module over $\Lambda$ is a family $\mathbb{V}=\left(\mathbb{V}_\lambda\right)_{\lambda \in \Lambda}$ of vector spaces equipped with linear maps $v_\lambda^{\lambda^{\prime}}: \mathbb{V}_\lambda \rightarrow \mathbb{V}_{\lambda^{\prime}}$ such that, for all $ \lambda \leqslant \lambda^{\prime} \leqslant \lambda^{\prime \prime} \in \Lambda$,
$$ v_\lambda^\lambda=i d$$
and 
$$v_{\lambda^{\prime}}^{\lambda^{\prime \prime}} \circ v_\lambda^{\lambda^{\prime}}=v_\lambda^{\lambda^{\prime \prime}}.$$
\end{defi}
The typical persistence module that we will consider in this paper is, for a function $f:\mathbb{R}^{d}\rightarrow\mathbb{R}$ and $s\in\mathbb{N}$, the family of homology groups $\mathbb{V}_{f,s}=\left(H_{s}\left(\mathcal{F}_{\lambda}\right)\right)_{\lambda\in\mathbb{R}}$ equipped with $v_{\lambda}^{\lambda'}$ the linear maps induced by the inclusion $\mathcal{F}_{\lambda}\subset\mathcal{F}_{\lambda^{\prime}}$. To be more precise, in this paper, $H_{s}(.)$ is the singular homology functor in degree $s$ with coefficients in a field (typically $\mathbb{Z}/2\mathbb{Z}$). Hence, $H_{s}\left(\mathcal{F}_{\lambda}\right)$ is a vector space.\\\\
Under the \( q \)-tameness assumption (see Definition \ref{def: q-tame} in Appendix \ref{q-tame appendix}), a persistence module \( \mathbb{V} \) can be fully described by a multi-set in \( \mathbb{R}^2 \), known as its persistence diagram:  
\[
\operatorname{dgm}(\mathbb{V}) = \{(b_j, d_j) \mid j \in J\} \subset \overline{\mathbb{R}}^2.
\]  
In this representation, each \( b_j \) represents the birth time of a topological feature (e.g., the emergence of a connected component or the formation of a cycle), while \( d_j \) represents its death time (e.g., the merging of two components or the filling of a cycle). The difference \( d_j - b_j \) corresponds to the feature’s lifetime. For a detailed construction of persistence diagrams, we refer the reader to \cite{chazal2013}. In cases where we consider the collection of modules \( (\mathbb{V}_{f,s})_{s\in\mathbb{N}} \) associated with the sublevel sets filtration of \( f \), we denote by \( \operatorname{dgm}(f) \) the collection of associated persistence diagrams.

\begin{figure}[H]
\centering
\begin{subfigure}{.5\textwidth}
  \centering
  \includegraphics[scale=0.55]{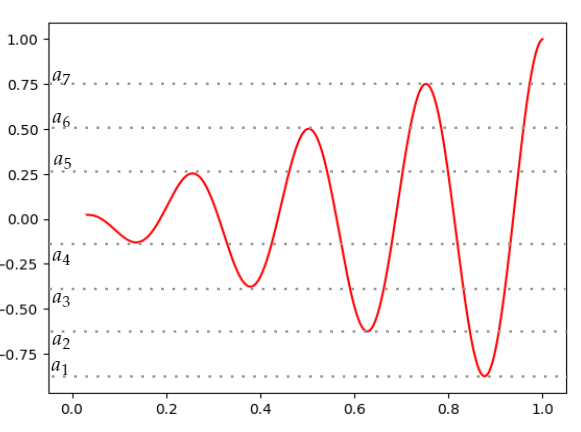}
  \caption{graph of $f$}
  \label{fig:sub1}
\end{subfigure}%
\begin{subfigure}{.5\textwidth}
  \centering
  \includegraphics[scale=0.55]{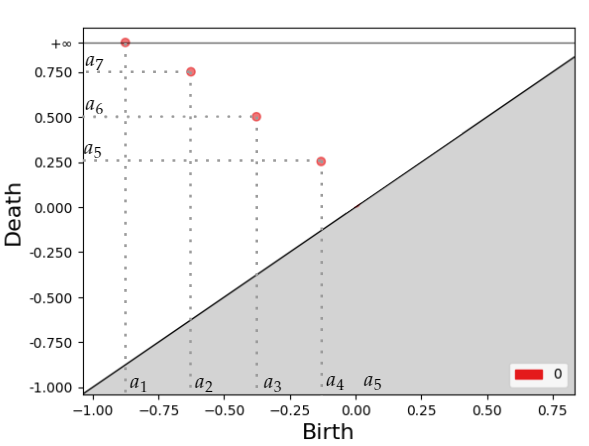}
  \caption{$H_{0}$-persistence diagram of $f$}
  \label{fig:sub2}
\end{subfigure}
\caption{Graph of $f(x)=x\cos(8\pi x)$ over $[0,1]$ and the (0-th) persistence diagram associated to its sublevel sets filtration. $a_{1}$, ..., $a_{4}$ correspond to local minima of $f$ and thus birth times in $\operatorname{dgm}(f)$. $a_{5}$, ..., $a_{7}$ correspond to local maxima of $f$ and thus death times in $\operatorname{dgm}(f)$.}
\end{figure}
\subsection{Bottleneck distance, interleaved modules and algebraic stability}
To compare persistence diagrams, we need a distance. A popular such distance, due to its stability properties, is the bottleneck distance. This distance is defined as the infimum over all matchings between points in diagrams of the maximal sup norm distance between two matched points. To allow for matchings between diagrams that do not contain the same number of points, the diagonal is added to the diagrams. This distance will be used in this work to evaluate the quality of our estimation procedures.
\begin{defi}
The bottleneck distance between two persistence diagrams $D_{1}$ and $D_{2}$ is,
$$d_{b}\left(D_{1},D_{2}\right)=\underset{\gamma\in\Gamma}{\inf}\underset{p\in D_{1}\cup  \Delta }{\sup}||p-\gamma(p)||_{\infty}$$
    with $\Gamma$ the set of all bijections between $D_{1}\cup \Delta $ and $D_{2}\cup \Delta$, where $\Delta=\{(x,x), x\in\mathbb{R}\}$.
\end{defi}
Another notion that will be the key to proving our upper bounds is the notion of interleaving between persistence modules. We use in particular the fact that if two modules are $\varepsilon-$interleaved, then the bottleneck distance between their diagrams is upper bounded by $\varepsilon$ in bottleneck distance.
\begin{defi}
\label{def: interleav}
Two persistence modules $\mathbb{V}=\left(\mathbb{V}_{\lambda}\right)_{\lambda\in I\subset \mathbb{R}}$ and $\mathbb{W}=\left(\mathbb{W}_{\lambda}\right)_{\lambda\in I\subset \mathbb{R}}$ are said to be $\varepsilon$-interleaved if there exist two families of linear maps, which we will refer to as (persistence module) morphisms: $$\phi=\left(\phi_{\lambda}:\mathbb{V}_{\lambda}\rightarrow\mathbb{W}_{\lambda+\varepsilon}\right)_{\lambda\in I\subset \mathbb{R}}\text{ and } \psi=\left(\psi_{\lambda}:\mathbb{W}_{\lambda}\rightarrow\mathbb{V}_{\lambda+\varepsilon}\right)_{\lambda\in I\subset \mathbb{R}}$$
such that, for all $\lambda<\lambda'$, the following diagrams commutes,
\begin{center}
\begin{tikzcd}
	{\mathbb{V}_{\lambda }} && {\mathbb{V}_{\lambda '}} & {\mathbb{W}_{\lambda }} && {\mathbb{W}_{\lambda '}} \\
	{\mathbb{W}_{\lambda +\varepsilon }} && {\mathbb{W}_{\lambda '+\varepsilon }} & {\mathbb{V}_{\lambda +\varepsilon }} && {\mathbb{V}_{\lambda '+\varepsilon }} \\
	{\mathbb{V}_{\lambda }} && {\mathbb{V}_{\lambda+2\varepsilon }} & {\mathbb{W}_{\lambda}} && {\mathbb{W}_{\lambda+2\varepsilon }} \\
	& {\mathbb{W}_{\lambda +\varepsilon }} &&& {\mathbb{V}_{\lambda +\varepsilon }}
	\arrow["{\phi _{\lambda }}"', from=1-1, to=2-1]
	\arrow["{\phi _{\lambda '}}", from=1-3, to=2-3]
	\arrow["{w_{\lambda +\varepsilon }^{\lambda ' +\varepsilon }}"', from=2-1, to=2-3]
	\arrow["{w_{\lambda }^{\lambda '}}", from=1-4, to=1-6]
	\arrow["{\psi_{\lambda }}"', from=1-4, to=2-4]
	\arrow["{v_{\lambda +\varepsilon }^{\lambda ' +\varepsilon }}"', from=2-4, to=2-6]
	\arrow["{\psi_{\lambda '}}", from=1-6, to=2-6]
	\arrow["{v_{\lambda }^{\lambda +2\varepsilon}}", from=3-1, to=3-3]
	\arrow["{\phi _{\lambda }}"', from=3-1, to=4-2]
	\arrow["{\psi_{\lambda+\varepsilon}}"', from=4-2, to=3-3]
	\arrow["{\psi_{\lambda }}"', from=3-4, to=4-5]
	\arrow["{v_{\lambda }^{\lambda '}}", from=1-1, to=1-3]
	\arrow["{w_{\lambda }^{\lambda +2\varepsilon}}", from=3-4, to=3-6]
	\arrow["{\phi _{\lambda+\varepsilon}}"', from=4-5, to=3-6]
\end{tikzcd}
\end{center}
\end{defi}
\begin{figure}
    \centering
    \begin{tikzpicture}[x=0.75pt,y=0.75pt,yscale=-1,xscale=1]

\draw    (191.4,252.6) -- (387.4,56.6) ;
\draw  [draw opacity=0][fill={rgb, 255:red, 74; green, 144; blue, 226 }  ,fill opacity=1 ] (206,210) .. controls (206,207.79) and (207.79,206) .. (210,206) .. controls (212.21,206) and (214,207.79) .. (214,210) .. controls (214,212.21) and (212.21,214) .. (210,214) .. controls (207.79,214) and (206,212.21) .. (206,210) -- cycle ;
\draw  [draw opacity=0][fill={rgb, 255:red, 74; green, 144; blue, 226 }  ,fill opacity=1 ] (216,90) .. controls (216,87.79) and (217.79,86) .. (220,86) .. controls (222.21,86) and (224,87.79) .. (224,90) .. controls (224,92.21) and (222.21,94) .. (220,94) .. controls (217.79,94) and (216,92.21) .. (216,90) -- cycle ;
\draw  [draw opacity=0][fill={rgb, 255:red, 74; green, 144; blue, 226 }  ,fill opacity=1 ] (294,60) .. controls (294,57.79) and (295.79,56) .. (298,56) .. controls (300.21,56) and (302,57.79) .. (302,60) .. controls (302,62.21) and (300.21,64) .. (298,64) .. controls (295.79,64) and (294,62.21) .. (294,60) -- cycle ;
\draw  [draw opacity=0][fill={rgb, 255:red, 74; green, 144; blue, 226 }  ,fill opacity=1 ] (186,38.67) .. controls (186,36.46) and (187.79,34.67) .. (190,34.67) .. controls (192.21,34.67) and (194,36.46) .. (194,38.67) .. controls (194,40.88) and (192.21,42.67) .. (190,42.67) .. controls (187.79,42.67) and (186,40.88) .. (186,38.67) -- cycle ;
\draw  [draw opacity=0][fill={rgb, 255:red, 74; green, 144; blue, 226 }  ,fill opacity=1 ] (266,100.67) .. controls (266,98.46) and (267.79,96.67) .. (270,96.67) .. controls (272.21,96.67) and (274,98.46) .. (274,100.67) .. controls (274,102.88) and (272.21,104.67) .. (270,104.67) .. controls (267.79,104.67) and (266,102.88) .. (266,100.67) -- cycle ;
\draw  [draw opacity=0][fill={rgb, 255:red, 208; green, 2; blue, 27 }  ,fill opacity=1 ] (216,39.33) .. controls (216,37.12) and (217.79,35.33) .. (220,35.33) .. controls (222.21,35.33) and (224,37.12) .. (224,39.33) .. controls (224,41.54) and (222.21,43.33) .. (220,43.33) .. controls (217.79,43.33) and (216,41.54) .. (216,39.33) -- cycle ;
\draw  [draw opacity=0][fill={rgb, 255:red, 208; green, 2; blue, 27 }  ,fill opacity=1 ] (223.33,104.67) .. controls (223.33,102.46) and (225.12,100.67) .. (227.33,100.67) .. controls (229.54,100.67) and (231.33,102.46) .. (231.33,104.67) .. controls (231.33,106.88) and (229.54,108.67) .. (227.33,108.67) .. controls (225.12,108.67) and (223.33,106.88) .. (223.33,104.67) -- cycle ;
\draw  [draw opacity=0][fill={rgb, 255:red, 208; green, 2; blue, 27 }  ,fill opacity=1 ] (268,88) .. controls (268,85.79) and (269.79,84) .. (272,84) .. controls (274.21,84) and (276,85.79) .. (276,88) .. controls (276,90.21) and (274.21,92) .. (272,92) .. controls (269.79,92) and (268,90.21) .. (268,88) -- cycle ;
\draw  [draw opacity=0][fill={rgb, 255:red, 208; green, 2; blue, 27 }  ,fill opacity=1 ] (280,60) .. controls (280,57.79) and (281.79,56) .. (284,56) .. controls (286.21,56) and (288,57.79) .. (288,60) .. controls (288,62.21) and (286.21,64) .. (284,64) .. controls (281.79,64) and (280,62.21) .. (280,60) -- cycle ;
\draw  [draw opacity=0][fill={rgb, 255:red, 74; green, 144; blue, 226 }  ,fill opacity=1 ] (230,149.33) .. controls (230,147.12) and (231.79,145.33) .. (234,145.33) .. controls (236.21,145.33) and (238,147.12) .. (238,149.33) .. controls (238,151.54) and (236.21,153.33) .. (234,153.33) .. controls (231.79,153.33) and (230,151.54) .. (230,149.33) -- cycle ;
\draw    (183.87,148.93) -- (228,149.32) ;
\draw [shift={(230,149.33)}, rotate = 180.5] [color={rgb, 255:red, 0; green, 0; blue, 0 }  ][line width=0.75]    (10.93,-3.29) .. controls (6.95,-1.4) and (3.31,-0.3) .. (0,0) .. controls (3.31,0.3) and (6.95,1.4) .. (10.93,3.29)   ;
\draw  [draw opacity=0][fill={rgb, 255:red, 208; green, 2; blue, 27 }  ,fill opacity=1 ] (241.33,139.33) .. controls (241.33,137.12) and (243.12,135.33) .. (245.33,135.33) .. controls (247.54,135.33) and (249.33,137.12) .. (249.33,139.33) .. controls (249.33,141.54) and (247.54,143.33) .. (245.33,143.33) .. controls (243.12,143.33) and (241.33,141.54) .. (241.33,139.33) -- cycle ;
\draw  [draw opacity=0][fill={rgb, 255:red, 208; green, 2; blue, 27 }  ,fill opacity=1 ] (242,159.33) .. controls (242,157.12) and (243.79,155.33) .. (246,155.33) .. controls (248.21,155.33) and (250,157.12) .. (250,159.33) .. controls (250,161.54) and (248.21,163.33) .. (246,163.33) .. controls (243.79,163.33) and (242,161.54) .. (242,159.33) -- cycle ;
\draw  [dash pattern={on 0.84pt off 2.51pt}]  (190,38.67) -- (220,39.33) ;
\draw  [dash pattern={on 0.84pt off 2.51pt}]  (283.67,60) -- (298,60) ;
\draw  [dash pattern={on 0.84pt off 2.51pt}]  (272,88) -- (270,100.67) ;
\draw  [dash pattern={on 0.84pt off 2.51pt}]  (220,90) -- (227.33,104.67) ;
\draw  [dash pattern={on 0.84pt off 2.51pt}]  (234,149.33) -- (246,159.33) ;
\draw  [dash pattern={on 0.84pt off 2.51pt}]  (234,149.33) -- (245.33,139.33) ;
\draw  [dash pattern={on 0.84pt off 2.51pt}]  (210,210) -- (221.87,221.6) ;
\draw  [draw opacity=0][fill={rgb, 255:red, 208; green, 2; blue, 27 }  ,fill opacity=1 ] (341.33,198.67) .. controls (341.33,196.46) and (343.12,194.67) .. (345.33,194.67) .. controls (347.54,194.67) and (349.33,196.46) .. (349.33,198.67) .. controls (349.33,200.88) and (347.54,202.67) .. (345.33,202.67) .. controls (343.12,202.67) and (341.33,200.88) .. (341.33,198.67) -- cycle ;
\draw  [draw opacity=0][fill={rgb, 255:red, 74; green, 144; blue, 226 }  ,fill opacity=1 ] (341,224) .. controls (341,221.79) and (342.79,220) .. (345,220) .. controls (347.21,220) and (349,221.79) .. (349,224) .. controls (349,226.21) and (347.21,228) .. (345,228) .. controls (342.79,228) and (341,226.21) .. (341,224) -- cycle ;

\draw (106.67,140) node [anchor=north west][inner sep=0.75pt]   [align=left] {{\small multiplicity 2}};
\draw (359,193.4) node [anchor=north west][inner sep=0.75pt]    {$D_{1}$};
\draw (358,218.4) node [anchor=north west][inner sep=0.75pt]    {$D_{2}$};

\end{tikzpicture}

    \caption{Optimal matching for the bottleneck distance between $D_{1}$ and $D_{2}$.}
\end{figure}
\begin{customthm}{(algebraic stability \citep{Chazal2009})}
Let $\mathbb{V}$ and $\mathbb{W}$ be two $q-$tame persistence modules (see definition \ref{def: q-tame} in Appendix \ref{q-tame appendix}). If $\mathbb{V}$ and $\mathbb{W}$ are $\varepsilon-$interleaved then,
$$d_{b}\left(\operatorname{dgm}(\mathbb{V}),\operatorname{dgm}(\mathbb{W})\right)\leq \varepsilon$$
\end{customthm}
Note that in the context of sublevel set persistence, a direct consequence of this theorem is the stability of persistence diagrams with respect to the sup-norm distance between functions. More precisely, if two $q$-tame real-valued functions $f$ and $g$ satisfy $\|f - g\|_\infty \leq \varepsilon$, then their sublevel sets filtrations satisfy
\[
\mathcal{F}_{\lambda - \varepsilon} \subset \mathcal{G}_{\lambda} \subset \mathcal{F}_{\lambda + \varepsilon}, \quad \text{for all } \lambda \in \mathbb{R}.
\]
The linear maps induced in homology by these inclusions define an $\varepsilon$-interleaving. By the algebraic stability theorem, this implies:
\[
d_{b}(\operatorname{dgm}(f), \operatorname{dgm}(g)) \leq \varepsilon.
\]
\section{Upper bounds}
This section is devoted to the proof of Theorem \ref{estimation borne sup}. First, we introduce some notations. For a point $x\in [0,1]^{d}$ and $b>0$, we denote the closed Euclidean ball centered at $x$ of radius $b$:
$$B_2(x,b)=\{y\in[0,1]^{d}, ||x-y||_2\leq b\}$$
Similarly, for a set $A\subset \mathbb{R}^{d}$ and $b\geq 0$, we denote,
$$A^{b}=B_{2}(A,b)=\left\{x\in [0,1]^{d} \text{ s.t. } d_{2}\left(\{x\},A\right)\leq b\right\}$$ 
and 
$$A^{-b}=\left(\left(A^{c}\right)^{b}\right)^{c}.$$
We also define, for all $h>0$:
$$N_{h}=\frac{\max\limits_{H\in C_{h}}\left| \frac{\sigma}{nh^{d}}\sum\limits_{x_i\in H}\varepsilon_{i}\right|}{\sqrt{2\sigma^{2}\frac{\log\left(1/h^{d}\right)}{nh^{d}}}}.$$
\textbf{Proof outline.} The strategy is to construct an interleaving between the estimated and true persistence modules and then apply the algebraic stability theorem \citep{Chazal2009}. To do so, we need to construct two families of linear maps $\phi$ and $\psi$ (i.e., module morphisms):
$$\mathbb{V}_{f,s}\overset{\phi}{\longrightarrow}\widehat{\mathbb{V}}_{f,s}\overset{\psi}{\longrightarrow}\mathbb{V}_{f,s}$$
satisfying Definition \ref{def: interleav}. The outline of our proof is as follows:
\begin{itemize}
    \item We show in Lemma \ref{interleavish 3} that for some $\beta_1\geq 0$, depending on $N_h$, $\widehat{\mathcal{F}}_\lambda\subset \mathcal{F}^{\sqrt{d}h}_{\lambda+\beta_1h^{\alpha}}$, which induces a linear map from the homology group of $\widehat{\mathcal{F}}_\lambda$ to the homology group of $\mathcal{F}^{\sqrt{d}h}_{\lambda+\beta_1h^{\alpha}}$.
    \item We identify in Lemma \ref{lemmaFiltEquiv} two families of sets $(\mathcal{K}_{\lambda})_{\lambda\in\mathbb{R}}$ and $(\mathcal{G}_{\lambda})_{\lambda\in\mathbb{R}}$ such that for some $\beta_2,\beta_3\geq 0$, depending on $N_h$, $\mathcal{F}^{\sqrt{d}h}_{\lambda+\beta_1 h^{\alpha}}\subset \mathcal{K}_{\lambda+\beta_2 h^{\alpha}}$, $\mathcal{K}_{\lambda+\beta_2 h^{\alpha}}$ (deformation) retracts onto $\mathcal{G}_{\lambda+\beta_2 h^{\alpha}}$ and $\mathcal{G}_{\lambda+\beta_2 h^{\alpha}}\subset \mathcal{F}_{\lambda+\beta_3 h^{\alpha}}$ which induces a linear map from the homology group of $\mathcal{F}^{\sqrt{d}h}_{\lambda+\beta_1 h^{\alpha}}$ to the homology group of $\mathcal{F}_{\lambda+\beta_3 h^{\alpha}}$.
    \item The composition of the two linear maps obtained previously induces a linear map from the homology group of $\widehat{\mathcal{F}}_\lambda$ to the homology group of $\mathcal{F}_{\lambda+\beta_3 h^{\alpha}}$. This is (essentially) our $\phi$.
    \item We proceed similarly to construct $\psi$, identifying in Lemma \ref{lemmaFiltEquiv2} two families $(\mathcal{N}_{\lambda})_{\lambda\in\mathbb{R}}$ and $(\mathcal{M}_{\lambda})_{\lambda\in\mathbb{R}}$ such that for some $\beta_4,\beta_5\geq 0$, depending on $N_h$, $\mathcal{F}_{\lambda}\subset \mathcal{N}_{\lambda+\beta_4 h^{\alpha}}$, $\mathcal{N}_{\lambda+\beta_4 h^{\alpha}}$ (deformation) retracts onto $\mathcal{M}_{\lambda+\beta_4 h^{\alpha}}$ and $\mathcal{M}_{\lambda+\beta_4 h^{\alpha}}\subset \mathcal{F}^{-\sqrt{d}h}_{\lambda+\beta_5 h^{\alpha}}$. Moreover, in Lemma \ref{interleavish 3} we show that for some $\beta_6\geq 0$, depending on $N_h$, $\mathcal{F}^{-\sqrt{d}h}_{\lambda+\beta_5 h^{\alpha}}\subset \widehat{\mathcal{F}}_{\lambda+\beta_6h^{\alpha}}$.
    \item We prove in Lemma \ref{lemma-histo 2} that the supports of these retractions, when applied to $x\in \widehat{\mathcal{F}}_{\lambda}$, remain within $\widehat{\mathcal{F}}_{\lambda+\beta_7h^{\alpha}}$, for some $\beta_{7}\geq 0$. Then, we use this additional property to show that the $\phi$ and $\psi$ we construct satisfy Definition \ref{def: interleav} and define an interleaving that depends on $N_h$.
    \item We conclude using a standard concentration argument on $N_h$ (Lemma \ref{lemma noise 2}).
\end{itemize}
Following this outline, we formally state Lemmas \ref{interleavish 3}-\ref{lemma noise 2}, which form the fundamental building blocks of our proof.
\begin{lmm}
\label{interleavish 3}
Let $f:[0,1]^{d}\rightarrow \mathbb{R}$. For all $\lambda\in\mathbb{R}$ and $h>0$ satisfying (\ref{calibration}),
$$\mathcal{F}_{\lambda-\sqrt{2}\sigma N_{h}h^{\alpha}}^{-\sqrt{d}h}\subset \widehat{\mathcal{F}}_{\lambda} \subset \mathcal{F}_{\lambda+\sqrt{2}\sigma N_{h}h^{\alpha}}^{\sqrt{d}h}.$$
\end{lmm}
The proof of Lemma \ref{interleavish 3} is provided in Appendix \ref{Proof interleavish}. This double inclusion induces a morphism from $(H_{s}(\mathcal{F}_{\lambda}^{-\sqrt{d}h}))_{\lambda\in\mathbb{R}}$ to $\widehat{\mathbb{V}}_{f,s}$ and a morphism from $\widehat{\mathbb{V}}_{f,s}$ to the module $(H_{s}(\mathcal{F}_{\lambda}^{\sqrt{d}h}))_{\lambda\in\mathbb{R}}$.\\\\
Now, what we need to construct the desired interleaving is, for arbitrary $h>0$:
\begin{itemize}
    \item a morphism from $\mathbb{V}_{f,s}$ to $(H_{s}(\mathcal{F}_{\lambda}^{-h}))_{\lambda\in\mathbb{R}}$
    \item a morphism from $(H_{s}(\mathcal{F}_{\lambda}^{h}))_{\lambda\in\mathbb{R}}$ to $\mathbb{V}_{f,s}$
\end{itemize}
We begin with the construction of the morphism from $(H_s(\mathcal{F}_{\lambda}^{h}))_{\lambda\in\mathbb{R}}$ to $\mathbb{V}_{f,s}$. The main difficulty arises near the boundaries of the regular pieces. When we thicken $\mathcal{F}_{\lambda}$ by $h$, the resulting set $\mathcal{F}_{\lambda}^{h}$ may extend into regions where the signal is significantly larger than $\lambda$. This prevents us from establishing a straightforward inclusion of the form:  
\[
\mathcal{F}^{h}_{\lambda} \subset \mathcal{F}_{\lambda+\varepsilon}
\]  
for a reasonable $\varepsilon$. In particular, this is why standard techniques based on sup-norm stability cannot be directly applied in this setting.  

To formalize this issue, we define the set of problematic points as  
\[
S_{\lambda,h}=\left(\left(\bigcup_{i=1}^{l}\partial M_{i}\cap ]0,1[^{d}\right)^{h}\setminus \mathcal{F}_{\lambda+Lh^{\alpha}}\right)\cap \mathcal{F}_{\lambda}^{h}.
\]  
Intuitively, $S_{\lambda,h}$ consists of the regions where $\mathcal{F}_{\lambda}^{h}$ extends beyond $\mathcal{F}_{\lambda+Lh^{\alpha}}$, preventing a controlled inclusion. Our goal is to smoothly "push" these points into $\mathcal{F}_{\lambda+\varepsilon}$ for some $\varepsilon \asymp h^{\alpha}$.
To achieve this, we will construct a deformation retraction using Theorem 4.8 of \cite{Fed59}. Under Assumption \textbf{A3}, this theorem guarantees that for every \( x \in [0,1]^d \) at a (Euclidean) distance strictly smaller than \( R \) from \( ]0,1[^d \cap \bigcup_{i=1}^{l} \partial M_i \), there exists a unique closest point in \( ]0,1[^d \cap \bigcup_{i=1}^{l} \partial M_i \), denoted by \( \xi(x) \), i.e., there exists a well-defined projection map. Moreover, the function \( \xi \) is continuous.

Using Assumptions \textbf{A1} and \textbf{A2}, we can show (see the proof of Lemma \ref{lemmaFiltEquiv}) that for all \( x \in S_{\lambda,h} \), the projection \( \xi(x) \) belongs to \( \mathcal{F}_{\lambda+\varepsilon} \) for some \( \varepsilon \asymp h^{\alpha} \). The idea is then to map each point \( x \in S_{\lambda,h} \) to \( \xi(x) \), moving along the line segment \( [x, \xi(x)] \). However, turning this idea into a proper deformation retraction requires several technical adjustments.

The first issue is that the segments \( [x, \xi(x)] \) for \( x \in S_{\lambda,h} \) may not be  contained entirely in \( \mathcal{F}_{\lambda}^{h} \). To resolve this, we consider a slightly larger set, adding the segments \( [x, \xi(x)] \) for all \( x \in S_{\lambda,h} \):  
\[
\mathcal{K}_{\lambda,h} := \mathcal{F}_{\lambda}^{h} \cup \left( \bigcup\limits_{x\in S_{\lambda,h}} [x, \xi(x)] \right) \subset \mathcal{F}_{\lambda}^{2h}.
\]   
\begin{figure}[H]
    \centering
\begin{tikzpicture}[x=0.75pt,y=0.75pt,yscale=-1,xscale=1]

\draw  [fill={rgb, 255:red, 74; green, 144; blue, 226 }  ,fill opacity=1 ] (420.09,191.43) -- (522.6,113.34) -- (545.5,143.4) -- (442.99,221.49) -- cycle ;
\draw  [draw opacity=0][fill={rgb, 255:red, 155; green, 155; blue, 155 }  ,fill opacity=1 ] (138.5,66.6) .. controls (158.5,56.6) and (263.2,51.15) .. (243.2,71.15) .. controls (223.2,91.15) and (211,103) .. (231,133) .. controls (251,163) and (161,163) .. (141,133) .. controls (121,103) and (118.5,76.6) .. (138.5,66.6) -- cycle ;
\draw   (86,26) -- (327.5,26) -- (327.5,267.5) -- (86,267.5) -- cycle ;
\draw  [dash pattern={on 4.5pt off 4.5pt}] (138,146.75) .. controls (138,108.78) and (168.78,78) .. (206.75,78) .. controls (244.72,78) and (275.5,108.78) .. (275.5,146.75) .. controls (275.5,184.72) and (244.72,215.5) .. (206.75,215.5) .. controls (168.78,215.5) and (138,184.72) .. (138,146.75) -- cycle ;
\draw  [fill={rgb, 255:red, 208; green, 2; blue, 27 }  ,fill opacity=0.1 ] (179.2,84.15) .. controls (202,71.8) and (251.5,78.6) .. (231.5,98.6) .. controls (211.5,118.6) and (226.5,125.6) .. (246.5,155.6) .. controls (266.5,185.6) and (116.2,220.95) .. (96.2,190.95) .. controls (76.2,160.95) and (118.9,169.36) .. (121.3,131.16) .. controls (123.7,92.95) and (156.4,96.5) .. (179.2,84.15) -- cycle ;
\draw    (128.2,109.15) .. controls (123.2,96.8) and (123.2,81.65) .. (133.7,70.15) .. controls (144.2,58.65) and (248.7,52.65) .. (245.2,66.15) .. controls (241.7,79.65) and (232.7,76.65) .. (231.2,83.65) ;
\draw  [fill={rgb, 255:red, 155; green, 155; blue, 155 }  ,fill opacity=1 ] (357,267.5) -- (357,35.44) .. controls (421.37,17.35) and (594.92,144.89) .. (533.5,147.6) .. controls (469.89,150.41) and (411.52,225.85) .. (402.62,267.5) -- (357,267.5) -- cycle ;
\draw  [dash pattern={on 0.84pt off 2.51pt}] (206.5,73.6) .. controls (206.5,59.79) and (217.69,48.6) .. (231.5,48.6) .. controls (245.31,48.6) and (256.5,59.79) .. (256.5,73.6) .. controls (256.5,87.41) and (245.31,98.6) .. (231.5,98.6) .. controls (217.69,98.6) and (206.5,87.41) .. (206.5,73.6) -- cycle ;
\draw  [dash pattern={on 0.84pt off 2.51pt}]  (231.5,48.6) -- (357,26) ;
\draw  [dash pattern={on 0.84pt off 2.51pt}]  (231.5,98.6) -- (357,267.5) ;
\draw  [fill={rgb, 255:red, 208; green, 2; blue, 27 }  ,fill opacity=0.13 ] (357,267.5) -- (357,152.87) .. controls (383.45,147.26) and (467.13,230.33) .. (447.57,249.63) .. controls (441.4,255.71) and (435.42,261.52) .. (430.66,267.5) -- (357,267.5) -- cycle ;
\draw   (357,26) -- (598.5,26) -- (598.5,267.5) -- (357,267.5) -- cycle ;
\draw  [draw opacity=0][dash pattern={on 4.5pt off 4.5pt}] (360.66,150.4) .. controls (381.1,158.01) and (401.06,170.77) .. (418.31,188.39) .. controls (440.74,211.32) and (454.85,238.65) .. (460.11,265.49) -- (344.04,261.07) -- cycle ; \draw  [dash pattern={on 4.5pt off 4.5pt}] (360.66,150.4) .. controls (381.1,158.01) and (401.06,170.77) .. (418.31,188.39) .. controls (440.74,211.32) and (454.85,238.65) .. (460.11,265.49) ;  
\draw  [fill={rgb, 255:red, 74; green, 74; blue, 74 }  ,fill opacity=1 ] (542.55,143.4) .. controls (542.55,141.77) and (543.87,140.45) .. (545.5,140.45) .. controls (547.13,140.45) and (548.45,141.77) .. (548.45,143.4) .. controls (548.45,145.03) and (547.13,146.35) .. (545.5,146.35) .. controls (543.87,146.35) and (542.55,145.03) .. (542.55,143.4) -- cycle ;
\draw  [fill={rgb, 255:red, 74; green, 74; blue, 74 }  ,fill opacity=1 ] (440.04,221.49) .. controls (440.04,219.86) and (441.37,218.54) .. (442.99,218.54) .. controls (444.62,218.54) and (445.94,219.86) .. (445.94,221.49) .. controls (445.94,223.12) and (444.62,224.44) .. (442.99,224.44) .. controls (441.37,224.44) and (440.04,223.12) .. (440.04,221.49) -- cycle ;

\draw (212,188.4) node [anchor=north west][inner sep=0.75pt]    {$M_{1}$};
\draw (99,233.4) node [anchor=north west][inner sep=0.75pt]    {$M_{2}$};
\draw (181,110.4) node [anchor=north west][inner sep=0.75pt]    {$\mathcal{F}_{\lambda }^{h}$};
\draw (98,173.4) node [anchor=north west][inner sep=0.75pt]  [color={rgb, 255:red, 208; green, 2; blue, 27 }  ,opacity=1 ]  {$\mathcal{F}_{\lambda +Lh^{\alpha }}$};
\draw (138,66.4) node [anchor=north west][inner sep=0.75pt]  [color={rgb, 255:red, 0; green, 0; blue, 0 }  ,opacity=1 ]  {$S_{\lambda ,h}$};
\draw (507.33,185.4) node [anchor=north west][inner sep=0.75pt]  [color={rgb, 255:red, 74; green, 144; blue, 226 }  ,opacity=1 ]  {$\mathcal{K}_{\lambda ,h}$};
\draw (554,133.4) node [anchor=north west][inner sep=0.75pt]    {$x$};
\draw (453,217.4) node [anchor=north west][inner sep=0.75pt]    {$\xi ( x)$};

\end{tikzpicture}

\caption{Illustration of the sets $\mathcal{F}_{\lambda}^{h}$, $\mathcal{F}_{\lambda+Lh^{\alpha}}$, $S_{\lambda,h}$, and $\mathcal{K}_{\lambda,h}$. The region $\mathcal{F}_{\lambda+Lh^{\alpha}}$ is shown in red, while $\mathcal{F}_{\lambda}^{h}$ is depicted in grey. The part of $\mathcal{F}_{\lambda}^{h}$ not included in $\mathcal{F}_{\lambda+Lh^{\alpha}}$ corresponds to $S_{\lambda,h}$. The blue region represents the union of segments added to $\mathcal{F}_{\lambda}^{h}$ to form $\mathcal{K}_{\lambda,h}$.}
    \label{fig:enter-label}
\end{figure}
Additionally, to ensure the continuity of the deformation retraction, we must avoid moving excessively points of \( \bigcup_{x\in S_{\lambda,h}} [x,\xi(x)] \cap M_{i} \) that are close to \( \mathcal{F}_{\lambda+Lh^{\alpha}} \cap M_{i} \) for all \( i \in \{1, \dots, l\} \). To achieve this, we introduce the condition (\textbf{C1}): there exists some \( i \in \{1, \dots, l\} \) such that  
\begin{equation}
\label{cond 1}
x \in \bigcup\limits_{x\in S_{\lambda,h}} [x,\xi(x)] \cap M_{i}
\end{equation}  
and  
\begin{equation}
\label{cond 2}
d_{2} \left(\xi(x),M_{i} \cap \mathcal{F}_{\lambda+Lh^{\alpha}}\right) \geq 2h - ||x - \xi(x)||_{2}.
\end{equation}  

We then define the function \( F_{\lambda,h}: \mathcal{K}_{\lambda,h} \times [0,1] \to \mathcal{K}_{\lambda,h} \), which will serve as our deformation retraction, by:  
\[
F_{\lambda,h}(x) = 
\begin{cases} 
      (1 - t)x + t \left( \xi(x) + \left( 2h - d_{2} \left(\xi(x),M_{i} \cap \mathcal{F}_{\lambda+Lh^{\alpha}}\right) \right)_{+} \frac{x - \xi(x)}{||x - \xi(x)||_{2}} \right), & \text{ under (\textbf{C1})},  \\
      x, & \text{ otherwise.} 
\end{cases}
\]

where \((.)_+ = \max(., 0)\) denotes the positive part. For  \( x\in \bigcup_{z\in S_{\lambda,h}} [z,\xi(z)] \), the function \( F_{\lambda,h} \) pushes \( x \) along the segment  
\[
\left[x, \xi(x) + \left( 2h - d_{2} \left(\xi(x),M_{i} \cap \mathcal{F}_{\lambda+Lh^{\alpha}}\right) \right)_{+} \frac{x - \xi(x)}{||x - \xi(x)||_{2}}\right] \subset [x, \xi(x)].
\]
From a geometric point of view, \( F_{\lambda,h}(x) \) pushes in the direction of the normalized vector \( -\frac{x - \xi(x)}{\|x - \xi(x)\|_2} \), which corresponds to the opposite of the gradient of the distance function to the set \( \bigcup_{i=1}^{l} \partial M_i \cap (0,1)^d \) at the point \( x \). This perspective links naturally to existing results that use the (generalized) gradient of the distance function to construct deformation retracts \citep[see for instance][Theorem 12]{kim2020homotopy}. Moreover, as \( \bigcup_{i=1}^{l} \partial M_i \cap (0,1)^d \) is an \( C^{1,1} \) hypersurface, it is worth noting that the vector \( -\frac{x - \xi(x)}{\|x - \xi(x)\|_2} \) generates the normal cone of the hypersurface at \( \xi(x) \) \citep[see][Theorem 4.8.]{Fed59}.

Note that the term \( \left( 2h - d_{2}\left( \xi(x), M_{i} \cap \mathcal{F}_{\lambda+Lh^{\alpha}} \right) \right)_{+} \) and the condition (\textbf{C1}) will be instrumental in establishing the continuity of \( F_{\lambda,h} \), as they ensure that, for any point \( x \in \bigcup_{x\in S_{\lambda,h}} [x,\xi(x)] \cap M_{i} \), the closer \( x \) is to \( \mathcal{F}_{\lambda+Lh^{\alpha}} \cap M_{i} \), the less it will be displaced by the deformation \( F_{\lambda,h} \). 

Finally, we define  
\[
\mathcal{G}_{\lambda,h} = \operatorname{Im} \left( x \longmapsto F_{\lambda,h}(x,1) \right),
\]  
which will serve as our deformation retract. With this, we have all the necessary formalism to state Lemma \ref{lemmaFiltEquiv}.
\begin{figure}[H]
    \centering
\begin{tikzpicture}[x=0.75pt,y=0.75pt,yscale=-1,xscale=1]

\draw  [fill={rgb, 255:red, 155; green, 155; blue, 155 }  ,fill opacity=1 ] (177.82,132.9) .. controls (177.82,101.14) and (248.56,75.4) .. (335.82,75.4) .. controls (423.08,75.4) and (493.82,101.14) .. (493.82,132.9) .. controls (493.82,164.65) and (423.08,190.4) .. (335.82,190.4) .. controls (248.56,190.4) and (177.82,164.65) .. (177.82,132.9) -- cycle ;
\draw  [fill={rgb, 255:red, 208; green, 2; blue, 27 }  ,fill opacity=0.08 ] (96,17.6) -- (173.12,17.6) .. controls (160.91,40.4) and (163.61,57.28) .. (181.21,96.06) .. controls (199.89,137.23) and (126.17,141.38) .. (96,108.51) -- (96,17.6) -- cycle ;
\draw   (96,17.6) -- (569.5,17.6) -- (569.5,273.6) -- (96,273.6) -- cycle ;
\draw [color={rgb, 255:red, 208; green, 2; blue, 27 }  ,draw opacity=1 ]   (335.82,75.4) -- (335.82,132.9) ;
\draw [shift={(335.82,109.15)}, rotate = 270] [fill={rgb, 255:red, 208; green, 2; blue, 27 }  ,fill opacity=1 ][line width=0.08]  [draw opacity=0] (8.93,-4.29) -- (0,0) -- (8.93,4.29) -- cycle    ;
\draw [color={rgb, 255:red, 208; green, 2; blue, 27 }  ,draw opacity=1 ]   (361.5,76.4) -- (363.2,132.7) ;
\draw [shift={(362.5,109.55)}, rotate = 268.27] [fill={rgb, 255:red, 208; green, 2; blue, 27 }  ,fill opacity=1 ][line width=0.08]  [draw opacity=0] (8.93,-4.29) -- (0,0) -- (8.93,4.29) -- cycle    ;
\draw [color={rgb, 255:red, 208; green, 2; blue, 27 }  ,draw opacity=1 ]   (386.8,78.2) -- (387.2,134.7) ;
\draw [shift={(387.04,111.45)}, rotate = 269.59] [fill={rgb, 255:red, 208; green, 2; blue, 27 }  ,fill opacity=1 ][line width=0.08]  [draw opacity=0] (8.93,-4.29) -- (0,0) -- (8.93,4.29) -- cycle    ;
\draw [color={rgb, 255:red, 208; green, 2; blue, 27 }  ,draw opacity=1 ]   (411.2,82.05) -- (411.2,134.7) ;
\draw [shift={(411.2,113.37)}, rotate = 270] [fill={rgb, 255:red, 208; green, 2; blue, 27 }  ,fill opacity=1 ][line width=0.08]  [draw opacity=0] (8.93,-4.29) -- (0,0) -- (8.93,4.29) -- cycle    ;
\draw [color={rgb, 255:red, 208; green, 2; blue, 27 }  ,draw opacity=1 ]   (437.7,89.05) -- (437.2,133.7) ;
\draw [shift={(437.39,116.37)}, rotate = 270.64] [fill={rgb, 255:red, 208; green, 2; blue, 27 }  ,fill opacity=1 ][line width=0.08]  [draw opacity=0] (8.93,-4.29) -- (0,0) -- (8.93,4.29) -- cycle    ;
\draw [color={rgb, 255:red, 208; green, 2; blue, 27 }  ,draw opacity=1 ]   (461.2,97.55) -- (461.2,135.7) ;
\draw [shift={(461.2,121.62)}, rotate = 270] [fill={rgb, 255:red, 208; green, 2; blue, 27 }  ,fill opacity=1 ][line width=0.08]  [draw opacity=0] (8.93,-4.29) -- (0,0) -- (8.93,4.29) -- cycle    ;
\draw [color={rgb, 255:red, 208; green, 2; blue, 27 }  ,draw opacity=1 ]   (481.2,110.55) -- (481.2,134.7) ;
\draw [shift={(481.2,127.62)}, rotate = 270] [fill={rgb, 255:red, 208; green, 2; blue, 27 }  ,fill opacity=1 ][line width=0.08]  [draw opacity=0] (8.93,-4.29) -- (0,0) -- (8.93,4.29) -- cycle    ;
\draw [color={rgb, 255:red, 208; green, 2; blue, 27 }  ,draw opacity=1 ]   (310.4,76.6) -- (311.2,133.7) ;
\draw [shift={(310.87,110.15)}, rotate = 269.2] [fill={rgb, 255:red, 208; green, 2; blue, 27 }  ,fill opacity=1 ][line width=0.08]  [draw opacity=0] (8.93,-4.29) -- (0,0) -- (8.93,4.29) -- cycle    ;
\draw [color={rgb, 255:red, 208; green, 2; blue, 27 }  ,draw opacity=1 ]   (283.7,79.55) -- (287.2,133.7) ;
\draw [shift={(285.77,111.61)}, rotate = 266.3] [fill={rgb, 255:red, 208; green, 2; blue, 27 }  ,fill opacity=1 ][line width=0.08]  [draw opacity=0] (8.93,-4.29) -- (0,0) -- (8.93,4.29) -- cycle    ;
\draw [color={rgb, 255:red, 208; green, 2; blue, 27 }  ,draw opacity=1 ]   (258.2,83.55) -- (259.2,131.7) ;
\draw [shift={(258.8,112.62)}, rotate = 268.81] [fill={rgb, 255:red, 208; green, 2; blue, 27 }  ,fill opacity=1 ][line width=0.08]  [draw opacity=0] (8.93,-4.29) -- (0,0) -- (8.93,4.29) -- cycle    ;
\draw [color={rgb, 255:red, 208; green, 2; blue, 27 }  ,draw opacity=1 ]   (236.2,89.05) -- (237.2,127.7) ;
\draw [shift={(236.83,113.37)}, rotate = 268.52] [fill={rgb, 255:red, 208; green, 2; blue, 27 }  ,fill opacity=1 ][line width=0.08]  [draw opacity=0] (8.93,-4.29) -- (0,0) -- (8.93,4.29) -- cycle    ;
\draw [color={rgb, 255:red, 208; green, 2; blue, 27 }  ,draw opacity=1 ]   (213.2,97.55) -- (215.2,127.7) ;
\draw [shift={(214.53,117.61)}, rotate = 266.2] [fill={rgb, 255:red, 208; green, 2; blue, 27 }  ,fill opacity=1 ][line width=0.08]  [draw opacity=0] (8.93,-4.29) -- (0,0) -- (8.93,4.29) -- cycle    ;
\draw [color={rgb, 255:red, 208; green, 2; blue, 27 }  ,draw opacity=1 ]   (193.2,108.05) -- (194.2,125.7) ;
\draw [shift={(193.98,121.87)}, rotate = 266.76] [fill={rgb, 255:red, 208; green, 2; blue, 27 }  ,fill opacity=1 ][line width=0.08]  [draw opacity=0] (8.93,-4.29) -- (0,0) -- (8.93,4.29) -- cycle    ;
\draw  [dash pattern={on 4.5pt off 4.5pt}]  (98.82,131.4) -- (572.82,134.4) ;
\draw  [fill={rgb, 255:red, 126; green, 211; blue, 33 }  ,fill opacity=1 ] (180.28,122.73) -- (296.19,132.9) -- (177.82,132.9) .. controls (177.82,129.43) and (178.66,126.03) .. (180.28,122.73) -- cycle ;
\draw  [fill={rgb, 255:red, 208; green, 2; blue, 27 }  ,fill opacity=1 ] (177.6,122.5) .. controls (177.6,121.12) and (178.72,120) .. (180.1,120) .. controls (181.48,120) and (182.6,121.12) .. (182.6,122.5) .. controls (182.6,123.88) and (181.48,125) .. (180.1,125) .. controls (178.72,125) and (177.6,123.88) .. (177.6,122.5) -- cycle ;
\draw  [fill={rgb, 255:red, 208; green, 2; blue, 27 }  ,fill opacity=0.1 ] (335.78,247.4) .. controls (219.15,247.4) and (123.75,200.05) .. (116.5,140.22) -- (116.5,132.9) -- (555.5,132.9) .. controls (555.5,196.13) and (457.13,247.4) .. (335.78,247.4) -- cycle ;

\draw (110,242.4) node [anchor=north west][inner sep=0.75pt]    {$M_{1}$};
\draw (532,29.4) node [anchor=north west][inner sep=0.75pt]    {$M_{2}$};
\draw (332,44.4) node [anchor=north west][inner sep=0.75pt]  [color={rgb, 255:red, 208; green, 2; blue, 27 }  ,opacity=1 ]  {$F_{\lambda ,h}$};

\end{tikzpicture}

    \caption{Illustration of the deformation retraction $F_{\lambda,h}$. The region $\mathcal{F}_{\lambda+Lh^{\alpha}}$ is shown in red, while $\mathcal{F}_{\lambda}^{h}$ is depicted in gray. Points closer to $\mathcal{F}_{\lambda+Lh^{\alpha}}$ experience less displacement under $F_{\lambda,h}$, which ensures its continuity. The green region represents elements of $\mathcal{G}_{\lambda,h}$ that do not belong to $\mathcal{F}_{\lambda+Lh^{\alpha}}$ but will be included in $\mathcal{F}_{\lambda+\varepsilon}$ for some $\varepsilon \asymp h^{\alpha}$, as ensured by Assumptions \textbf{A1} and \textbf{A2} (see proof of Lemma \ref{lemmaFiltEquiv}).}
    \label{fig:enter-label}
\end{figure}
\begin{lmm}
\label{lemmaFiltEquiv}
    For all $0<h<\frac{R}{2}$ and $\lambda\in\mathbb{R}$, $F_{\lambda,h}$ is a deformation retraction of $\mathcal{K}_{\lambda,h}$ onto $\mathcal{G}_{\lambda,h}$. Furthermore, we have $\mathcal{F}_{\lambda}^{h}\subset \mathcal{K}_{\lambda,h}\subset \mathcal{F}_{\lambda}^{2h}$ and $\mathcal{F}_{\lambda}\subset\mathcal{G}_{\lambda,h}\subset \mathcal{F}_{\lambda+L(1+3^{\alpha})h^{\alpha}}$.
\end{lmm}
The proof of Lemma \ref{lemmaFiltEquiv} is provided in Appendix \ref{proof Filtequiv}. Combining the inclusion $\mathcal{F}_{\lambda}^{\sqrt{d}h}\subset \mathcal{K}_{\lambda,\sqrt{d}h}$, the retraction from $\mathcal{K}_{\lambda,\sqrt{d}h}$ to $\mathcal{G}_{\lambda,\sqrt{d}h}$, and the inclusion $\mathcal{G}_{\lambda,\sqrt{d}h}\subset \mathcal{F}_{\lambda+L(1+3^{\alpha})d^{\alpha/2}h^{\alpha}}$, provided by Lemma \ref{lemmaFiltEquiv}, induces a morphism from $(H_{s}(\mathcal{F})_{\lambda}^{\sqrt{d}h}))_{\lambda\in\mathbb{R}}$ to $\mathbb{V}_{f,s}$.\\\\
Now, we construct the morphism from \( \mathbb{V}_{f,s} \) to \( (H_{s}(\mathcal{F}_{\lambda}^{-h}))_{\lambda \in \mathbb{R}} \). The idea is similar to the construction of the previous morphism, with the key difference being that here we want to map points of \( \mathcal{F}_{\lambda} \), that are not in \( \mathcal{F}_{\lambda+\varepsilon}^{-h} \) for reasonable \( \varepsilon \), into \( \mathcal{F}_{\lambda+\varepsilon}^{-h} \). These problematic points are typically contained in:
\[
P_{\lambda,h} = \left( \left( \bigcup_{i=1}^{l} \partial M_i \cap ]0,1[^d \right)^h \setminus \mathcal{F}_{\lambda+2Lh^\alpha}^{-h} \right) \cap \mathcal{F}_\lambda,
\]
which are located near the boundaries of the regular pieces. The idea is to push these points away from the boundaries. To do this, we define the map:  
\[
\gamma_h(x) = 
\begin{cases} 
x + \frac{\left( h - d_2\left(x, \bigcup_{i=1}^{l} \partial M_i \cap ]0,1[^d \right) \right)}{d_2\left(x, \bigcup_{i=1}^{l} \partial M_i \cap ]0,1[^d \right)} \left(x - \xi(x)\right), & \text{if } x \in \left( \bigcup_{i=1}^{l} \partial M_i \cap ]0,1[^d \right)^h \setminus \bigcup_{i=1}^{l} \partial M_i, \\
x, & \text{if } x \notin \left( \bigcup_{i=1}^{l} \partial M_i \cap ]0,1[^d \right)^h.
\end{cases}
\]
The function \( \gamma_h \) pushes points from the \( h \)-neighborhood of the boundaries by following the gradient of the distance function to $\bigcup_{i=1}^{l} \partial M_i \cap ]0,1[^d $, moving them outside of this \( h \)-neighborhood. It should be noted that, for all $x\in  \left( \bigcup_{i=1}^{l} \partial M_i \cap ]0,1[^d \right)^h \setminus \bigcup_{i=1}^{l} \partial M_i$, the vectors $\gamma_{h}(x)-x$ and $\xi(x)-x$ are collinear and point in opposite directions.
Consequently, $\gamma_{h}(x)-x$ also generates the normal cone of $\bigcup_{i=1}^{l} \partial M_i \cap ]0,1[^d $ at $\xi(x)$. Moreover, this function can be continuously extended to \( P_{\lambda,h} \) using Assumption \textbf{A3} (see the proof of Lemma \ref{lemmaFiltEquiv2}). We denote this extension by \( \gamma_{h,\lambda} \). 
\begin{figure}[H]
    \centering
    \begin{tikzpicture}[x=0.75pt,y=0.75pt,yscale=-1,xscale=1]

\draw   (192,17) -- (439,17) -- (439,264) -- (192,264) -- cycle ;
\draw  [draw opacity=0][dash pattern={on 4.5pt off 4.5pt}] (323.33,263.4) .. controls (292.14,238.2) and (271.5,194.36) .. (271.5,144.5) .. controls (271.5,88.18) and (297.83,39.55) .. (335.95,16.77) -- (384.75,144.5) -- cycle ; \draw  [dash pattern={on 4.5pt off 4.5pt}] (323.33,263.4) .. controls (292.14,238.2) and (271.5,194.36) .. (271.5,144.5) .. controls (271.5,88.18) and (297.83,39.55) .. (335.95,16.77) ;  
\draw  [draw opacity=0][dash pattern={on 0.84pt off 2.51pt}] (406.33,263.4) .. controls (375.14,238.2) and (354.5,194.36) .. (354.5,144.5) .. controls (354.5,88.18) and (380.83,39.55) .. (418.95,16.77) -- (467.75,144.5) -- cycle ; \draw  [dash pattern={on 0.84pt off 2.51pt}] (406.33,263.4) .. controls (375.14,238.2) and (354.5,194.36) .. (354.5,144.5) .. controls (354.5,88.18) and (380.83,39.55) .. (418.95,16.77) ;  
\draw [color={rgb, 255:red, 74; green, 144; blue, 226 }  ,draw opacity=1 ]   (193,144.5) -- (439.33,143.83) ;
\draw  [fill={rgb, 255:red, 208; green, 2; blue, 27 }  ,fill opacity=1 ] (267.18,143.5) .. controls (267.18,141.16) and (269.19,139.27) .. (271.67,139.27) .. controls (274.14,139.27) and (276.15,141.16) .. (276.15,143.5) .. controls (276.15,145.84) and (274.14,147.73) .. (271.67,147.73) .. controls (269.19,147.73) and (267.18,145.84) .. (267.18,143.5) -- cycle ;
\draw  [fill={rgb, 255:red, 155; green, 155; blue, 155 }  ,fill opacity=1 ] (295.68,143.5) .. controls (295.68,141.16) and (297.69,139.27) .. (300.17,139.27) .. controls (302.64,139.27) and (304.65,141.16) .. (304.65,143.5) .. controls (304.65,145.84) and (302.64,147.73) .. (300.17,147.73) .. controls (297.69,147.73) and (295.68,145.84) .. (295.68,143.5) -- cycle ;
\draw  [fill={rgb, 255:red, 208; green, 2; blue, 27 }  ,fill opacity=1 ] (351.51,143.83) .. controls (351.51,141.5) and (353.52,139.6) .. (356,139.6) .. controls (358.48,139.6) and (360.49,141.5) .. (360.49,143.83) .. controls (360.49,146.17) and (358.48,148.06) .. (356,148.06) .. controls (353.52,148.06) and (351.51,146.17) .. (351.51,143.83) -- cycle ;
\draw    (274,154.18) -- (353.67,154.49) ;
\draw [shift={(356.67,154.5)}, rotate = 180.22] [fill={rgb, 255:red, 0; green, 0; blue, 0 }  ][line width=0.08]  [draw opacity=0] (7.14,-3.43) -- (0,0) -- (7.14,3.43) -- cycle    ;
\draw [shift={(271,154.17)}, rotate = 0.22] [fill={rgb, 255:red, 0; green, 0; blue, 0 }  ][line width=0.08]  [draw opacity=0] (7.14,-3.43) -- (0,0) -- (7.14,3.43) -- cycle    ;

\draw (304,120.07) node [anchor=north west][inner sep=0.75pt]    {$x$};
\draw (235.67,121.4) node [anchor=north west][inner sep=0.75pt]    {$\xi ( x)$};
\draw (361.33,120.07) node [anchor=north west][inner sep=0.75pt]    {$\gamma _{h}( x)$};
\draw (197.33,235.73) node [anchor=north west][inner sep=0.75pt]    {$M_{1}$};
\draw (337.33,34.4) node [anchor=north west][inner sep=0.75pt]    {$M_{2}$};
\draw (307.67,157.73) node [anchor=north west][inner sep=0.75pt]    {$h$};

\end{tikzpicture}
    \caption{Illustration of $\xi$ and $\gamma_h$. In blue is represented the normal cone of $\bigcup_{i=1}^{l} \partial M_i \cap ]0,1[^d $ at $\xi(x)$.  }
\end{figure}

Our goal now is to use \( \gamma_{h,\lambda} \) to push smoothly points of \( P_{\lambda,h} \), constructing again a deformation retraction. Similarly to what we did in Lemma \ref{lemmaFiltEquiv}, for some \( x \in P_{\lambda,h} \), the line segment \( [x, \gamma_{h,\lambda}(x)] \) may not be contained entirely in \( \mathcal{F}_{\lambda} \). Therefore, we consider the larger set, adding the segments \( [x, \gamma_{h,\lambda}(x)] \) for all \( x \in P_{\lambda,h} \):

\[
\mathcal{N}_{\lambda,h} := \mathcal{F}_{\lambda} \cup \left( \bigcup\limits_{x \in P_{\lambda,h}} [x, \gamma_{h,\lambda}(x)] \right) \subset \mathcal{F}_{\lambda + Lh^\alpha}.
\]
\begin{figure}[h]
    \centering
\begin{tikzpicture}[x=0.75pt,y=0.75pt,yscale=-1,xscale=1]

\draw  [color={rgb, 255:red, 0; green, 0; blue, 0 }  ,draw opacity=1 ][fill={rgb, 255:red, 74; green, 144; blue, 226 }  ,fill opacity=1 ] (496.53,215.57) .. controls (496.82,216.33) and (497.08,217.09) .. (497.33,217.87) -- (416.36,264.13) .. controls (413.99,260.33) and (411.53,256.59) .. (408.97,252.92) .. controls (400.31,237.26) and (390.45,222.32) .. (379.5,208.24) .. controls (421.59,194.33) and (487.3,222.3) .. (496.53,215.57) -- cycle ;
\draw  [dash pattern={on 0.84pt off 2.51pt}]  (495.14,217.09) -- (517.73,203.1) ;
\draw   (42,26) -- (283.5,26) -- (283.5,267.5) -- (42,267.5) -- cycle ;
\draw  [draw opacity=0][fill={rgb, 255:red, 155; green, 155; blue, 155 }  ,fill opacity=1 ] (169.5,44.6) .. controls (194.2,32.41) and (232.53,61.2) .. (242.5,70.6) .. controls (252.47,80) and (227.49,65.65) .. (222.5,76.6) .. controls (217.51,87.55) and (218.31,98.79) .. (204.73,110.1) .. controls (191.16,121.41) and (185.6,194.96) .. (169.54,220.57) .. controls (153.49,246.17) and (51.62,206.42) .. (75.7,164.48) .. controls (99.78,122.54) and (144.8,56.79) .. (169.5,44.6) -- cycle ;
\draw  [fill={rgb, 255:red, 208; green, 2; blue, 27 }  ,fill opacity=0.07 ] (206.5,63.8) .. controls (233.5,79.8) and (232.6,92.61) .. (200.5,142.6) .. controls (168.4,192.59) and (214.55,227.99) .. (198.5,253.6) .. controls (182.45,279.21) and (26.42,203.54) .. (50.5,161.6) .. controls (74.58,119.66) and (82.89,88.29) .. (105.5,72.8) .. controls (128.11,57.31) and (179.5,47.8) .. (206.5,63.8) -- cycle ;
\draw  [dash pattern={on 4.5pt off 4.5pt}] (104.72,124.43) .. controls (104.72,78.13) and (142.72,40.6) .. (189.61,40.6) .. controls (236.49,40.6) and (274.5,78.13) .. (274.5,124.43) .. controls (274.5,170.72) and (236.49,208.25) .. (189.61,208.25) .. controls (142.72,208.25) and (104.72,170.72) .. (104.72,124.43) -- cycle ;
\draw  [dash pattern={on 0.84pt off 2.51pt}] (211.5,77.6) .. controls (211.5,63.79) and (222.69,52.6) .. (236.5,52.6) .. controls (250.31,52.6) and (261.5,63.79) .. (261.5,77.6) .. controls (261.5,91.41) and (250.31,102.6) .. (236.5,102.6) .. controls (222.69,102.6) and (211.5,91.41) .. (211.5,77.6) -- cycle ;
\draw  [dash pattern={on 0.84pt off 2.51pt}]  (236.5,52.6) -- (313,26) ;
\draw  [dash pattern={on 0.84pt off 2.51pt}]  (236.5,102.6) -- (313,267.5) ;
\draw  [color={rgb, 255:red, 0; green, 0; blue, 0 }  ,draw opacity=1 ][fill={rgb, 255:red, 155; green, 155; blue, 155 }  ,fill opacity=1 ] (481.93,193.39) .. controls (543.2,261.86) and (389.75,157.32) .. (359.08,237.11) .. controls (355.23,247.12) and (351.94,257.17) .. (348.97,267.24) -- (314.9,267.24) .. controls (314.59,266.87) and (314.27,266.49) .. (313.95,266.12) -- (313.95,50.37) .. controls (387.16,87.02) and (447.58,155) .. (481.93,193.39) -- cycle ;
\draw   (313,26) -- (554.5,26) -- (554.5,267.5) -- (313,267.5) -- cycle ;
\draw  [draw opacity=0][dash pattern={on 4.5pt off 4.5pt}] (367.74,25.59) .. controls (398.63,41.21) and (437.45,77.14) .. (471.89,124.79) .. controls (510.67,178.46) and (533.57,232.74) .. (534.51,267.07) -- (423.23,159.95) -- cycle ; \draw  [dash pattern={on 4.5pt off 4.5pt}] (367.74,25.59) .. controls (398.63,41.21) and (437.45,77.14) .. (471.89,124.79) .. controls (510.67,178.46) and (533.57,232.74) .. (534.51,267.07) ;  
\draw  [fill={rgb, 255:red, 74; green, 74; blue, 74 }  ,fill opacity=1 ] (491.43,217.87) .. controls (491.43,216.24) and (492.75,214.92) .. (494.38,214.92) .. controls (496.01,214.92) and (497.33,216.24) .. (497.33,217.87) .. controls (497.33,219.5) and (496.01,220.82) .. (494.38,220.82) .. controls (492.75,220.82) and (491.43,219.5) .. (491.43,217.87) -- cycle ;
\draw  [fill={rgb, 255:red, 74; green, 74; blue, 74 }  ,fill opacity=1 ] (413.8,263.24) .. controls (413.8,261.61) and (415.13,260.29) .. (416.75,260.29) .. controls (418.38,260.29) and (419.7,261.61) .. (419.7,263.24) .. controls (419.7,264.87) and (418.38,266.19) .. (416.75,266.19) .. controls (415.13,266.19) and (413.8,264.87) .. (413.8,263.24) -- cycle ;
\draw  [fill={rgb, 255:red, 74; green, 74; blue, 74 }  ,fill opacity=1 ] (514.78,203.1) .. controls (514.78,201.47) and (516.1,200.15) .. (517.73,200.15) .. controls (519.36,200.15) and (520.68,201.47) .. (520.68,203.1) .. controls (520.68,204.73) and (519.36,206.05) .. (517.73,206.05) .. controls (516.1,206.05) and (514.78,204.73) .. (514.78,203.1) -- cycle ;
\draw [color={rgb, 255:red, 155; green, 155; blue, 155 }  ,draw opacity=1 ]   (425.41,264.48) -- (518.89,211.48) ;
\draw [shift={(521.5,210)}, rotate = 150.45] [fill={rgb, 255:red, 155; green, 155; blue, 155 }  ,fill opacity=1 ][line width=0.08]  [draw opacity=0] (8.93,-4.29) -- (0,0) -- (8.93,4.29) -- cycle    ;
\draw [shift={(422.8,265.96)}, rotate = 330.45] [fill={rgb, 255:red, 155; green, 155; blue, 155 }  ,fill opacity=1 ][line width=0.08]  [draw opacity=0] (8.93,-4.29) -- (0,0) -- (8.93,4.29) -- cycle    ;
\draw  [color={rgb, 255:red, 0; green, 0; blue, 0 }  ,draw opacity=1 ] (242.5,70.6) .. controls (252.47,80) and (227.49,65.65) .. (222.5,76.6) .. controls (222.44,76.74) and (222.37,76.88) .. (222.31,77.02) .. controls (219.17,72.3) and (213.87,68.17) .. (206.5,63.8) .. controls (192.86,55.72) and (173,54.15) .. (153.77,56.57) .. controls (159.54,51.04) and (164.88,46.88) .. (169.5,44.6) .. controls (194.2,32.41) and (232.53,61.2) .. (242.5,70.6) -- cycle ;
\draw  [fill={rgb, 255:red, 208; green, 2; blue, 27 }  ,fill opacity=0.13 ] (313.95,267.24) -- (313.95,142.06) .. controls (357.37,174.89) and (392.76,217.77) .. (416.66,267.24) -- (313.95,267.24) -- cycle ;

\draw (209,162.4) node [anchor=north west][inner sep=0.75pt]    {$M_{1}$};
\draw (55,233.4) node [anchor=north west][inner sep=0.75pt]    {$M_{2}$};
\draw (177.36,40.49) node [anchor=north west][inner sep=0.75pt]    {$P_{\lambda ,h}$};
\draw (138,116.4) node [anchor=north west][inner sep=0.75pt]    {$\mathcal{F}_{\lambda }$};
\draw (142,227.4) node [anchor=north west][inner sep=0.75pt]    {$\mathcal{\textcolor[rgb]{0.82,0.01,0.11}{F}}\textcolor[rgb]{0.82,0.01,0.11}{_{\lambda +2Lh^{\alpha }}^{-h}}$};
\draw (516.96,179.68) node [anchor=north west][inner sep=0.75pt]    {$\xi ( x)$};
\draw (476.76,198.68) node [anchor=north west][inner sep=0.75pt]    {$x$};
\draw (356.24,242.88) node [anchor=north west][inner sep=0.75pt]    {$\gamma _{\lambda ,h}( x)$};
\draw (478.38,233.22) node [anchor=north west][inner sep=0.75pt]    {$\textcolor[rgb]{0.61,0.61,0.61}{h}$};
\draw (498,243.4) node [anchor=north west][inner sep=0.75pt]  [color={rgb, 255:red, 74; green, 144; blue, 226 }  ,opacity=1 ]  {$\mathcal{N}_{\lambda ,h}$};

\end{tikzpicture}

    \caption{Illustration of the sets $\mathcal{F}_{\lambda}$, $\mathcal{F}^{-h}_{\lambda+2Lh^{\alpha}}$, $P_{\lambda,h}$, and $\mathcal{N}_{\lambda,h}$. The region $\mathcal{F}_{\lambda+2Lh^{\alpha}}^{-h}$ is shown in red, while $\mathcal{F}_{\lambda}$ is depicted in gray. The part of $\mathcal{F}_{\lambda}$ not included in $\mathcal{F}_{\lambda+2Lh^{\alpha}}^{-h}$ corresponds to $P_{\lambda,h}$. The blue region represents the union of segments added to $\mathcal{F}_{\lambda}$ to form $\mathcal{N}_{\lambda,h}$.}
    \label{fig:enter-label}
\end{figure}

To ensure the continuity of the deformation retraction, we need to control how much we move the points of \( \bigcup_{x \in P_{\lambda,h}} [x, \gamma_{h,\lambda}(x)] \cap M_i \) that are close to \( \mathcal{F}_{\lambda + 2Lh^\alpha} \cap (\overline{M_i})^c \). Thus, we introduce condition (\textbf{C2}): there exists \( i \in \{1, \dots, l\} \) such that
\begin{equation}
\label{cond 3}
x \in \bigcup\limits_{x \in P_{\lambda,h}} [x, \gamma_{h,\lambda}(x)] \cap M_i
\end{equation}
and
\begin{equation}
\label{cond 4}
d_2\left(\gamma_{h,\lambda}(x), (\overline{M_i})^c \cap \mathcal{F}_{\lambda + 2Lh^\alpha}\right) \geq 3h - \|x - \gamma_{h,\lambda}(x)\|_2.
\end{equation}

We then define \( H_{\lambda,h}: \mathcal{N}_{\lambda,h} \times [0,1] \rightarrow \mathcal{N}_{\lambda,h} \), which will serve as our deformation retraction, by $H_{\lambda,h}(x) = $
\[
\begin{cases} 
(1 - t)x + t \left( \gamma_{h,\lambda}(x) + \left( 3h - d_2\left( \gamma_{h,\lambda}(x), (\overline{M_i})^c \cap \mathcal{F}_{\lambda + 2Lh^\alpha} \right) \right)_+ \frac{x - \gamma_{h,\lambda}(x)}{\|x - \gamma_{h,\lambda}(x)\|_2} \right), & \text{if (\textbf{C2})}, \\
x, & \text{else}.
\end{cases}
\]
Finally, we denote \( \mathcal{M}_{\lambda,h} = \operatorname{Im}(x \mapsto H_{\lambda,h}(x,1)) \), which will serve as our deformation retract. We can now state Lemma \ref{lemmaFiltEquiv2}.
\begin{figure}[H]
    \centering
\begin{tikzpicture}[x=0.75pt,y=0.75pt,yscale=-1,xscale=1]

\draw  [fill={rgb, 255:red, 155; green, 155; blue, 155 }  ,fill opacity=1 ] (589.82,117.8) -- (589.82,210.6) -- (261.5,210.6) .. controls (249.79,222.34) and (233.59,229.6) .. (215.7,229.6) .. controls (179.97,229.6) and (151,200.63) .. (151,164.9) .. controls (151,129.17) and (179.97,100.2) .. (215.7,100.2) .. controls (232.87,100.2) and (248.47,106.89) .. (260.05,117.8) -- (589.82,117.8) -- cycle ;
\draw   (116,34) -- (589.5,34) -- (589.5,290) -- (116,290) -- cycle ;
\draw [color={rgb, 255:red, 208; green, 2; blue, 27 }  ,draw opacity=1 ] [dash pattern={on 4.5pt off 4.5pt}]  (115.32,166.73) -- (589.32,169.73) ;
\draw  [dash pattern={on 4.5pt off 4.5pt}]  (115.82,114.8) -- (589.82,117.8) ;
\draw    (600.7,120.85) -- (600.7,166.35) ;
\draw [shift={(600.7,169.35)}, rotate = 270] [fill={rgb, 255:red, 0; green, 0; blue, 0 }  ][line width=0.08]  [draw opacity=0] (8.93,-4.29) -- (0,0) -- (8.93,4.29) -- cycle    ;
\draw [shift={(600.7,117.85)}, rotate = 90] [fill={rgb, 255:red, 0; green, 0; blue, 0 }  ][line width=0.08]  [draw opacity=0] (8.93,-4.29) -- (0,0) -- (8.93,4.29) -- cycle    ;
\draw  [fill={rgb, 255:red, 208; green, 2; blue, 27 }  ,fill opacity=0.11 ] (275.07,238.15) .. controls (257.47,260.48) and (232.44,274.42) .. (204.67,274.42) .. controls (164.92,274.42) and (130.79,245.84) .. (116,205.02) -- (116,114.29) .. controls (130.79,73.47) and (164.92,44.89) .. (204.67,44.89) .. controls (247.82,44.89) and (284.35,78.57) .. (296.7,125.03) .. controls (313.76,153.05) and (345.83,161.07) .. (367.39,169.73) -- (589.32,169.73) -- (589.32,290) -- (432.93,290) .. controls (369.79,278.4) and (314.19,259.53) .. (275.07,238.15) -- cycle ;
\draw [color={rgb, 255:red, 208; green, 2; blue, 27 }  ,draw opacity=1 ]   (575.4,116.6) -- (575.5,170.6) ;
\draw [shift={(575.46,148.6)}, rotate = 269.89] [fill={rgb, 255:red, 208; green, 2; blue, 27 }  ,fill opacity=1 ][line width=0.08]  [draw opacity=0] (8.93,-4.29) -- (0,0) -- (8.93,4.29) -- cycle    ;
\draw [color={rgb, 255:red, 208; green, 2; blue, 27 }  ,draw opacity=1 ]   (544.4,118.6) -- (544.5,169.6) ;
\draw [shift={(544.46,149.1)}, rotate = 269.89] [fill={rgb, 255:red, 208; green, 2; blue, 27 }  ,fill opacity=1 ][line width=0.08]  [draw opacity=0] (8.93,-4.29) -- (0,0) -- (8.93,4.29) -- cycle    ;
\draw [color={rgb, 255:red, 208; green, 2; blue, 27 }  ,draw opacity=1 ]   (514.4,118.6) -- (514.5,169.6) ;
\draw [shift={(514.46,149.1)}, rotate = 269.89] [fill={rgb, 255:red, 208; green, 2; blue, 27 }  ,fill opacity=1 ][line width=0.08]  [draw opacity=0] (8.93,-4.29) -- (0,0) -- (8.93,4.29) -- cycle    ;
\draw [color={rgb, 255:red, 208; green, 2; blue, 27 }  ,draw opacity=1 ]   (481.4,117.6) -- (481.5,170.2) ;
\draw [shift={(481.46,148.9)}, rotate = 269.89] [fill={rgb, 255:red, 208; green, 2; blue, 27 }  ,fill opacity=1 ][line width=0.08]  [draw opacity=0] (8.93,-4.29) -- (0,0) -- (8.93,4.29) -- cycle    ;
\draw [color={rgb, 255:red, 208; green, 2; blue, 27 }  ,draw opacity=1 ]   (449.4,117.6) -- (449.5,170.2) ;
\draw [shift={(449.46,148.9)}, rotate = 269.89] [fill={rgb, 255:red, 208; green, 2; blue, 27 }  ,fill opacity=1 ][line width=0.08]  [draw opacity=0] (8.93,-4.29) -- (0,0) -- (8.93,4.29) -- cycle    ;
\draw [color={rgb, 255:red, 208; green, 2; blue, 27 }  ,draw opacity=1 ]   (417.38,117.78) -- (418.4,164.8) ;
\draw [shift={(418,146.29)}, rotate = 268.76] [fill={rgb, 255:red, 208; green, 2; blue, 27 }  ,fill opacity=1 ][line width=0.08]  [draw opacity=0] (8.93,-4.29) -- (0,0) -- (8.93,4.29) -- cycle    ;
\draw [color={rgb, 255:red, 208; green, 2; blue, 27 }  ,draw opacity=1 ]   (385.5,117.2) -- (386.5,154.2) ;
\draw [shift={(386.14,140.7)}, rotate = 268.45] [fill={rgb, 255:red, 208; green, 2; blue, 27 }  ,fill opacity=1 ][line width=0.08]  [draw opacity=0] (8.93,-4.29) -- (0,0) -- (8.93,4.29) -- cycle    ;
\draw [color={rgb, 255:red, 208; green, 2; blue, 27 }  ,draw opacity=1 ]   (352.31,117.78) -- (352.5,140.2) ;
\draw [shift={(352.45,133.99)}, rotate = 269.51] [fill={rgb, 255:red, 208; green, 2; blue, 27 }  ,fill opacity=1 ][line width=0.08]  [draw opacity=0] (8.93,-4.29) -- (0,0) -- (8.93,4.29) -- cycle    ;
\draw [color={rgb, 255:red, 208; green, 2; blue, 27 }  ,draw opacity=1 ]   (325.88,117.43) -- (326.46,131.94) ;
\draw [shift={(326.37,129.68)}, rotate = 267.7] [fill={rgb, 255:red, 208; green, 2; blue, 27 }  ,fill opacity=1 ][line width=0.08]  [draw opacity=0] (8.93,-4.29) -- (0,0) -- (8.93,4.29) -- cycle    ;
\draw  [fill={rgb, 255:red, 126; green, 211; blue, 33 }  ,fill opacity=1 ] (428.72,169.62) -- (367.2,169.62) .. controls (345.72,161.08) and (313.78,153.16) .. (296.79,125.5) .. controls (296.35,123.87) and (295.88,122.25) .. (295.38,120.65) -- (296.19,118.55) -- (428.72,169.62) -- cycle ;
\draw  [fill={rgb, 255:red, 208; green, 2; blue, 27 }  ,fill opacity=1 ] (292.35,117.5) .. controls (292.35,116.12) and (293.47,115) .. (294.85,115) .. controls (296.23,115) and (297.35,116.12) .. (297.35,117.5) .. controls (297.35,118.88) and (296.23,120) .. (294.85,120) .. controls (293.47,120) and (292.35,118.88) .. (292.35,117.5) -- cycle ;

\draw (130,258.8) node [anchor=north west][inner sep=0.75pt]    {$M_{1}$};
\draw (550,41.4) node [anchor=north west][inner sep=0.75pt]    {$M_{2}$};
\draw (609.2,133.6) node [anchor=north west][inner sep=0.75pt]    {$h$};
\draw (436,79) node [anchor=north west][inner sep=0.75pt]  [color={rgb, 255:red, 208; green, 2; blue, 27 }  ,opacity=1 ]  {$H_{\lambda ,h}$};

\end{tikzpicture}

 \caption{Illustration of the deformation retraction $H_{\lambda,h}$. The region $\mathcal{F}^{-h}_{\lambda+2Lh^{\alpha}}$ is shown in red, while $\mathcal{F}_{\lambda}$ is depicted in gray. Points closer to $\mathcal{F}^{-h}_{\lambda+2Lh^{\alpha}}$ experience less displacement under $F_{\lambda,h}$, which ensures its continuity. The green region represents elements of $\mathcal{M}_{\lambda,h}$ that do not belong to $\mathcal{F}^{-h}_{\lambda+2Lh^{\alpha}}$ but will be included in $\mathcal{F}^{-h}_{\lambda+\varepsilon}$ for some $\varepsilon \asymp h^{\alpha}$ (see proof of Lemma \ref{lemmaFiltEquiv2}).}
    \label{fig:enter-label}
\end{figure}
\begin{lmm}
\label{lemmaFiltEquiv2}
For all $0<h<\frac{R}{2}$ and $\lambda\in\mathbb{R}$, $H_{\lambda,h}$ is a deformation retraction from $\mathcal{N}_{\lambda,h}$ onto $\mathcal{M}_{\lambda,h}$. Furthermore, we have $\mathcal{F}_{\lambda}^{-h}\subset\mathcal{M}_{\lambda,h}\subset \mathcal{F}_{\lambda+(2+5^{\alpha})Lh^{\alpha}}^{-h}$ and $\mathcal{F}_{\lambda}\subset\mathcal{N}_{\lambda,h}\subset \mathcal{F}_{\lambda+Lh^{\alpha}}.$
\end{lmm}
The proof of Lemma \ref{lemmaFiltEquiv2} is provided in Appendix \ref{proof Filtequiv 2}. Similarly to Lemma \ref{lemmaFiltEquiv}, Lemma \ref{lemmaFiltEquiv2} provides a morphism from $\mathbb{V}_{f,s}$ to $(H_{s}(\mathcal{F}_{\lambda}^{-\sqrt{d}h}))_{\lambda\in\mathbb{R}}$.\\\\ 
Next, we state a technical result: Lemma \ref{lemma-histo 2}, which will allow us to prove that the morphisms we constructed previously induce an interleaving. 
\begin{lmm}
\label{lemma-histo 2}
Let $\lambda\in\mathbb{R}$ and $h\asymp \left(\frac{\log(n)}{n}\right)^{\frac{1}{d+2\alpha}}$, for sufficiently large $n$, we have, for all $f\in S_{d}(M,L,\alpha,R)$, for all $x\in \widehat{\mathcal{F}}_{\lambda}\cap S_{\lambda+\sqrt{2}\sigma N_h h^{\alpha},\sqrt{d}h}$,
\begin{equation}
\label{lemma-histo 2a}
\left[x,\xi(x)\right]\subset \widehat{\mathcal{F}}_{\lambda+(c_{1}\sqrt{2}\sigma N_{h}+c_{2})h^{\alpha}}
\end{equation}
and for all $x\in \widehat{\mathcal{F}}_{\lambda}\cap P_{\lambda,\sqrt{d}h}$,
\begin{equation}
\label{lemma-histo 2b}
\left[x,\gamma_{\lambda,\sqrt{d}h}(x)\right]\subset \widehat{\mathcal{F}}_{\lambda+(c_{1}\sqrt{2}\sigma N_{h}+c_{2})h^{\alpha}}
\end{equation}
with $c_{1}=2(1+(2d)^{d})$, $c_{2}=(1+(2d)^{d})(5L\sqrt{d}+2\kappa M)$ and $\kappa$ a constant depending only on $d$ and $R$. 
\end{lmm}
In particular, as $\widehat{\mathcal{F}}_{\lambda}\subset\mathcal{F}_{\lambda+\sqrt{2}\sigma N_h h^{\alpha}}^{\sqrt{d}h}$ (by Lemma \ref{interleavish 3}), Assertion (\ref{lemma-histo 2a}) of Lemma~\ref{lemma-histo 2} ensures that, for sufficiently large $n$ and for all \( x \in \widehat{\mathcal{F}}_{\lambda} \):
\[
F_{\lambda+\sqrt{2}\sigma N_{h}h^{\alpha},\sqrt{d}h}(x,[0,1])\subset\widehat{\mathcal{F}}_{\lambda+ch}
\]
for some constant \( c \) depending on the model parameters and \( N_h \). This implies that any chain \( C\in C_{s}(\widehat{\mathcal{F}}_{\lambda}) \), \( s \in \mathbb{N} \), is homologically equivalent, within \( \widehat{\mathcal{F}}_{\lambda+ch} \), to its deformed counterpart:
$$F^{\#}_{\lambda+\sqrt{2}\sigma N_{h}h^{\alpha},\sqrt{d}h}(C,1).$$
Similar guarantees also hold for the deformation retraction involved in Lemma \ref{lemmaFiltEquiv2}. More precisely, it follows from Assertion (\ref{lemma-histo 2b}) of Lemma~\ref{lemma-histo 2}, that for all \( x \in \widehat{\mathcal{F}}_{\lambda}\cap \mathcal{F}_{\lambda}\), 
$$H_{\lambda,\sqrt{d}h}(x,[0,1])\subset \widehat{\mathcal{F}}_{\lambda+ch}$$
and consequently for any $C\in C_{s}\left(\widehat{\mathcal{F}}_{\lambda}\cap \mathcal{F}_{\lambda}\right)$, $C$ is homologically equivalent, within \( \widehat{\mathcal{F}}_{\lambda+ch} \) to:
$$H^{\#}_{\lambda,\sqrt{d}h}(C,1).$$
These key properties will be instrumental in showing that the morphisms \( \phi \) and \( \psi \) satisfy Definition~\ref{def: interleav}, in particular, to prove that the diagrams involved commute. This will allow us to establish an interleaving with a constant depending on \( N_h \). To complete the proof, we will therefore require a concentration inequality for \( N_h \), which is stated in the following lemma and proved in Appendix~\ref{proof noise}.
\begin{lmm}

\label{lemma noise 2}
Let $h>\sfrac{1}{N}$,
$$\mathbb{P}\left(N_{h}\geq t\right)\leq 2\left(\frac{1}{h}\right)^{d}\exp\left(-t^{2}\log\left(1/h^{d}\right)\right).$$
Consequently, there exist two constants $C_{0}$ and $C_{1}$ depending only on $d$ such that, for all $h<1$,
$$\mathbb{P}\left(N_{h}\leq t \right)\leq C_{0}\exp\left(-C_{1}t^{2}\right).$$
\end{lmm}
\label{Upperbounds section}
Equipped with these lemmas, we can now prove Theorem \ref{MainProp}, from which follows Theorem \ref{estimation borne sup}.
\begin{thm}
\label{MainProp}
Let $h\asymp \left(\frac{\log(n)}{n}\right)^{\frac{1}{d+2\alpha}}$. There exist $\Tilde{C_{0}}$ and $\Tilde{C_{1}}$ (depending only on $M$, $L$, $\alpha$, $R$ and $\sigma$) such that, for all $t>0$,
$$\sup\limits_{f\in S_{d}(M,L,\alpha,R)}\mathbb{P}\left(d_{b}\left(\widehat{\operatorname{dgm}(f)},\operatorname{dgm}(f)\right)\geq t\left(\frac{\log(n)}{n}\right)^{\frac{\alpha}{d+2\alpha}}\right)\leq \Tilde{C_{0}}\exp\left(-\Tilde{C_{1}}t^{2}\right).$$ 
\end{thm}
\begin{proof}
It suffices to show the result for large $n$ (up to rescaling $\Tilde{C_{0}}$). Hence, we assume that $n$ is such that $2\sqrt{d}h<R$ and Lemma \ref{lemma-histo 2} holds. The proof is essentially a combination of the lemmas stated throughout this section. We begin by specifying the exact construction of $\phi_{\lambda}$, which will be obtained by the composition of the following linear maps:
$$j_{1,\lambda}:H_{s}\left(\widehat{\mathcal{F}}_{\lambda}\right)\rightarrow H_{s}\left(\mathcal{K}_{\lambda+\sqrt{2}\sigma N_{h}h^{\alpha},\sqrt{d}h}\right)$$
the map induced by the inclusion $\widehat{\mathcal{F}}_{\lambda}\subset  \mathcal{F}_{\lambda+\sqrt{2}\sigma N_{h}h^{\alpha}}^{\sqrt{d}h}\subset \mathcal{K}_{\lambda+\sqrt{2}\sigma N_{h}h^{\alpha},\sqrt{d}h}$ obtained by combining Lemma \ref{interleavish 3} and \ref{lemmaFiltEquiv},
$$j_{2,\lambda}:H_{s}\left(\mathcal{K}_{\lambda+\sqrt{2}\sigma N_{h}h^{\alpha},\sqrt{d}h}\right)\rightarrow H_{s}\left(\mathcal{G}_{\lambda+\sqrt{2}\sigma N_{h}h^{\alpha},\sqrt{d}h}\right)$$
induced by the deformation retraction from Lemma \ref{lemmaFiltEquiv}, and, with $k_{1}=Ld^{\alpha/2}\left(1+3^{\alpha}\right)+c_{2}$,
$$j_{3,\lambda}:H_{s}\left(\mathcal{G}_{\lambda+\sqrt{2}\sigma N_{h}h^{\alpha},\sqrt{d}h}\right)\rightarrow H_{s}\left(\mathcal{F}_{\lambda+\left(k_{1}+(1+c_{1})\sqrt{2}\sigma N_{h}\right)h^{\alpha}}\right)$$
the map induced by inclusion, following again from Lemma \ref{lemmaFiltEquiv}. Now, we can define:
$$
\left\{
    \begin{array}{ll}
\phi_{\lambda}:H_{s}\left(\widehat{\mathcal{F}}_{\lambda}\right)\rightarrow H_{s}\left(\mathcal{F}_{\lambda+\left(k_{1}+(1+c_{1})\sqrt{2}\sigma N_{h}\right)h^{\alpha}}\right)\\
    \phi_{\lambda}= j_{3,\lambda}\circ j_{2,\lambda}\circ j_{1,\lambda}
    \end{array}
\right.
$$
This provides the first persistence module morphism $\phi$. Let us construct the second one by specifying the construction of $\psi_{\lambda}$, which will be obtained by composition of the following linear maps:
$$j_{4,\lambda}:H_{s}\left(\mathcal{F}_{\lambda}\right)\rightarrow H_{s}\left(\mathcal{N}_{\lambda,\sqrt{d}h}\right)$$
the map induced by the inclusion $\mathcal{F}_{\lambda}\subset\mathcal{N}_{\lambda,\sqrt{d}h}$ from Lemma \ref{lemmaFiltEquiv2},
$$j_{5,\lambda}:H_{s}\left(\mathcal{N}_{\lambda,\sqrt{d}h}\right)\rightarrow H_{s}\left(\mathcal{M}_{\lambda,\sqrt{d}h}\right)$$
the map induced by the deformation retraction from Lemma \ref{lemmaFiltEquiv2}, and, with $k_{2}=Ld^{\alpha/2}(2+5^{\alpha})+c_{2}$,
$$j_{6,\lambda}:H_{s}\left(\mathcal{M}_{\lambda,\sqrt{d}h}\right)\rightarrow H_{s}\left(\widehat{\mathcal{F}}_{\lambda+\left(k_{2}+(1+c_{1})\sqrt{2}\sigma N_{h}\right)h^{\alpha}}\right)$$
induced by the inclusion $\mathcal{M}_{\lambda,\sqrt{d}h}\subset \mathcal{F}_{\lambda+k_{2}h^{\alpha}}^{-\sqrt{d}h} \subset \widehat{\mathcal{F}}_{\lambda+\left(k_{2}+(1+c_1)\sqrt{2}\sigma N_{h}\right)h^{\alpha}}$, from the combination of Lemma \ref{interleavish 3} and \ref{lemmaFiltEquiv2}. We then define:
$$
\left\{
    \begin{array}{ll}
\psi_{\lambda}:H_{s}\left(\mathcal{F}_{\lambda}\right)\longrightarrow H_{s}\left(\widehat{\mathcal{F}}_{\lambda+\left(k_{2}+(1+c_{1})\sqrt{2}\sigma N_{h}\right) h^{\alpha}}\right)\\
\psi_{\lambda}=j_{6,\lambda}\circ j_{5,\lambda}\circ j_{4,\lambda}
    \end{array}
\right.
$$
which provides the second morphism $\psi$. Now, we show that $\psi$ and $\phi$ induce an interleaving between $\widehat{\mathbb{V}}_{f,s}$ and $\mathbb{V}_{s,f}$. More precisely, we show that the following diagrams commute for all $\lambda<\lambda'$. For compactness of notations, let
$K_{1}=k_{1}+(1+c_{1})\sqrt{2}\sigma N_{h}$ and $K_{2}=k_{2}+(1+c_{1})\sqrt{2}\sigma N_{h}$.
\begin{equation}
\label{diagram2}
\begin{tikzcd}
	{H_{s}\left(\widehat{\mathcal{F}}_{\lambda}\right)} && {H_{s}\left(\widehat{\mathcal{F}}_{\lambda'}\right)} \\
	\\
	{H_{s}\left(\mathcal{F}_{\lambda+K_{1}h^{\alpha}}\right)} && {H_{s}\left(\mathcal{F}_{\lambda'+K_{1}h^{\alpha}}\right)}
	\arrow[from=1-1, to=3-1,"\phi_{\lambda}"]
	\arrow[from=1-1, to=1-3,"\widehat{v}_{\lambda}^{\lambda'}"]
	\arrow[from=1-3, to=3-3,"\phi_{\lambda'}"]
	\arrow[from=3-1, to=3-3,"v_{\lambda+K_{1}h^{\alpha}}^{\lambda'+K_{1}h^{\alpha}}"]
\end{tikzcd}
\end{equation}
\begin{equation}
\label{diagram1}
\begin{tikzcd}
	{H_{s}\left(\mathcal{F}_{\lambda}\right)} &&& {H_{s}\left(\mathcal{F}_{\lambda'}\right)} \\
	\\
	{H_{s}\left(\widehat{\mathcal{F}}_{\lambda+K_{2}h^{\alpha}}\right)} &&& {H_{s}\left(\widehat{\mathcal{F}}_{\lambda'+K_{2}h^{\alpha}}\right)}
	\arrow[from=1-1, to=3-1,"\psi_{\lambda}"]
	\arrow[from=1-1, to=1-4,"v_{\lambda}^{\lambda'}"]
	\arrow[from=1-4, to=3-4,"\psi_{\lambda'}"]
	\arrow[from=3-1, to=3-4,"\widehat{v}_{\lambda+K_{2}h^{\alpha}}^{\lambda'+K_{2}h^{\alpha}} "]
\end{tikzcd}
\end{equation}
\begin{equation}
\label{diagram3}
\begin{tikzcd}
	{H_{s}\left(\widehat{\mathcal{F}}_{\lambda}\right)} && {H_{s}\left(\widehat{\mathcal{F}}_{\lambda+(K_{1}+K_{2})h^{\alpha}}\right)} \\
	\\
	& {H_{s}\left(\mathcal{F}_{\lambda+K_{1}h^{\alpha}}\right)}
	\arrow[from=1-1, to=3-2,"\phi_{\lambda}"]
	\arrow[from=3-2, to=1-3,"\psi_{\lambda+K_{1}h^{\alpha}}"]
	\arrow[from=1-1, to=1-3,"\widehat{v}_{\lambda}^{\lambda+(K_{1}+K_{2})h^{\alpha}}"]
\end{tikzcd}
\end{equation}
\begin{equation}
\label{diagram4}
\begin{tikzcd}
	{H_{s}\left(\mathcal{F}_{\lambda}\right)} && {H_{s}\left(\mathcal{F}_{\lambda+(K_{1}+K_{2})h^{\alpha}}\right)} \\
	\\
	& {H_{s}\left(\widehat{\mathcal{F}}_{\lambda+K_{2}h^{\alpha}}\right)}
	\arrow[from=1-1, to=1-3,"v_{\lambda}^{\lambda+(K_{1}+K_{2})h^{\alpha}}"]
	\arrow[from=1-1, to=3-2,"\psi_{\lambda}"]
	\arrow[from=3-2, to=1-3,"\phi_{\lambda+K_{2}h^{\alpha}}"]
\end{tikzcd}
\end{equation}
\begin{itemize}
\item  \textbf{Diagram \ref{diagram2}:} We can rewrite the diagram as (unspecified maps are simply induced by set inclusion),
 \begin{equation*}
    \begin{tikzcd}
	{H_{s}\left(\widehat{\mathcal{F}}_{\lambda}\right)} && {H_{s}\left(\widehat{\mathcal{F}}_{\lambda'}\right)} \\
	{H_{s}\left(\mathcal{K}_{\lambda+\sqrt{2}\sigma N_{h}h^{\alpha},\sqrt{d}h}\right)} && {H_{s}\left(\mathcal{K}_{\lambda'+\sqrt{2}\sigma N_{h}h^{\alpha},\sqrt{d}h}\right)} \\
	{H_{s}\left(\mathcal{G}_{\lambda+\sqrt{2}\sigma N_{h}h^{\alpha},\sqrt{d}h}\right)} && {H_{s}\left(\mathcal{G}_{\lambda'+\sqrt{2}\sigma N_{h}h^{\alpha},\sqrt{d}h}\right)} \\
	{H_{s}\left(\mathcal{F}_{\lambda+K_{1}h^{\alpha}}\right)} && {H_{s}\left(\mathcal{F}_{\lambda'+K_{1}h^{\alpha}}\right)}	\arrow["\widehat{v}_{\lambda}^{\lambda'}",from=1-1, to=1-3,]
	\arrow[from=1-1, to=2-1]
	\arrow["{j_{2,\lambda}}", tail reversed, from=2-1, to=3-1]
	\arrow[from=2-1, to=2-3]
	\arrow["{j_{2,\lambda'}}", tail reversed, from=2-3, to=3-3]
	\arrow[from=3-1, to=4-1]
	\arrow[from=3-1, to=3-3]
	\arrow[from=3-3, to=4-3]
    \arrow[from=1-3, to=2-3]
        \arrow[from=2-1, to=3-1]
        \arrow[from=2-3, to=3-3]
        \arrow["\widehat{v}_{\lambda+K_{1}h^{\alpha}}^{\lambda'+K_{1}h^{\alpha}}",from=4-1, to=4-3]
\end{tikzcd}
 \end{equation*}
 As $j_{2,\lambda}$ and $j_{2,\lambda'}$ arise from deformation retractions and other maps are simply induced by inclusion, all faces of Diagram \ref{diagram2} commute. Consequently, Diagram \ref{diagram2} commutes.
\item \textbf{Diagram \ref{diagram1}:} it can be decomposed similarly to Diagram \ref{diagram2}. One can check that the same reasoning applies.
\item \textbf{Diagram \ref{diagram3}:} Here, the proof is slightly more technical and requires more subtle geometric arguments, in particular involving Lemma \ref{lemma-histo 2}. Let $C\in C_{s}\left(\widehat{\mathcal{F}}_{\lambda}\right)$ and $[C]$ be its classes in $H_{s}\left(\widehat{\mathcal{F}}_{\lambda}\right)$. To prove the commutativity of Diagram \ref{diagram3}, it suffices to show that:
\begin{equation}
\label{eq: diag10-0}
\psi_{\lambda+K_1 h_{\alpha}}\left(\phi_{\lambda}([C])\right)=[C]\in H_{s}\left(\widehat{\mathcal{F}}_{\lambda+(K_1+K_2)h^{\alpha}}\right). 
\end{equation}
First, observe that $\phi_{\lambda}$ maps $[C]$ to $[C']$ with,
$$C'=F^{\#}_{\lambda+\sqrt{2}\sigma N_{h}h^{\alpha},\sqrt{d}h}(C,1).$$ 
We begin by proving that:
\begin{equation}
\label{eq: diag10-3}
[C^{'}]=[C]\in H_{s}\left(\widehat{\mathcal{F}}_{\lambda+(K_1+K_2)h^{\alpha}}\right).
\end{equation}
Denote $\overline{F}_{\lambda+\sqrt{2}\sigma N_{h}h^{\alpha},\sqrt{d}h}$ the restriction of $F_{\lambda+\sqrt{2}\sigma N_{h}h^{\alpha},\sqrt{d}h}$ to:
$$F_{\lambda+\sqrt{2}\sigma N_{h}h^{\alpha},\sqrt{d}h}( \widehat{\mathcal{F}}_{\lambda},[0,1]).$$ It is a deformation retraction from $F_{\lambda+\sqrt{2}\sigma N_{h}h^{\alpha},\sqrt{d}h}( \widehat{\mathcal{F}}_{\lambda},[0,1])$ to $F_{\lambda+\sqrt{2}\sigma N_{h}h^{\alpha},\sqrt{d}h}( \widehat{\mathcal{F}}_{\lambda},1)$. Thus, by homology invariance under deformation retraction,
\begin{align*}
[C]&=\left[\overline{F}^{\#}_{\lambda+\sqrt{2}\sigma N_{h}h^{\alpha},\sqrt{d}h}(C,1)\right]\\
&=\left[F^{\#}_{\lambda+\sqrt{2}\sigma N_{h}h^{\alpha},\sqrt{d}h}(C,1)\right]\\
&=[C']\in H_{s}\left(F_{\lambda+\sqrt{2}\sigma N_{h}h^{\alpha},\sqrt{d}h}( \widehat{\mathcal{F}}_{\lambda},[0,1])\right).
\end{align*}
Assertion (\ref{lemma-histo 2a}) of Lemma \ref{lemma-histo 2} ensures that 
$$F_{\lambda+\sqrt{2}\sigma N_{h}h^{\alpha},\sqrt{d}h}( \widehat{\mathcal{F}}_{\lambda},[0,1])\subset \widehat{\mathcal{F}}_{\lambda+K_{1}h^{\alpha}}$$
and thus,
\begin{equation*}
[C]=[C']\in H_{s}\left(\widehat{\mathcal{F}}_{\lambda+K_{1}h^{\alpha}}\right).
\end{equation*}
As $\widehat{\mathcal{F}}_{\lambda+K_{1}h^{\alpha}}\subset \widehat{\mathcal{F}}_{\lambda+(K_{1}+K_2)h^{\alpha}}$, this proves \eqref{eq: diag10-3}. Next, as $C^{'}\in C_{s}(\widehat{\mathcal{F}}_{\lambda+K_{1}h^{\alpha}}\cap\mathcal{F}_{\lambda+K_{1}h^{\alpha}})$, $\psi_{\lambda+K_{1}h^{\alpha}}$ maps $[C']$ to $[C^{''}]$, with,
$$C^{''}=H^{\#}_{\lambda+K_{1}h^{\alpha},\sqrt{d}h}(C',1).$$
Let us prove that: 
\begin{equation}
\label{eq: diag10-4}
[C^{'}]=[C^{''}]\in H_{s}\left(\widehat{\mathcal{F}}_{\lambda+(K_1+K_2)h^{\alpha}}\right).
\end{equation}
Denote $\overline{H}_{\lambda+K_{1}h^{\alpha},\sqrt{d}h}$ the restriction of $H_{\lambda+K_{1}h^{\alpha},\sqrt{d}h}$ to:
$$H_{\lambda+K_{1}h^{\alpha},\sqrt{d}h}\left(\widehat{\mathcal{F}}_{\lambda+K_{1}h^{\alpha}}\cap\mathcal{F}_{\lambda+K_{1}h^{\alpha}},[0,1]\right).$$ It is a deformation retraction and thus, again, by homology invariance under deformation retract,
\begin{align*}
    [C']&=\left[\overline{H}^{\#}_{\lambda+K_{1}h^{\alpha},\sqrt{d}h}(C',1)\right]\\
    &=\left[H^{\#}_{\lambda+K_{1}h^{\alpha},\sqrt{d}h}(C',1)\right]\\
    &=[C^{''}]\in H_{s}\left(H_{\lambda+K_{1}h^{\alpha},\sqrt{d}h}\left(\widehat{\mathcal{F}}_{\lambda+K_{1}h^{\alpha}}\cap\mathcal{F}_{\lambda+K_{1}h^{\alpha}},[0,1]\right)\right)
\end{align*}
Assertion (\ref{lemma-histo 2b}) of Lemma \ref{lemma-histo 2} then ensures that:
$$H_{\lambda+K_{1}h^{\alpha},\sqrt{d}h}\left(\widehat{\mathcal{F}}_{\lambda+K_{1}h^{\alpha}}\cap\mathcal{F}_{\lambda+K_{1}h^{\alpha}},[0,1]\right)\subset\widehat{\mathcal{F}}_{\lambda+(K_{1}+K_{2})h^{\alpha}}$$
and consequently proves \eqref{eq: diag10-4}. Combining \eqref{eq: diag10-4} and \eqref{eq: diag10-3} we have \eqref{eq: diag10-0} and therefore Diagram \ref{diagram3} commutes.
\item \textbf{Diagram \ref{diagram4}:} Let $C\in C_{s}\left(\mathcal{F}_{\lambda}\right)$, $\psi_{\lambda}$ maps $[C]$ to $[C']$, with,
$$C'=H^{\#}_{\lambda,\sqrt{d}h}(C,1).$$
As $\mathcal{M}_{\lambda,\sqrt{d}h}\subset \mathcal{G}_{\lambda+K_{2}h^{\alpha},\sqrt{d}h}$, the linear map $\phi_{\lambda+K_{2}h^{\alpha}}$ behaves as an inclusion induced map, mapping $[C']$ to $[C']$. 
From Lemma \ref{lemmaFiltEquiv2}, we have:
$$H_{\lambda,\sqrt{d}h}(\mathcal{F}_{\lambda},[0,1])\subset \mathcal{N}_{\lambda,\sqrt{d}h}\subset \mathcal{F}_{\lambda+(K_{1}+K_{2})h^{\alpha}}.$$
Thus,
$$[C]=[C']\in H_{s}\left(\mathcal{F}_{\lambda+(K_{1}+K_{2})h^{\alpha}}\right)$$
and Diagram \ref{diagram4} commutes.
\end{itemize}
The commutativity of diagrams \ref{diagram1},\ref{diagram2},\ref{diagram3} and \ref{diagram4} means that $\widehat{\mathbb{V}}_{f,s}$ and $\mathbb{V}_{f,s}$ are $(K_{1}+K_{2})h^{\alpha}$ interleaved, and thus we get from the algebraic stability theorem \citep{Chazal2009} that
\begin{equation*}
d_{b}\left(\operatorname{dgm}\left(\widehat{\mathbb{V}}_{f,s}\right),\operatorname{dgm}\left(\mathbb{V}_{f,s}\right)\right)\leq (K_{1}+K_{2})h^{\alpha} 
\end{equation*}
and as it holds for all $s\in\mathbb{N}$,
$$d_{b}\left(\widehat{\operatorname{dgm}(f)},\operatorname{dgm}(f)\right)\leq(K_{1}+K_{2})h^{\alpha}.$$
Now, using Lemma \ref{lemma noise 2}, this implies that, for all $f\in S_{d}(M,L,\alpha,R)$,
\begin{align*}
&\quad\mathbb{P}\left(d_{b}\left(\widehat{\operatorname{dgm}(f)},\operatorname{dgm}(f)\right)\geq th^{\alpha}\right)\\
&\leq \mathbb{P}\left(K_{1}+K_{2}\geq t\right)\\
&=\mathbb{P}\left(N_{h}\geq \frac{t-k_{1}-k_{2}}{2\sqrt{2}\sigma(1+c_{1})}\right)\\
&\leq C_{0}\exp\left(-C_{1}\left(\frac{t-k_{1}-k_{2}}{2\sqrt{2}\sigma (1+c_{1})}\right)^{2}\right)\\
&\leq C_{0}\exp\left(2\frac{C_{1}(k_{1}+k_{2})}{\left(2\sqrt{2}\sigma (1+c_{1})\right)^{2}}t\right)\exp\left(-C_{1}\left(\frac{k_{1}+k_{2}}{2\sqrt{2}\sigma (1+c_{1})}\right)^{2}\right)\exp\left(-\frac{C_{1}}{\left(2\sqrt{2}\sigma (1+c_{1})\right)^{2}}t^{2}\right)\\
&=C_{0}\exp\left(-C_{1}\left(\frac{k_{1}+k_{2}}{(2\sqrt{2}\sigma (1+c_{1}))}\right)^{2}\right)\exp\left(2\frac{C_{1}(k_{1}+k_{2})}{\left(2\sqrt{2}\sigma (1+c_{1})\right)^{2}}t-\frac{C_{1}}{2\left(2\sqrt{2}\sigma (1+c_{1})\right)^{2}}t^{2}\right)\\
&\quad\times \exp\left(-\frac{C_{1}}{2\left(2\sqrt{2}\sigma (1+c_{1})\right)^{2}}t^{2}\right)
\end{align*}
Now, as for all $t\in\mathbb{R}$,
$$2\frac{C_{1}(k_{1}+k_{2})}{\left(2\sqrt{2}\sigma (1+c_{1})\right)^{2}}t-\frac{C_{1}}{2\left(2\sqrt{2}\sigma (1+c_{1})\right)^{2}}t^{2}\leq \frac{2C_{1}(k_1+k_2)^{2}}{\left(2\sqrt{2}\sigma (1+c_{1})\right)^{2}}$$
we have:
\begin{align*}
&\quad\mathbb{P}\left(d_{b}\left(\widehat{\operatorname{dgm}(f)},\operatorname{dgm}(f)\right)\geq th^{\alpha}\right)\\
&\leq C_{0}\exp\left(-C_{1}\left(\frac{k_{1}+k_{2}}{(2\sqrt{2}\sigma (1+c_{1}))}\right)^{2}+\frac{2C_{1}(k_1+k_2)^{2}}{\left(2\sqrt{2}\sigma (1+c_{1})\right)^{2}}\right)\exp\left(-\frac{C_{1}}{2\left(2\sqrt{2}\sigma (1+c_{1})\right)^{2}}t^{2}\right)\\
&= C_{0}\exp\left(\frac{C_{1}(k_1+k_2)^{2}}{\left(2\sqrt{2}\sigma (1+c_{1})\right)^{2}}\right)\exp\left(-\frac{C_{1}}{2\left(2\sqrt{2}\sigma (1+c_{1})\right)^{2}}t^{2}\right)
\end{align*}
and the result follows.
\end{proof}\vspace{0.25cm}
We can derive from this result bounds in expectation, proving Theorem \ref{estimation borne sup}.
\begin{thm}
\label{estimation borne sup}
Let $h\asymp \left(\frac{\log(n)}{n}\right)^{\frac{1}{d+2\alpha}}$, 
$$\underset{f\in S_{d}(M,L,\alpha,R)}{\sup}\quad\mathbb{E}\left(d_{b}\left(\widehat{\operatorname{dgm}(f)},\operatorname{dgm}(f)\right)\right)\lesssim \left(\frac{\log(n)}{n}\right)^{\frac{\alpha}{d+2\alpha}}$$
\end{thm}
\begin{proof}
 The sub-Gaussian concentration provided by Theorem \ref{MainProp} ensures that, for all $t>0$,
$$\mathbb{P}\left(\frac{d_{b}\left(\widehat{\operatorname{dgm}(f)},\operatorname{dgm}(f)\right)}{h^{\alpha}}\geq t\right)\leq \Tilde{C}_{0}\exp\left(-\Tilde{C}_{1}t^{2}\right)$$
 Now, we have:
 \begin{align*}
 &\quad\mathbb{E}\left(\frac{d_{b}\left(\widehat{\operatorname{dgm}(f)},\operatorname{dgm}(f)\right)}{h^{\alpha}}\right)\\
&=\int_{0}^{+\infty}\mathbb{P}\left(\frac{d_{b}\left(\widehat{\operatorname{dgm}(f)},\operatorname{dgm}(f)\right)}{h^{\alpha}}\geq t\right)dt\\
     &\leq \int_{0}^{+\infty}\Tilde{C}_{0}\exp\left(-\Tilde{C}_{1}t^{2}\right)dt<+\infty.
\end{align*}

\end{proof}
\section{Lower bounds}
\label{section LB}
In this section, we prove that the rates obtained in the previous section are optimal by deriving lower bounds on the minimax risk. 
\begin{thm}
\label{lowerbound}
    $$\underset{\widehat{\operatorname{dgm}(f)}}{\inf}\quad\underset{f\in S_{d}(M,L,\alpha,R)}{\sup}\quad\mathbb{E}\left(d_{b}\left(\widehat{\operatorname{dgm}(f)},\operatorname{dgm}(f)\right)\right)\gtrsim \left(\frac{\log(n)}{n}\right)^{\frac{\alpha}{d+2\alpha}}.$$
Where the infimum is taken over all the estimators of $\operatorname{dgm}(f)$.
\end{thm}
The proof follows a standard method to provide minimax lower bounds, as presented in Section 2 of \cite{TsybakovBook}. The idea is, for any $r_{n}=o\left(\left(\frac{\log(n)}{n}\right)^{\frac{\alpha}{d+2\alpha}}\right)$, to exhibit a finite collection of functions in $S_{d}(M,L,\alpha,R)$ such that their persistence diagrams are pairwise at distance $2r_{n}$ but indistinguishable, with high certainty.\vspace{0.25cm}
\begin{proof}
Let 
$$f_{0}(x_{1},...,x_{d})=\frac{\min(M,L)}{2\sqrt{d}}|x_{1}|^{\alpha}$$
and 
for $m$ integer in $\left[0, \lfloor1/h\rfloor\right]$,
$$f_{h,m}(x_{1},...,x_{d})=f_{0}-\frac{\min(L,M)}{\sqrt{d}}\left(h^{\alpha}-||(x_{1},...,x_{d})-m/\lfloor1/h\rfloor(1,...,1)||^{\alpha}_{\infty}\right)_{+}$$
$f_{0}$ and the $f_{h,m}$ are $(L,\alpha)-$Hölder-continuous and bounded by $M$. Thus, they belong to $S_{d}(M,L,\alpha,R)$ for all $R>0$.\\\\
We have $\operatorname{dgm}(f_{0})=\{(0,+\infty)\}$ and for all $0<m<\lfloor1/h\rfloor$, integer, $\operatorname{dgm}(f_{h,m})=$
$$\left\{(0,+\infty),\left(\frac{\min(L,M)}{2\sqrt{d}}\left(\frac{m}{\lfloor1/h\rfloor}\right)^{\alpha}-\frac{L}{\sqrt{d}}h^{\alpha},\frac{\min(L,M)}{2\sqrt{d}}\left(\frac{m}{\lfloor1/h\rfloor}\right)^{\alpha}-\frac{\min(L,M)}{2\sqrt{d}}h^{\alpha}\right)\right\}.$$
Thus, for all $0<m\ne m'<\lfloor1/h\rfloor$, integers,
$$d_{b}\left(\operatorname{dgm}(f_{0}),\operatorname{dgm}(f_{m,h})\right)\geq \frac{\min(L,M)h^{\alpha}}{2\sqrt{d}}\text{ and }d_{b}\left(\operatorname{dgm}(f_{m,h}),\operatorname{dgm}(f_{m'})\right)\geq  \frac{\min(L,M)h^{\alpha}}{2\sqrt{d}}.$$
We set $r_{n}=\frac{\min(L,M)h^{\alpha}}{4\sqrt{d}}$, then,
$$d_{b}\left(\operatorname{dgm}(f_{0}),\operatorname{dgm}(f_{d,h,m^{k'},\alpha})\right)\geq 2r_{n}\text{ and }d_{b}\left(\operatorname{dgm}(f_{h,m}),\operatorname{dgm}(f_{h,m'})\right)\geq 2r_{n}.$$
For a fixed signal $f$, denote $\mathbb{P}^{n}_{f}=\bigotimes^{n}\mathbb{P}_{f}$ the joint distribution of the observations $(X_1,...,X_n)$. From Theorem 2.5 of \cite{TsybakovBook}, it now suffices to show that if $r_{n}=o\left(\left(\sfrac{\log(n)}{n}\right)^{\frac{\alpha}{d+2\alpha}}\right)$, then,
\begin{equation}
\label{chi2}
  \frac{1}{\left(\left\lfloor \frac{1}{h}\right\rfloor-2\right)\log\left(\left\lfloor \frac{1}{h}\right\rfloor-2\right)}\sum\limits_{0<m<\lfloor1/h\rfloor}KL\left(\mathbb{P}^{n}_{f_{h,m}},\mathbb{P}^{n}_{f_{0}}\right)
\end{equation}
converges to zero when $n$ converges to infinity, where $KL$ is the Kullback-Leibler divergence between probability distributions. We denote $H_{m}$ the hypercube defined by $||(x_{1},...,x_{d})-m/\lfloor1/h\rfloor(1,...,1)||_{\infty}\leq h$
\begin{align}
\label{CameronMartin}
&KL\left(\mathbb{P}^{n}_{f_{h,m}},\mathbb{P}^{n}_{f_{0}}\right)\nonumber\\
&=\sum_{i=1}^{n}KL\left(\mathcal{N}(f_{h,m}(x_{i}),\sigma),\mathcal{N}\left(f_{0}(x_{i}),\sigma\right)\right)\nonumber\\
&=\sum_{i=1}^{n}\frac{(f_{h,m}(x_{i})-f_{0}(x_{i}))^{2}}{2\sigma^2}\nonumber\\
&=\sum_{x_i\in H_m}\frac{(f_{h,m}(x_{i})-f_{0}(x_{i}))^{2}}{2\sigma^2}\nonumber\\
&=\frac{\min(M,L)^2}{4d\sigma^2}\sum_{x_i\in H_m}\left(h^{\alpha}-||x_i-m/\lfloor1/h\rfloor(1,...,1)||^{\alpha}_{\infty}\right)_{+}^{2}\nonumber\\
&\leq \frac{\min(M,L)^2}{4d\sigma^2}|H_m\cap G_n|h^{2\alpha}\nonumber\\
&\lesssim nh^{2\alpha+d}\nonumber
\end{align}
 Hence, if $h=o\left(\left(\sfrac{\log(n)}{n}\right)^{\frac{1}{d+2\alpha}}\right)$, we have that (\ref{chi2}) converges to zero. Consequently, if $r_{n}=o\left(\left(\sfrac{\log(n)}{n}\right)^{\frac{\alpha}{d+2\alpha}}\right)$, then $h=o\left(\left(\sfrac{\log(n)}{n}\right)^{\frac{1}{d+2\alpha}}\right)$ and we get the conclusion. 
\end{proof}\vspace{0.25cm}

Following the remark made in the Introduction, it is worth noting that the previous proof establishes a stronger result. Throughout the proof, we have only considered \((L,\alpha)\)-Hölder continuous functions, implying that the lower bound we obtain also applies to the minimax risk over Hölder spaces. As a direct consequence, this formally confirms that the rates obtained in \cite{BCL2009} and \cite{Chung2009} are minimax.
\section{Numerical illustrations}
\label{sec: illustration}
The present work focuses mainly on addressing the statistical difficulty of estimating persistence diagrams for noisy signals. However, for practical applicability, the computational aspects of the proposed strategy should also be discussed. Although it is easy to implement, it suffers from computational limitations. Since the proposed estimator constructs persistence diagrams using cubical homology, for a window size \( h \), it requires, in the worst case, \( O(h^{-d}) \) memory to store the cubical complex associated with the sublevel filtration, and the computation of persistence diagrams runs in \( O(h^{-3d}) \). Following Theorem \ref{MainProp}, the \( h \) that leads to the minimax convergence rate is of order \( O((\ln(n)/n)^{1/(2\alpha+d)}) \), which then induces a complexity of \( O(n^{d/(2\alpha+d)}) = O(n) \) in memory and \( O(n^{3d/(2\alpha+d)}) = O(n^3) \) in computation, which can be prohibitive in many practical scenarios. Note that this is not specific to our method. Any standard plug-in estimator suffers from similar shortcomings. This is likely the main limitation preventing the broader use of persistence diagrams in signal processing. However, efficient computation of persistent homology for cubical complexes is an active area of research \citep[see e.g.,][]{Wagner2012,Jacquette15,Otter2017,DLOTKO2018,Kaji20}, and the proposed methods may benefit from future progress in this direction. Also note that in some cases, this complexity can be improved. For example, if one is only interested in 0th-order persistent homology (i.e., tracking the evolution of connected components), the complexity for computing the \( H_0 \)-persistence diagram reduces to
$O\bigl(h^{-d} \alpha(h^{-d})\bigr) = O\bigl(n \, \alpha(n)\bigr)$, where \( \alpha \) denotes the inverse of the Ackermann function \citep[see][Section VII.2]{Wagner2012}. Since \( \alpha \) grows extremely slowly, this implies that the algorithm runs essentially in linear time in this case.

Nonetheless, we believe that our strategy can still be valuable in certain practical scenarios, particularly in lower-dimensional settings. To illustrate this, we provide numerical examples in dimension two, situating ourselves in the common practical scenario of image analysis in the presence of additive noise. In this context, Assumptions \textbf{A0}-\textbf{A3} essentially allow us to consider images displaying objects with smooth boundaries, with no intersections between objects' boundaries or between objects and the boundary of the image. We consider the following two toy examples:
\begin{figure}[h]
\centering
\begin{subfigure}{.5\textwidth}
  \centering
  \includegraphics[width=1\linewidth]{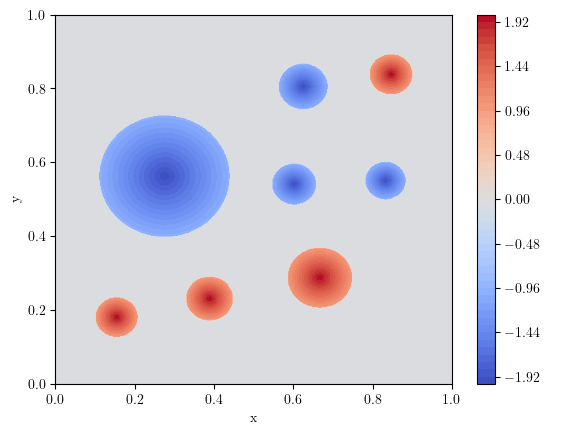}
  \caption{Graph of $f_{1}$}
  \label{fig:sub1}
\end{subfigure}%
\begin{subfigure}{.5\textwidth}
  \centering
  \includegraphics[width=1\linewidth]{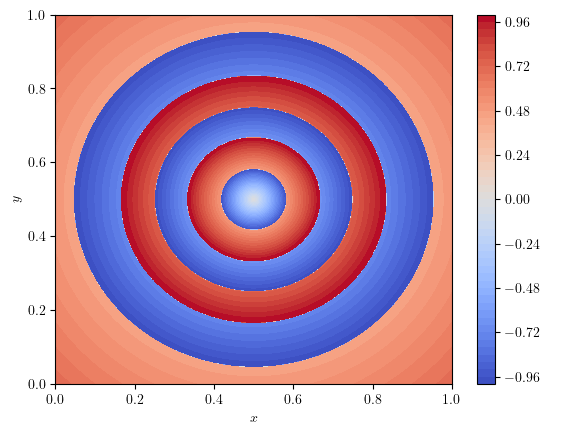}
  \caption{Graph of $g_1$}
  \label{fig:sub2}
\end{subfigure}
\caption{Graph of $f_{1}$ and $g_1$.}
\label{fig:test}
\end{figure}
\begin{itemize}
    \item First, we consider $k=8$ non-intersecting discs $D_1,...,D_k$, chosen randomly (random centers $(x_1,y_1),..(x_k,y_k)$ and random radii $r_1,...,r_k$) in $[0,1]^2$. We also ensure that $D_1,...,D_k$ do not intersect the boundary of $[0,1]^2$. We then consider the family of signal $f_\alpha:[0,1]^2\rightarrow\mathbb{R}$, $\alpha\in]0,1]$ defined by:
    $$f_{\alpha}(x)=\begin{cases}(-1)^{i}\left(2-\left(\frac{||x-x_i||_2}{r_i}\right)^{\alpha}\right)&\text{ if }x\in D_{i}, 1\leq i\leq k\\
    0&\text{ else}
    \end{cases}.$$
    These are simply sums of bump functions with exponent $\alpha$ on the discs $D_1,...,D_k$, shifted to create discontinuity at the boundaries of each $D_1,...,D_k$. Note that $f_\alpha$ belongs to $S_2(M,L,R,\alpha)$, for, $M=2$, $L=(1/\min_{1\leq i\leq k} r_i)^{\alpha}$ and some $R>0$ depending on $\min_{1\leq i\leq k} r_i$ and the minimal distance of the discs $D_1,...,D_k$ to the boundary of $[0,1]^2$. 
    \item Second, we choose $p=5$ discs $\tilde{D}_1\subset...\subset\tilde{D}_p$ centered at $(0,0)$ with raddii $1/2>\tilde{r}_5=1/2.2>\tilde{r}_4=1/3>\tilde{r}_3=1/4>\tilde{r}_2=1/6>\tilde{r}_1=1/12$ and consider the family of signal $g_{\alpha}:[0,1]^2\rightarrow\mathbb{R}$, $\alpha\in]0,1]$, defined by:
    $$g_{\alpha}(x)=\begin{cases}
    -\left(\frac{||x-x_i||_2}{\tilde{r}_1}\right)^{\alpha}&\text{ if }x\in \tilde{D}_{1}\\
    (-1)^{i}\left(\frac{||x-x_i||_2}{\tilde{r}_i}\right)^{\alpha}&\text{ if }x\in \tilde{D}_{i}\setminus \tilde{D}_{i-1}, 2\leq i\leq p\\
   ||x-x_i||_2^{\alpha} &\text{ else}
    \end{cases}.$$
In contrast to the previous example, where the objects in the image (represented by the disks \(D_1, \dots, D_k\)) are non-overlapping, now, we consider \(k\) overlapping objects (represented by the disks \(\tilde{D}_1, \dots, \tilde{D}_p\)) that are strictly contained within one another. This inclusion ensures that their boundaries do not intersect. As a result, it follows that \(g_{\alpha}\) belongs to \(S_2(M, L, R, \alpha)\), for $M=1$, \(L = \left(1/\min_{1 \leq i \leq k} \tilde{r}_i\right)^{\alpha}\), and some \(R > 0\) depending on \(\min(\tilde{r}_1, \tilde{r}_2 - \tilde{r}_1, \dots, \tilde{r}_p - \tilde{r}_{p-1})\) and the distance of the disks \(\tilde{D}_p\) to the boundary of \([0, 1]^2\).
\end{itemize}
For these two examples, we run simulations to illustrate that, although the estimation of the signals in the sup-norm is out of reach, the plug-in histogram estimators of their persistence diagrams are consistent. To achieve this, for varying resolutions of observations \(\{N \times N, N \in [10: 50: 510]\}\) (where $[N_1:q:N_2]$ denotes the set $\{N_1,N_1+q,N_1+2q,...,N_2\}$), we generate noisy versions of the images associated with \(f_{\alpha}\) and \(g_{\alpha}\) by adding independent Gaussian noise to each pixel, with a standard deviation of \(\sigma = 0.1\). We then estimate the signal via histograms. We choose the window size according to Theorem \ref{MainProp}, taking $h=\frac{1}{10}\times(\ln(n)/n)^{1/(2\alpha+d)}$, and compute the diagrams associated to its sublevel sets using the Python package \cite{gudhi:urm} via its function \textsc{CubicalComplex} \citep{gudhi:CubicalComplex}.\\\\
We repeat each simulation \(r = 100\) times, from which we compute the average errors for both the bottleneck distance on the persistence diagrams and the sup-norm loss on the signals for each resolution. The true persistence diagrams are provided explicitly, whereas the sup-norm error is approximated by looking at the maximal errors on a grid with significantly higher resolution of \(800 \times 800\). To examine the influence of regularity, we perform these simulations for several values of \(\alpha\): \(\alpha = 1\), \(\alpha = \frac{7}{8}\), \(\alpha = \frac{2}{3}\), and \(\alpha = \frac{1}{2}\).\\\\
The results are presented in Figure \ref{fig:simu f} for \(f_{\alpha}\) and Figure \ref{fig:simu g} for \(g_{\alpha}\). They clearly show that the estimation of the signal is not consistent. This is intuitive because, at the pixels overlapping on the boundaries of the disks, consistent estimation is infeasible due to the discontinuity. However, we observe that the persistence diagrams are well-estimated, with errors in the bottleneck distance relatively rapidly converging to 0. This illustrates the inefficiency of the sup-norm stability theorems in these contexts, highlighting the relevance of the analysis we conducted.

Additionally, we compute the time required for each persistence diagram computation and calculate the average time for each resolution and each \(\alpha\). The results are shown in Table \ref{table: simu time}. The observed computation times appear reasonable, even for the highest resolutions. For these simulations, we consider relatively high regularity parameters \(\alpha\), but we still observe that computation times grow rapidly as the regularity decreases. In cases where regularity is low and resolution is high, using the theoretically optimal window \(h\) may become impractical, leading to a tradeoff between statistical performance and computational cost.
\begin{figure}[h]
    \centering
    \includegraphics[scale=0.5]{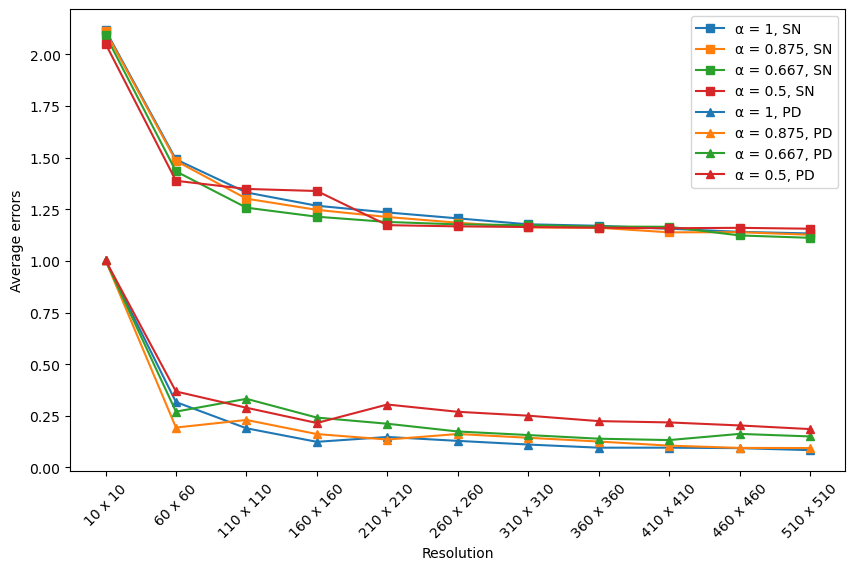}
    \caption{Simulation results for $f_{\alpha}$. The average bottleneck distance errors for diagram estimation (PD, marked with triangles) and average sup norm errors for signal estimation (SN, marked with squares) are displayed as functions of the observation resolution. The red curves correspond to $\alpha = 0.5$, the green to $\alpha = 2/3$, the orange to $\alpha = 7/8$, and the blue to $\alpha = 1$}
    \label{fig:simu f}
\end{figure}
\begin{figure}[h]
    \centering
    \includegraphics[scale=0.5]{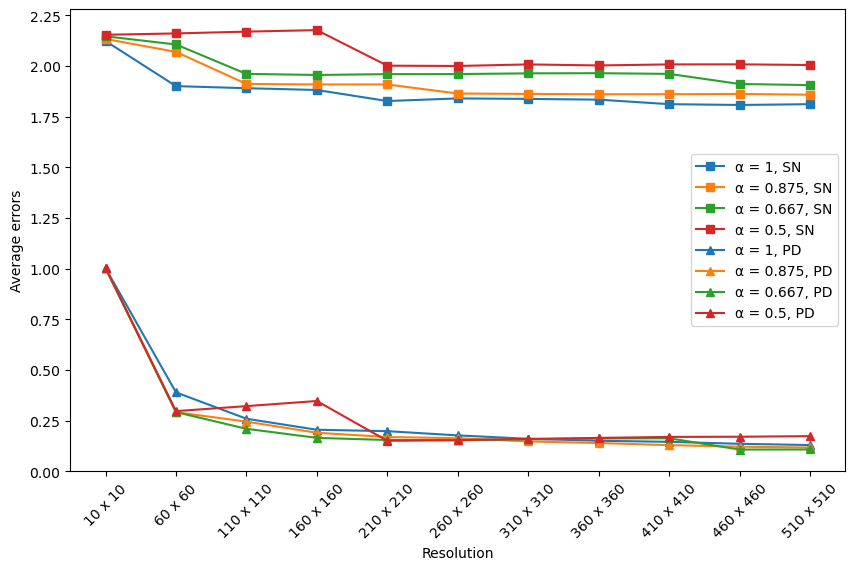}
    \caption{Simulation results for $g_{\alpha}$. The average bottleneck distance errors for diagram estimation (PD, marked with triangles) and average sup norm errors for signal estimation (SN, marked with squares) are displayed as functions of the observation resolution for $g_{\alpha}$. The red curves correspond to $\alpha = 0.5$, the green to $\alpha = 2/3$, the orange to $\alpha = 7/8$, and the blue to $\alpha = 1$.}
    \label{fig:simu g}
\end{figure}
\begin{table}[H]
\small
\begin{tabular}{cc|cccccc|}
\cline{3-8}
\multicolumn{1}{l}{}                                                              & \multicolumn{1}{l|}{}                  & \multicolumn{6}{c|}{\textbf{Resolution}}                                                                                                                                                                                                                                                                                                                                                 \\ \cline{3-8} 
\multicolumn{1}{l}{}                                                              & \multicolumn{1}{l|}{}                  & \multicolumn{1}{c|}{\textbf{$10\times10$}} & \multicolumn{1}{c|}{\textbf{$110\times110$}} & \multicolumn{1}{c|}{\textbf{$210\times210$}} & \multicolumn{1}{c|}{\textbf{$310\times310$}} & \multicolumn{1}{c|}{\textbf{$410\times410$}} & \textbf{$510\times510$} \\ \hline
\multicolumn{1}{|c|}{}                                    & \textbf{$1$}   & \multicolumn{1}{c|}{$9,62\times10^{-4}$}                           & \multicolumn{1}{c|}{$7,40\times10^{-2}$}                             & \multicolumn{1}{c|}{$2,06\times10^{-1}$}                             & \multicolumn{1}{c|}{$4,57\times10^{-1}$}                             & \multicolumn{1}{c|}{$8,18\times10^{-1}$}                             & $1,25$                                          \\ \cline{2-8} 
\multicolumn{1}{|c|}{}                                    & \textbf{$7/8$} & \multicolumn{1}{c|}{$1,06\times10^{-3}$}                           & \multicolumn{1}{c|}{$7,11\times10^{-2}$}                             & \multicolumn{1}{c|}{$2,69\times10^{-1}$}                             & \multicolumn{1}{c|}{$4,46\times10^{-1}$}                             & \multicolumn{1}{c|}{$9,88\times10^{-1}$}                             & $1,16$                                          \\ \cline{2-8} 
\multicolumn{1}{|c|}{}                                    & \textbf{$2/3$} & \multicolumn{1}{c|}{$9,71\times10^{-4}$}                           & \multicolumn{1}{c|}{$7,25\times10^{-2}$}                             & \multicolumn{1}{c|}{$2,76\times10^{-1}$}                             & \multicolumn{1}{c|}{$6,32\times10^{-1}$}                             & \multicolumn{1}{c|}{$1,47$}                                          & $1,71$                                          \\ \cline{2-8} 
\multicolumn{1}{|c|}{\multirow{-4}{*}{\textbf{$\alpha$}}} & \textbf{$1/2$} & \multicolumn{1}{c|}{$1,24\times 10^{-4}$}                          & \multicolumn{1}{c|}{$1,90\times10^{-1}$}                             & \multicolumn{1}{c|}{$2,90\times10^{-1}$}                             & \multicolumn{1}{c|}{$6,51\times10^{-1}$}                             & \multicolumn{1}{c|}{$1,49$}                                          & $2,44$                                          \\ \hline
\end{tabular}
\caption{Average computation time (in seconds) for computing the persistence diagram of $f_\alpha$ over $r = 100$ simulations, for varying resolutions and values of $\alpha$.}
\label{table: simu time}
\end{table}
\section{Discussion}
To date, the statistical analysis of persistence diagrams has largely depended on transferring established results from signal or density estimation, leveraging sup-norm stability. In contrast, the present work marks a departure from this paradigm. We introduce a finer statistical analysis of the plug-in histogram estimator, demonstrating that it achieves the minimax convergence rates over the classes \( S_d(M, L, \alpha, R) \), matching the known rates for Hölder-continuous signals. These functional classes encompass irregular signals that typically pose significant challenges for classical nonparametric methods. It is also worth noting that the statistical arguments used in this work are limited to a simple concentration bound (Lemma~\ref{lemma noise 2}), and can therefore be easily adapted to other standard models, such as the density estimation model or the Gaussian white noise model.

Beyond the theoretical results established in this work, our approach opens up a broader perspective on the inference of persistent homology. In particular, it reveals that the regularity assumptions traditionally imposed on the underlying signals can be meaningfully relaxed. This motivates further exploration of potential relaxations, especially of Assumption \textbf{A3}.

One promising direction involves controlling the \(\mu\)-reach, as introduced in \cite{mureach}, of the discontinuity set. Such an extension would significantly broaden the applicability of our results, enabling the treatment of signals with more complex discontinuity structures, such as sets containing multiple points (e.g., self-intersections) or corners. In the context of image analysis, as discussed in Section~\ref{sec: illustration}, this would allow us to consider objects with intersecting boundaries, boundaries intersecting the image frame, and objects with non-smooth boundaries.

However, as illustrated in Figure~\ref{mu reach problem}, the plug-in histogram estimator becomes inadequate in these more complex settings. Histograms may fail to capture accurately the topology of sublevel sets or may create false topological features due to multiple points and intersections with the boundary of the cube. We believe this limitation is not unique to histograms but likely extends to other standard plug-in estimators. Addressing these challenges may require moving beyond the plug-in framework, calling for the development of new inference techniques.
\begin{figure}
    \centering
    \subfloat[\centering A true cycle not captured by cubical approximation]{{\includegraphics[width=5.5cm]{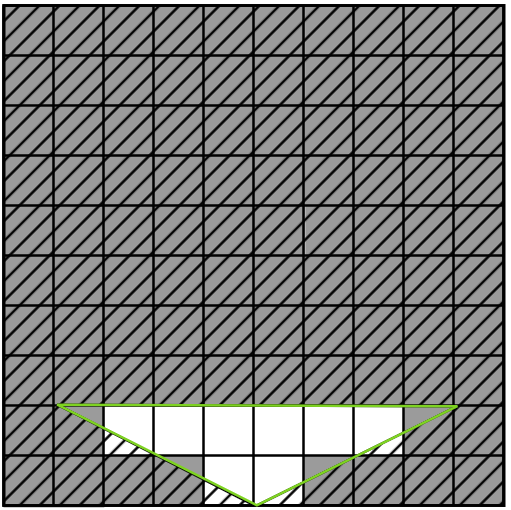} }}%
    \qquad
    \subfloat[\centering A false cycle created by cubical approximation]{{\includegraphics[width=5.5cm]{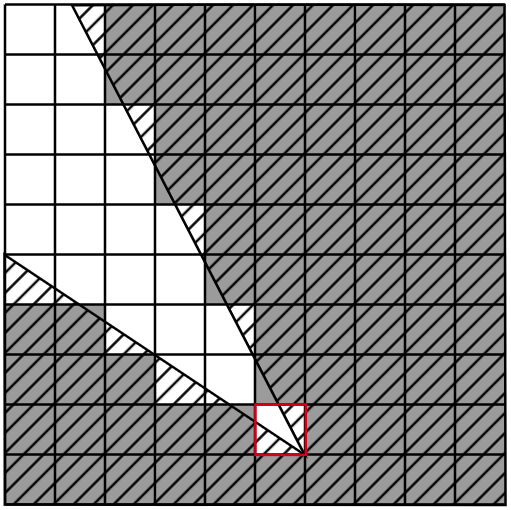} }}%
    \caption{$\lambda-$sublevel cubical approximation for $f$ the function defined as $0$ on the hatched area and $K$ outside (for arbitrarily large $K$). (a) displays a case where the histogram approximation fails to capture the true cycle in green, for at least all $0<\lambda<K/2$. (b) displays a case where the histogram approximation creates a cycle in red, not corresponding to any true cycle, with an arbitrarily long lifetime.  }%
    \label{mu reach problem}%
\end{figure}
\section*{Acknowledgements}
The author thanks anonymous reviewers for their helpful comments. The author also thanks Frédéric Chazal and Pascal Massart for our (many) helpful discussions. The author acknowledges the support of the ANR TopAI chair (ANR–19–CHIA–0001). 
\newpage
\appendix
\section{Proofs for $q-$tameness}
\label{q-tame appendix}
This section is devoted to proving the claim that the persistence diagrams we considered in this work are well-defined. Following the construction of \cite{chazal2013}, it suffices to prove that the associated persistence modules are $q-$tame.
\begin{defi}
\label{def: q-tame}
A persistence module $\mathbb{V}$ is said to be $q$-tame if for all $\lambda<\lambda^{\prime} \in \Lambda, \operatorname{rank}\left(v_\lambda^{\lambda^{\prime}}\right)$ is finite.
\end{defi}
\begin{lmm}
\label{verifyC2b}
    Let $f\in S_{d}(M,L,\alpha,R)$. For all $ s\in\mathbb{N}$ and $h< \frac{R}{2}$, for all $\lambda\in\mathbb{R}$, there exists a morphism $\phi_{\lambda}$ such that
    \begin{equation}
    \label{diagC2}
        \begin{tikzcd}
	{H_{s}\left(\mathcal{F}_{\lambda}\right)} && {H_{s}\left(\mathcal{F}_{\lambda+L(1+3^{\alpha})h^{\alpha}}\right)} \\
	& {H_{s}\left(\mathcal{F}_{\lambda}^{h}\right)}
	\arrow[from=1-1, to=1-3]
	\arrow[from=1-1, to=2-2]
	\arrow["{\phi_{\lambda}}"', from=2-2, to=1-3]
    \end{tikzcd}
    \end{equation}
    is a commutative diagram (unspecified maps come from set inclusions).
\end{lmm}
\begin{proof}
    Let $\Tilde{\phi}_{\lambda}:H_{s}\left(\mathcal{K}_{\lambda,h}\right)\rightarrow H_{s}\left(\mathcal{G}_{\lambda,h}\right)$ be the linear map induced by the deformation retraction from Lemma \ref{lemmaFiltEquiv}. We also denote $i_{1,\lambda}:H_{s}\left(\mathcal{F}_{\lambda}^{h}\right)\rightarrow H_{s}\left(\mathcal{K}_{\lambda,h}\right)$ the linear map induced by the inclusion $\mathcal{F}_{\lambda}^{h}\subset \mathcal{K}_{\lambda,h}$ and $i_{2,\lambda}:H_{s}\left(\mathcal{G}_{\lambda,h}\right)\rightarrow H_{s}\left(\mathcal{F}_{\lambda+L(1+3^{\alpha})h^{\alpha}}\right)$ the linear map induced by the inclusion $\mathcal{G}_{\lambda,h}\subset \mathcal{F}_{\lambda+L(1+3^{\alpha})h^{\alpha}}$, also provided by Lemma \ref{lemmaFiltEquiv}. We take $\phi_{\lambda}=i_{2,\lambda}\circ\Tilde{\phi}_{\lambda}\circ i_{1,\lambda}$. Diagram \ref{diagC2} then is (unspecified maps are induced by set inclusion),
    \begin{equation}
     \label{diagdecomp}
        \begin{tikzcd}
	{H_{s}\left(\mathcal{F}_{\lambda}\right)} &&& {H_{s}\left(\mathcal{F}_{\lambda+L(1+3^{\alpha})h^{\alpha}}\right)} \\
	{\quad\quad\quad (F1)} & {(F2)} & {\quad(F3)} \\
	{H_{s}\left(\mathcal{F}_{\lambda}^{h}\right)} & {H_{s}\left(\mathcal{K}_{\lambda,h}\right)} && {H_{s}\left(\mathcal{G}_{\lambda,h}\right)}
	\arrow[from=1-1, to=1-4]
	\arrow[from=1-1, to=3-1]
	\arrow["{i_{1,\lambda}}"', from=3-1, to=3-2]
	\arrow["{i_{2,\lambda}}"', from=3-4, to=1-4]
	\arrow[ from=3-2, to=3-4]
	\arrow["{\tilde{\phi}_\lambda}", from=3-4, to=3-2]
	\arrow[from=1-1, to=3-2]
	\arrow[from=1-1, to=3-4]
\end{tikzcd}
\end{equation}
Faces $(F1)$ and $(F3)$ simply commutes by inclusion. Face $(F2)$ commutes as $\Tilde{\phi}_{\lambda}$ is induced by a deformation retraction. All faces of Diagram \ref{diagdecomp} commute, hence Diagram \ref{diagdecomp} (and equivalently Diagram \ref{diagC2}) is commutative.
\end{proof}
\begin{prp}
\label{LmmQtame1}
Let $f\in S_{d}(M,L,\alpha,R)$ then $f$ is $q$-tame.
\end{prp}
\begin{proof}
 Let $s\in\mathbb{N}$ and $\mathbb{V}_{s,f}$ be the persistence module (for the $s-$th homology) associated to the sublevel sets filtration $\mathcal{F}$. For fixed levels $\lambda<\lambda'$ denote $v_{\lambda}^{\lambda'}$ the associated map. Let $\lambda\in\mathbb{R}$ and $h< \frac{R}{2}$. By Lemma \ref{verifyC2b},  $v_{\lambda}^{\lambda+L(1+3^{\alpha})h^{\alpha}}=\phi_{\lambda}\circ \Tilde{i}_{\lambda}$, with $\Tilde{i}_{\lambda}:H_{s}\left(\mathcal{F}_{\lambda}\right)\rightarrow H_{s}\left(\mathcal{F}_{\lambda}^{h}\right)$. By Assumptions \textbf{A1} and $\textbf{A2}$, $\mathcal{F}_{\lambda}$ is compact. As $[0,1]^{d}$ is triangulable, $\mathcal{F}_{\lambda}$ is covered by finitely many cells of the triangulation, there is a finite simplicial complex $K$ such that $\overline{\mathcal{F}_{\lambda}}\subset K \subset \mathcal{F}_{\lambda}^{h}$. Consequently, $\Tilde{i}_{\lambda}$ factors through the finite-dimensional space $H_{s}(K)$ and is then of finite rank. Thus, $v_{\lambda}^{\lambda+L(1+3^{\alpha})^{\alpha}h^{\alpha}}$ is of finite rank for all $0<h< \frac{R}{2}$. As for any $\lambda<\lambda'<\lambda^{''}$, $v_{\lambda}^{\lambda^{''}}=v_{\lambda'}^{\lambda^{''}}\circ v_{\lambda}^{\lambda'}$, we have that $v_{\lambda}^{\lambda'}$ is of finite rank for all $\lambda<\lambda'$. Hence, $f$ is $q$-tame. 
\end{proof}\vspace{0.25cm}
\begin{prp}
\label{qtamenessEstim}
Let $f\in S_{d}(M,L,\alpha,R)$ then, for all $s\in\mathbb{N}$, $\widehat{\mathbb{V}}^{h}_{s,f}$ is $q$-tame. 
\end{prp}
\begin{proof}
Let $h>0$ and $\lambda\in \mathbb{R}$. $\widehat{\mathcal{F}}_{\lambda}$ is a union of hypercubes of the regular grid $G_{h}$, thus, $H_{s}\left(\widehat{\mathcal{F}}_{\lambda}\right)$ is finite dimensional. Thus, $\widehat{\mathbb{V}}_{s,f}$ is $q$-tame by Theorem 1.1 of \cite{Crawley2012}. 
\end{proof}
\section{Proofs of technical lemmas}
\label{technical appendix}
\subsection{Proof of Lemma \ref{interleavish 3}}
\label{Proof interleavish}
This section is dedicated to the proof of Lemma \ref{interleavish 3}, from Section \ref{Upperbounds section}, which relies on the following lemma.
\begin{lmm}
\label{lmm1}
Let $f:[0,1]^{d}\rightarrow \mathbb{R}$ and $h$ satisfying (\ref{calibration}). Let $H\subset \mathcal{F}_{\lambda+\sqrt{2}\sigma N_{h}h^{\alpha}}^{c}\cap C_{h}$ and $H'\subset\mathcal{F}_{\lambda-\sqrt{2}\sigma N_{h}h^{\alpha}}\cap C_{h}$.
We have that:
$$\frac{1}{nh^{d}}\sum\limits_{x_i\in H}X_{i}> \lambda\text{ and }\frac{1}{nh^{d}}\sum\limits_{x_i\in H'}X_{i}< \lambda.$$
\end{lmm}
\begin{proof}
Let us consider here the case where in $H'\subset\mathcal{F}_{\lambda-\sqrt{2}\sigma N_{h}h^{\alpha}}$ (the proof being the same in both cases). We have: 
\begin{align*}
&\frac{1}{nh^{d}}\sum\limits_{x_i\in H'}X_{i}\\
&=\frac{1}{nh^{d}}\sum\limits_{x_i\in H'}f(x_{i})+\sigma\varepsilon_{i}\\
&\leq \lambda-\sqrt{2}\sigma N_{h}h^{\alpha} + N_{h}\sqrt{2}\sigma \sqrt{\frac{\log\left(1/h^{d}\right)}{nh^{d}}}\\
&< \lambda
\end{align*}
By the choice made for $h$. 
\end{proof}\vspace{0.25cm}
\begin{proof}[Proof of Lemma \ref{interleavish 3}]
Let $x\in \mathcal{F}_{\lambda-\sqrt{2}\sigma N_{h}h^{\alpha}}^{-\sqrt{d}h}$ and $H$ be the hypercube of $C_{h}$ containing $x$. We then have:
$$H\subset \mathcal{F}_{\lambda-\sqrt{2}\sigma N_{h}h^{\alpha}}.$$
Hence, by Lemma \ref{lmm1}, $\sum_{x_i\in H}X_i-nh^{d}\lambda<0$, thus, $$H\subset \widehat{\mathcal{F}}_{\lambda}.$$
Now, let $x\in\left(\mathcal{F}_{\lambda+\sqrt{2}\sigma N_{h}h^{\alpha}}^{\sqrt{d}h}\right)^{c}$, and $H$ be the hypercube of $C_{h}$ containing $x$. We then have:
$$H\subset \mathcal{F}_{\lambda+\sqrt{2}\sigma N_{h}h^{\alpha}}^{c}.$$
Hence, by Lemma \ref{lmm1}, $\sum_{x_i\in H}X_i-nh^{d}\lambda>0$, thus,
$$H\subset \widehat{\mathcal{F}}_{\lambda}^{c}$$
and Lemma \ref{interleavish 3} is proved. 
\end{proof}
\subsection{Proof of Lemma \ref{lemmaFiltEquiv}}
\label{proof Filtequiv}
This section is devoted to the proof of Lemma \ref{lemmaFiltEquiv} from Section \ref{Upperbounds section}.\vspace{0.25cm}
\begin{proof}
First, note that if $x$ belongs to $\bigcup_{x\in S_{\lambda,h}}\left[x,\xi(x)\right]$ then $x$ is at distance at most $h$ from the union of $(\partial M_{i})_{i=1,...,l}$, and thus $||x-\xi(x)||_{2}\leq h$ which proves that $\mathcal{K}_{\lambda,h}\subset \mathcal{F}_{\lambda}^{2h}$.  As $\mathcal{F}_{\lambda}\subset\left( \bigcup_{x\in S_{\lambda,h}}\left[x,\xi(x)\right]\right)^{c}\cap \mathcal{F}_{\lambda}^{h}$, $\mathcal{F}_{\lambda}\subset\mathcal{G}_{\lambda,h}$.\\\\
Now, we prove that $F_{\lambda,h}$ is a deformation retraction. By definition of $\mathcal{G}_{\lambda,h}$:
$$F_{\lambda,h}(x,1)\in \mathcal{G}_{\lambda,h},\quad \forall x\in \mathcal{K}_{\lambda,h}$$
and by definition of $F_{\lambda,h}$:
$$F_{\lambda,h}(x,0)=x,\quad \forall x\in \mathcal{K}_{\lambda,h}$$
Let $x\in \mathcal{K}_{\lambda,h}$ satisfying (\ref{cond 1}) and (\ref{cond 2}), we have $F_{\lambda,h}(x,t)\in[x,\xi(x)]$ for all $t\in[0,1]$. In particular, this implies that $\xi(F_{\lambda,h}(x,1))=\xi(x)$ thus, by definition of $F_{\lambda,h}$, $F_{\lambda,h}(F_{\lambda,h}(x,1),1)=F_{\lambda,h}(x,1)$. Otherwise, $F_{\lambda,h}(x,1)=x$. Hence,
$$F_{\lambda,h}(x,1)=x,\quad \forall x\in \mathcal{G}_{\lambda,h}.$$
The proof for the continuity of $F_{\lambda,h}$ is provided separately in Appendix \ref{proof continuity F}. Then $F_{\lambda,h}$ is a deformation retraction of $\mathcal{K}_{\lambda,h}$ onto $\mathcal{G}_{\lambda,h}$. \\\\
Now, we prove that $\mathcal{G}_{\lambda,h}\subset \mathcal{F}_{\lambda+(1+3^{\alpha})Lh^{\alpha}}$. Let $x\in\mathcal{K}_{\lambda,h}$, and suppose that $x\in\overline{M}_{i}\cap[0,1]^{d}$:
\begin{itemize}
    \item If $x$ does not satisfy (\ref{cond 1}), $F_{\lambda,h}(x,1)=x$, then, directly, by definition of $S_{\lambda,h}$, $x$ belongs $\mathcal{F}_{\lambda+Lh^{\alpha}}$.
    \item If $x$ satisfies (\ref{cond 1}) and $2h-d_{2}\left(\xi(x),M_{i}\cap \mathcal{F}_{\lambda+Lh^{\alpha}}\right)\geq 0$, as $F_{\lambda,h}(x,1)\in[x,\xi(x)]\subset \overline{M}_{i}$, we have
$$d_{2}\left(F_{\lambda,h}(x,1),M_{i}\cap \mathcal{F}_{\lambda+Lh^{\alpha}}\right)\leq 3h.$$
Thus, Assumptions \textbf{A1} and \textbf{A2} then ensure that (details concerning this claim are provided in Lemma \ref{lmm: A1-A2 imply}, Appendix \ref{sec: A1-A2 imply}),
$$F_{\lambda,h}(x,1)\in\mathcal{F}_{\lambda+L(1+3^{\alpha})h^{\alpha}}.$$
\item If $x$ satisfies (\ref{cond 1}) and $2h-d_{2}\left(\xi(x),M_{i}\cap \mathcal{F}_{\lambda+Lh^{\alpha}}\right)<0$, then, $F_{\lambda,h}(x,1)=\xi(x)$. Let $\varepsilon>0$, there exist $j\in\{1,...,l\}$, $i\ne j$ and $y\in \mathcal{F}_{\lambda}\cap M_{j}$, such that $||x-y||_{2}\leq h+\varepsilon$. Hence, $\xi(x)\in \partial M_{j}\cap]0,1[^{d}$ and $||\xi(x)-y||_{2}\leq 2h+\varepsilon$. Assumptions \textbf{A1} and \textbf{A2} then ensure that (see again Lemma \ref{lmm: A1-A2 imply} in Appendix \ref{sec: A1-A2 imply}),
$$\xi(x)\in \mathcal{F}_{\lambda+L(1+(2+\varepsilon)^{\alpha})h^{\alpha}}$$
as it holds for all $\varepsilon>0$,
$$F_{\lambda,h}(x,1)=\xi(x)\in \mathcal{F}_{\lambda+L(1+2^{\alpha})h^{\alpha}}.$$
\end{itemize}
Finally, combining cases, $\mathcal{G}_{\lambda,h}\subset\mathcal{F}_{\lambda+L(1+3^{\alpha})h^{\alpha}} $.
\end{proof}
\subsection{Proof of the continuity of $F_{\lambda,h}$}
\label{proof continuity F}
This section is devoted to the proof of $F_{\lambda,h}$, claimed in the proof of Lemma \ref{lemmaFiltEquiv}.
\begin{lmm}
\label{continuity F}
Let $R/2>h>0$ and $\lambda\in\mathbb{R}$, $F_{\lambda,h}$ is continuous.
\end{lmm}
\begin{proof}
Let $\delta,\delta'>0$, $x,y\in \mathcal{K}_{\lambda,h}$ such that $||x-y||_{2}\leq \delta$ and $t,s\in[0,1]$ such that $|t-s|\leq \delta'$. Let's check the different cases.\\\\
We begin with the cases where $x\in \overline{M}_{i}$ and $y\in\overline{M}_{j}$, $i\ne j$. Let $\delta<h$, then $||x-\xi(x)||_{2}\leq \delta$, $||y-\xi(y)||_{2}\leq \delta$, and thus $F_{\lambda,h}(x,t),F_{\lambda,h}(y,s)\in [x,\xi(x)]\cup[y,\xi(y)]\subset B_{2}(x,2\delta)$.
Then,
\begin{align*}
 ||F_{\lambda,h}(x,t)-F_{\lambda,h}(y,s)||_{2}\leq 4\delta
\end{align*}
And the conclusion follows in this case. From now on, we suppose that $x,y\in M_{i}$.
\begin{itemize}
\item If $x$ does not satisfy (\ref{cond 1}) or (\ref{cond 2}) and $y$ does not satisfy (\ref{cond 1}) or (\ref{cond 2}), then, directly,
$$||F_{\lambda,h}(x,t)-F_{\lambda,h}(y,s)||_{2}=||x-y||_{2}\leq \delta.$$
\item If $x$ satisfies (\ref{cond 1}) and (\ref{cond 2}), and $y$ does not satisfy (\ref{cond 1}), then, $y\in\mathcal{F}_{\lambda+Lh^{\alpha}}$. Thus, 
$$d_{2}\left(\xi(x),M_{i}\cap \mathcal{F}_{\lambda+Lh^{\alpha}}\right)\leq ||x-\xi(x)||+||x-y||_{2}$$
and thus, by (\ref{cond 2}),
$$||x-\xi(x)||\geq 2h-d_{2}\left(\xi(x),M_{i}\cap \mathcal{F}_{\lambda+Lh^{\alpha}}\right)\geq ||x-\xi(x)||-\delta.$$
Consequently,
$$\left\|\xi(x)+\left(2h-d_{2}\left(\xi(x),M_{i}\cap \mathcal{F}_{\lambda+Lh^{\alpha}}\right)\right)_{+}\frac{x-\xi(x)}{||x-\xi(x)||_{2}}-x\right\|_{2}\leq \delta$$
Then,
\begin{align*}
&||F_{\lambda,h}(x,t)-F_{\lambda,h}(y,s)||_{2}\\
&=\left\|(1-t)x+t\left(\xi(x)+\left(2h-d_{2}\left(\xi(x),M_{i}\cap \mathcal{F}_{\lambda+Lh^{\alpha}}\right)\right)_{+}\frac{x-\xi(x)}{||x-\xi(x)||_{2}}\right)-y\right\|_{2}\\
&\leq ||x-y||_{2}\quad+\left\|\xi(x)+\left(2h-d_{2}\left(\xi(x),M_{i}\cap \mathcal{F}_{\lambda+Lh^{\alpha}}\right)\right)_{+}\frac{x-\xi(x)}{||x-\xi(x)||_{2}}-x\right\|_{2}\\
&\leq 2\delta.
\end{align*}
\item If $x$ satisfies (\ref{cond 1}) and (\ref{cond 2}), and $y$ satisfies (\ref{cond 1}) but not (\ref{cond 2}). Then,
\begin{align*}
||x-\xi(x)||_{2}&\geq 2h-d_{2}\left(\xi(x),M_{i}\cap \mathcal{F}_{\lambda+Lh^{\alpha}}\right)\\
&=2h-d_{2}\left(\xi(y),M_{i}\cap \mathcal{F}_{\lambda+Lh^{\alpha}}\right)\\
&\quad+d_{2}\left(\xi(y),M_{i}\cap \mathcal{F}_{\lambda+Lh^{\alpha}}\right)-d_{2}\left(\xi(x),M_{i}\cap \mathcal{F}_{\lambda+Lh^{\alpha}}\right)\\
&\geq 2h-d_{2}\left(\xi(y),M_{i}\cap \mathcal{F}_{\lambda+Lh^{\alpha}}\right) -||\xi(x)-\xi(y)||_{2}\\
&\geq  ||y-\xi(y)||_{2}-||\xi(x)-\xi(y)||_{2}\\
&\geq ||x-\xi(x)||_{2}-2||\xi(x)-\xi(y)||_{2}-||x-y||_{2}.
\end{align*}
Hence,
$$\left\|\xi(x)+\left(2h-d_{2}\left(\xi(x),M_{i}\cap \mathcal{F}_{\lambda+Lh^{\alpha}}\right)\right)_{+}\frac{x-\xi(x)}{||x-\xi(x)||_{2}}-x\right\|_{2}\leq ||x-y||_{2}+2||\xi(x)-\xi(y)||_{2}$$
And thus, we have:
\begin{align*}
&||F_{\lambda,h}(x,t)-F_{\lambda,h}(y,s)||_{2}\\
&=\left\|(1-t)x+t\left(\xi(x)+\left(2h-d_{2}\left(\xi(x),M_{i}\cap \mathcal{F}_{\lambda+Lh^{\alpha}}\right)\right)_{+}\frac{x-\xi(x)}{||x-\xi(x)||_{2}}\right)-y\right\|_{2}\\
&\leq ||x-y||_{2}\quad+\left\|\xi(x)+\left(2h-d_{2}\left(\xi(x),M_{i}\cap \mathcal{F}_{\lambda+Lh^{\alpha}}\right)\right)_{+}\frac{x-\xi(x)}{||x-\xi(x)||_{2}}-x\right\|_{2}\\
&\leq 2||x-y||_{2}+2||\xi(x)-\xi(y)||_{2}\\
&\leq 2\delta +2||\xi(x)-\xi(y)||_{2}
\end{align*}
and we conclude by continuity of $\xi$.
\item Finally, if $x$ and $y$ both satisfy (\ref{cond 1}) and (\ref{cond 2}), then,
\begin{align*}
&||F_{\lambda,h}(x,t)-F_{\lambda,h}(y,s)||_{2}\\
&=\left\|(1-t)x+t\left(\xi(x)+\left(2h-d_{2}\left(\xi(x),M_{i}\cap \mathcal{F}_{\lambda+Lh^{\alpha}}\right)\right)_{+}\frac{x-\xi(x)}{||x-\xi(x)||_{2}}\right)\right.\\
&\left.\quad-(1-s)y-s\left(\xi(y)+\left(2h-d_{2}\left(\xi(y),M_{i}\cap \mathcal{F}_{\lambda+Lh^{\alpha}}\right)\right)_{+}\frac{y-\xi(y)}{||y-\xi(y)||_{2}}\right)\right\|_{2}
\end{align*}
and the conclusion follows again, in this case by continuity of $\xi$.
\end{itemize}
All possible cases have been checked, the proof is complete. 
\end{proof}
\subsection{Proof of Lemma \ref{lemmaFiltEquiv2}}
\label{proof Filtequiv 2}
This section is dedicated to proof of Lemma \ref{lemmaFiltEquiv2} from Section \ref{Upperbounds section}.\vspace{0.25cm}
\begin{proof}[Proof of Lemma \ref{lemmaFiltEquiv2}]
First, we check that $\gamma_{h}$ extends continuously to $P_{\lambda,h}$. Let $x\in P_{\lambda,h}\cap \bigcup_{i=1}^{l}\partial M_{i}\cap]0,1[^{d}$. Assumption~\textbf{A3} ensures that \( P_{\lambda,h} \subset \, ]0,1[^{d} \), and that there exist \( i, j \in \{1, \dots, l\} \) such that \(B_2(x,h) \subset M_i \cup M_j \cup (\partial M_i \cap \partial M_j)\) (details concerning this claim are provided in Appendix \ref{sec: detail A3}).
Now, if
$$B_{2}(x,h)\cap \mathcal{F}_{\lambda}\cap M_{i}\ne \emptyset \text{ and } B_{2}(x,h)\cap \mathcal{F}_{\lambda}\cap M_{j}\ne \emptyset$$
then by Assumptions \textbf{A2} and \textbf{A1}, $B_{2}(x,h)\subset\mathcal{F}_{\lambda+L2h^{\alpha}}$ (see Lemma \ref{lmm: A1-A2 imply} in Appendix \ref{sec: A1-A2 imply}) and thus $x\in\mathcal{F}_{\lambda+L2h^{\alpha}}^{-h}$. Hence, $B_{2}(x,h)\cap P_{\lambda,h}\cap M_{j}= \emptyset$ or $B_{2}(x,h)\cap P_{\lambda,h}\cap M_{i}= \emptyset$. Without loss of generality, suppose that $B_{2}(x,h)\cap P_{\lambda,h}\cap M_{j}= \emptyset$. Assumption \textbf{A3} imposes that $\bigcup_{i=1}^{l}\partial M_{i}\cap]0,1[^{d}$ is a $C^{1,1}$ hypersurface and thus ensures that, for all $x\in P_{\lambda,h}\cap\overline{M}_{i}\cap]0,1[^{d}$, $\lim_{y\in M_{i}\cap P_{\lambda,h}\rightarrow x}\gamma_{h}(y)$ exists. We can then define:
$$\gamma_{h,\lambda}(x)=\lim_{y\in M_{i}\cap P_{\lambda,h} \rightarrow x}\gamma_{h}(y).$$
Note that this intuitively follows from the fact that the $C^{1,1}$ assumption ensures that the normal cone of $\bigcup_{i=1}^{l}\partial M_{i}\cap]0,1[^{d}$ at $\xi(x)$, which contains the set of points $\{y\in P_{\lambda,h},\xi(y)=\xi(x)\}$ by Theorem 4.8 of \cite{Fed59}, is simply a line. Thus, there is only one point at distance $h$ from $\xi(x)$ in $\{y\in \overline{M}_{i}\cap P_{\lambda,h},\xi(y)=\xi(x)\}$, namely $\gamma_{h,\lambda}(x)$. Doing so for all $x\in P_{\lambda,h}\cap \bigcup_{i=1}^{l}\partial M_{i}\cap]0,1[^{d}$ extends continuously $\gamma_{h}(x)$ to $P_{\lambda,h}$.\\\\
Now, we prove that $H_{\lambda,h}$ is a deformation retraction. As $\mathcal{F}_{\lambda}^{-h}\subset \left(\bigcup_{x\in P_{\lambda,h}}\left[x,\gamma_{h,\lambda}(x)\right]\right)^{c}\cap \mathcal{F}_{\lambda} $, $\mathcal{F}_{\lambda}^{-h}\subset\mathcal{M}_{\lambda,h}$. Note that, by definition of $\mathcal{M}_{\lambda,h}$,
$$H_{\lambda,h}(x,1)\in \mathcal{M}_{\lambda,h},\quad\forall x\in \mathcal{N}_{\lambda,h}$$
and by definition of $H_{\lambda,h}$
$$H_{\lambda,h}(x,0)=x,\quad\forall x\in \mathcal{N}_{\lambda,h}$$
Let $x\in \mathcal{N}_{\lambda,h}$ satisfying (\ref{cond 3}) and (\ref{cond 4}). By construction $H_{\lambda,h}(x,t)\in[x,\gamma_{h,\lambda}(x)]$ for all $t\in[0,1]$. In particular, this implies that $\gamma_{\lambda,x}(H_{\lambda,h}(x,1))=\gamma_{h,\lambda}(x)$. Thus, $H_{\lambda,h}(H_{\lambda,h}(x,1),1)=H_{\lambda,h}(x,1)$. In other cases $H_{\lambda,h}(x,1)=x$. Hence,
$$H_{\lambda,h}(x,1)=x,\quad \forall x\in \mathcal{M}_{\lambda,h}.$$ 
The proof of the continuity of $H_{\lambda,h}$ is provided separately in Appendix \ref{proof continuity H}. Then $H_{\lambda,h}$ is a deformation retraction of $\mathcal{N}_{\lambda,h}$ onto $\mathcal{M}_{\lambda,h}$.\\\\
Now, we prove that $\mathcal{M}_{\lambda,h}\subset \mathcal{F}_{\lambda+L(2+5^{\alpha})h^{\alpha}}^{-h}$. Let $x\in\mathcal{N}_{\lambda,h}$, and suppose that $x\in\overline{M}_{i}\cap[0,1]^{d}$:
\begin{itemize}
    \item If $x$ does not satisfy (\ref{cond 3}), directly, $F(1,x)=x\in \mathcal{F}_{\lambda+2Lh^{\alpha}}^{-h}$.
    \item If $x$ satisfies (\ref{cond 3}) and $3h-d_{2}\left(\gamma_{h,\lambda}(x),(\overline{M}_{i})^{c}\cap \mathcal{F}_{\lambda+2Lh^{\alpha}}\right)\geq 0$, then there exists $z\in M_{j}\cap \mathcal{F}_{\lambda+2Lh^{\alpha}}$ with $j\ne i$ such that $||x-z||_{2}\leq 3h$ thus $||H_{\lambda,h}(x,1)-z||_{2}\leq 4h$. By Assumption \textbf{A3}, $B_{2}(H_{\lambda,h}(x,1),h)\subset\overline{M_{i}}\cup \overline{M}_{j}$ (see Lemma \ref{lmm: A3 imply} in Appendix \ref{sec: detail A3}) and by Assumptions \textbf{A1} and \textbf{A2} (see Lemma \ref{lmm: A1-A2 imply} in Appendix \ref{sec: A1-A2 imply}):
    $$B_{2}(H_{\lambda,h}(x,1),h)\cap\overline{M}_{j}\subset B_{2}(z,5h)\cap\overline{M}_{j}\subset\mathcal{F}_{\lambda+L(2+5^{\alpha})h^{\alpha}}$$
    and
    $$B_{2}(H_{\lambda,h}(x,1),h)\cap\overline{M}_{i}\subset B_{2}(x,2h)\subset\mathcal{F}_{\lambda+L3^{\alpha}h^{\alpha}}.$$
    Thus, it follows that:
$$H_{\lambda,h}(x,1)\in \mathcal{F}^{-h}_{\lambda+L(2+5^{\alpha})h^{\alpha}}.$$
\item If $x$ satisfies (\ref{cond 3}) and $3h-d_{2}\left(\gamma_{h,\lambda}(x),(\overline{M}_{i})^{c}\cap \mathcal{F}_{\lambda+2Lh^{\alpha}}\right)<0$,
then $H_{\lambda,h}(x,1)=\gamma_{h,\lambda}(x)$ and thus $B_{2}(H_{\lambda,h}(x,1),h)\subset \overline{M}_{i}$. As $||x-\gamma_{h,\lambda}(x)||_{2}\leq h$, it follows, by Assumptions \textbf{A1} and \textbf{A2} (see again Lemma \ref{lmm: A1-A2 imply} in Appendix \ref{sec: A1-A2 imply}), that 
$$H_{\lambda,h}(x,1)\in \mathcal{F}_{\lambda+L3^{\alpha}h^{\alpha}}^{-h}.$$
From the same reasoning, it also follows that $[x,\gamma_{h,\lambda}(x)]\subset \mathcal{F}_{\lambda+Lh^{\alpha}}$ and hence $\mathcal{N}_{\lambda,h}\subset \mathcal{F}_{\lambda+Lh^{\alpha}}$.
\end{itemize}
Combining the different cases, it follows that $\mathcal{M}_{\lambda,h}\subset \mathcal{F}_{\lambda+L(2+5^{\alpha})h^{\alpha}}^{-h}$.
\end{proof}
\subsection{Proof of the continuity of $H_{\lambda,h}$}
\label{proof continuity H}
This section is devoted to the proof of the continuity of $H_{\lambda,h}$, claimed in the proof of Lemma \ref{lemmaFiltEquiv2}.
\begin{lmm}
\label{continuity H}
Let $R/2>h>0$ and $\lambda\in\mathbb{R}$, $H_{\lambda,h}$ is continuous.
\end{lmm}
\begin{proof}
Let $\delta,\delta'>0$, $x,y\in \mathcal{N}_{\lambda,h}$ such that $||x-y||_{2}\leq \delta$ and $t,s\in[0,1]$ such that $|t-s|\leq \delta'$. Let's check the different cases.
\begin{itemize}
    \item If $x\in \overline{M}_{i}$, $\gamma_{\lambda,h}(x)\in M_{i}$ and $\gamma_{\lambda,h}(y)\in \overline{M}^{c}_{i}$, then:
    $$d_{2}\left(\gamma_{h,\lambda}(x),(\overline{M}_{i})^{c}\cap \mathcal{F}_{\lambda+2Lh^{\alpha}}\right)\leq 2h+\delta.$$ 
    Thus, either $x$ satistifies (\ref{cond 3}) and (\ref{cond 4}) and we have :
    $$||\gamma_{h,\lambda}(x)-x||_{2}\geq 3h-d_{2}\left(\gamma_{h,\lambda}(x),(\overline{M}_{i})^{c}\cap \mathcal{F}_{\lambda+2Lh^{\alpha}}\right) \geq ||\gamma_{h,\lambda}(x)-x||_{2}-\delta$$
and consequently
$$\left\|\gamma_{h,\lambda}(x)+\left(3h-d_{2}\left(\gamma_{h,\lambda}(x),(\overline{M}_{i})^{c}\cap \mathcal{F}_{\lambda+2Lh^{\alpha}}\right)\right)_{+}\frac{x-\gamma_{h,\lambda}(x)}{||x-\gamma_{h,\lambda}(x)||_{2}}-x\right\|_{2}\leq \delta.$$
which gives $||H_{\lambda,h}(x,t)-x||_2\leq \delta$. Or,  $x$ does not satistify (\ref{cond 3}) or (\ref{cond 4}) and $H_{\lambda,h}(x,t)=x$. In both cases $||H_{\lambda,h}(x,t)-x||_2\leq \delta$. Similarly we can show that $||H_{\lambda,h}(y,s)-y||_2\leq \delta$. Hence, 
$$||H_{\lambda,h}(x,t)-H_{\lambda,h}(y,s)||_{2}\leq ||H_{\lambda,h}(x,t)-x||_2+||x-y||_{2}+||H_{\lambda,h}(y,s)-y||_2\leq 3\delta.$$
\end{itemize}
From now on, we can suppose that $x,y\in \overline{M}_{i}$ and $\gamma_{\lambda,h}(x),\gamma_{\lambda,h}(y)\in M_i$.
\begin{itemize}
\item If $x$ does not satisfy (\ref{cond 3}) or (\ref{cond 4}), and $y$ does not satisfy (\ref{cond 3}) or (\ref{cond 4}), then directly,
$$||H_{\lambda,h}(x,t)-H_{\lambda,h}(y,s)||_{2}=||x-y||_{2}\leq \delta.$$
\item If $x$ satisfies (\ref{cond 3}) and (\ref{cond 4}) and $y\notin ((\partial M_{i}\cap]0,1[^{d})^{h})^{\circ}$, then, $d_{2}\left(x,\partial M_{i}\cap]0,1[^{d}\right)\geq h-\delta$ and thus $||x-\gamma_{h,\lambda}(x)||_{2}\leq \delta$. As $H_{\lambda,h}(x,t)\in[x,\gamma_{h,\lambda}(x)]$, we have:
$$||H_{\lambda,h}(x,t)-H_{\lambda,h}(y,s)||_{2}=||H_{\lambda,h}(x,t)-y||_{2}\leq ||x-y||_{2}+||x-H_{\lambda,h}(x,t)||_{2}\leq  2\delta.$$
\item If $x$ satisfies (\ref{cond 3}) and (\ref{cond 4}) and $y\in \left(\bigcup\limits_{x\in P_{\lambda,h}}\left[x,\gamma_{h,\lambda}(x)\right]\right)^{c}\cap\left(\left(\partial M_{i}\cap]0,1[^{d}\right)^{h}\right)^{\circ}$. Then, $y\in\mathcal{F}_{\lambda+2Lh^{\alpha}}^{-h}$. Thus, for sufficiently small $\delta$ there exists $z\in (\overline{M}_{i})^{c}\cap B_{2}(y,h)$ and,
$$d_{2}\left(\gamma_{h,\lambda}(x),(\overline{M}_{i})^{c}\cap \mathcal{F}_{\lambda+2Lh^{\alpha}}\right)-2h\leq||\gamma_{h,\lambda}(x)-z||_{2}-2h\leq ||x-y||_{2}\leq \delta.$$
Hence, by (\ref{cond 4}),
$$||\gamma_{h,\lambda}(x)-x||_{2}\geq 3h-d_{2}\left(\gamma_{h,\lambda}(x),(\overline{M}_{i})^{c}\cap \mathcal{F}_{\lambda+2Lh^{\alpha}}\right) \geq ||\gamma_{h,\lambda}(x)-x||_{2}-\delta$$
and consequently,
$$\left\|\gamma_{h,\lambda}(x)+\left(3h-d_{2}\left(\gamma_{h,\lambda}(x),(\overline{M}_{i})^{c}\cap \mathcal{F}_{\lambda+2Lh^{\alpha}}\right)\right)_{+}\frac{x-\gamma_{h,\lambda}(x)}{||x-\gamma_{h,\lambda}(x)||_{2}}-x\right\|_{2}\leq \delta.$$
We then have:
\begin{align*}
&||H_{\lambda,h}(x,t)-H_{\lambda,h}(y,s)||_{2}\\
&=\left\|(1-t)x+t\left(\gamma_{h,\lambda}(x)+\left(3h-d_{2}\left(\gamma_{h,\lambda}(x),(\overline{M}_{i})^{c}\cap \mathcal{F}_{\lambda+2Lh^{\alpha}}\right)\right)_{+}\frac{x-\gamma_{h,\lambda}(x)}{||x-\gamma_{h,\lambda}(x)||_{2}}\right)-y\right\|_{2}\\
&\leq \left\|\gamma_{h,\lambda}(x)+\left(3h-d_{2}\left(\gamma_{h,\lambda}(x),(\overline{M}_{i})^{c}\cap \mathcal{F}_{\lambda+2Lh^{\alpha}}\right)\right)_{+}\frac{x-\gamma_{h,\lambda}(x)}{||x-\gamma_{h,\lambda}(x)||_{2}}-x\right\|_{2} + ||x-y||_{2}\\
&\leq 2\delta
\end{align*}
\item If $x$ satisfies (\ref{cond 3}) and (\ref{cond 4}) and $y$ satisfies (\ref{cond 3}) but not (\ref{cond 4}), then,
\begin{align*}
||\gamma_{h,\lambda}(x)-x||_{2}&\geq3h-d_{2}\left(\gamma_{h,\lambda}(x),(\overline{M}_{i})^{c}\cap \mathcal{F}_{\lambda+2Lh^{\alpha}}\right)\\
&\geq ||y-\gamma_{h,\lambda}(y)||_{2}-||\gamma_{h,\lambda}(x)-\gamma_{h,\lambda}(y)||_{2}\\
&\geq ||x-\gamma_{h,\lambda}(x)||_{2}-2||\gamma_{h,\lambda}(x)-\gamma_{h,\lambda}(y)||_{2}-||x-y||_{2}.
\end{align*}
Thus,
\begin{align*}
&||H_{\lambda,h}(x,t)-H_{\lambda,h}(y,s)||_{2}\\
&=\left\|(1-t)x+t\left(\gamma_{h,\lambda}(x)+\left(3h-d_{2}\left(\gamma_{h,\lambda}(x),(\overline{M}_{i})^{c}\cap \mathcal{F}_{\lambda+2Lh^{\alpha}}\right)\right)_{+}\frac{x-\gamma_{h,\lambda}(x)}{||x-\gamma_{h,\lambda}(x)||_{2}}\right)-y\right\|_{2}\\
&\leq \left\|\gamma_{h,\lambda}(x)+\left(3h-d_{2}\left(\gamma_{h,\lambda}(x),(\overline{M}_{i})^{c}\cap \mathcal{F}_{\lambda+2Lh^{\alpha}}\right)\right)_{+}\frac{x-\gamma_{h,\lambda}(x)}{||x-\gamma_{h,\lambda}(x)||_{2}}-x\right\|_{2} + ||x-y||_{2}\\
&\leq 2||\gamma_{h,\lambda}(x)-\gamma_{h,\lambda}(y)||_{2}+2||x-y||_{2}\\
&\leq 2\delta + 2||\gamma_{h,\lambda}(x)-\gamma_{h,\lambda}(y)||_{2}
\end{align*}
and we conclude, in this case, by continuity of $\gamma_{h,\lambda}$.
\item Finally, if $x$ satisfies (\ref{cond 3}) and (\ref{cond 4}) and $y$ satisfies (\ref{cond 3}) and (\ref{cond 4}), then,
\begin{align*}
&||H_{\lambda,h}(x,t)-H_{\lambda,h}(y,s)||_{2}\\
&=\left\|(1-t)x+t\left(\gamma_{h,\lambda}(x)+\left(3h-d_{2}\left(\gamma_{h,\lambda}(x),(\overline{M}_{i})^{c}\cap \mathcal{F}_{\lambda+2Lh^{\alpha}}\right)\right)_{+}\frac{x-\gamma_{h,\lambda}(x)}{||x-\gamma_{h,\lambda}(x)||_{2}}\right)\right.\\ 
&\quad\left.-(1-s)y-s\left(\gamma_{h,\lambda}(y)+\left(3h-d_{2}\left(\gamma_{h,\lambda}(y),(\overline{M}_{i})^{c}\cap \mathcal{F}^{-h}_{\lambda+Lh^{\alpha}}\right)\right)_{+}\frac{y-\gamma_{h,\lambda}(y)}{||y-\gamma_{h,\lambda}(y)||_{2}}\right)\right\|_{2}
\end{align*}
and again the conclusion, follows in this case, by continuity of $\gamma_{h,\lambda}$.
\end{itemize}
All possible cases have been checked, the proof is complete. 
\end{proof}
\subsection{Proof of Lemma \ref{lemma-histo 2}}
 \label{proof lemma-histo 2}
This section is devoted to the proof of Lemma \ref{lemma-histo 2} from Section \ref{Upperbounds section}. The arguments for Assertions \eqref{lemma-histo 2a} and \eqref{lemma-histo 2b} rely on the same underlying ideas:
\begin{itemize}
    \item First, due to Assumption \textbf{A3} (see details in Appendix \ref{sec: detail A3}), in a sufficiently small neighborhood \( V \) of any point \( x \) close to the boundary of the pieces of \( f \), the boundary locally corresponds to the frontier between two pieces \( M_i \) and \( M_j \).
    \item Secondly, by Assumption \textbf{A3} again, this frontier can be well approximated (up to a quadratic error) by a hyperplane (see Figure \ref{fig: reach ball}).
\end{itemize}
We formalize this intuition in Lemma~\ref{lemma-histo 1}. 
\begin{lmm}
\label{lemma-histo 1}
Let $0<K$ and $0<h<1$ such that $Kh<R/2$. There exists a constant $C_{2}$ (depending only on $K$, $d$ and $R$) such that for all $i,j\in\{1,...,l\}$ and $x\in B_{2}(\left(\partial M_{i}\cap]0,1[^{d}\right),Kh)\cap \overline{M}_{i}$ such that $\xi(x)\in \partial M_{j}$. We have:
\begin{equation}
\label{eq: haussdorf MI}
d_{2}\left(B_{2}\left(\xi(x),Kh\right)\cap M_{j},B_{2}\left(\xi(x),Kh\right)\cap\underline{P}(x)\right)\leq C_{2}h^{2}
\end{equation}
with
$$\underline{P}(x)=\left\{z\in[0,1]^{d}\text{ s.t. }\left\langle z,\frac{x-\xi(x)}{\left\|x-\xi(x)\right\|_{2}}\right\rangle\leq \left\langle \xi(x) ,\frac{x-\xi(x)}{\left\|x-\xi(x)\right\|_{2}}\right\rangle\right\}$$
and
\begin{equation}
\label{eq: haussdorf MI 2}
d_{2}\left(B_{2}\left(\xi(x),Kh\right)\cap M_{i},B_{2}\left(\xi(x),Kh\right)\cap\overline{P}(x)\right)\leq C_{2}h^{2}
\end{equation}
with
$$\overline{P}(x)=\left\{z\in[0,1]^{d}\text{ s.t. }\left\langle z,\frac{x-\xi(x)}{\left\|x-\xi(x)\right\|_{2}}\right\rangle\geq \left\langle \xi(x) ,\frac{x-\xi(x)}{\left\|x-\xi(x)\right\|_{2}}\right\rangle\right\}.$$
Furthermore if $z\in \underline{P}(x)\cap B_{2}\left(\xi(x),Kh\right)$ then
\begin{equation}
\label{eq: sphere translation}
z-C_{2}h^{2}\frac{x-\xi(x)}{||\xi(x)-x||_{2}}\in \overline{M}_{j}
\end{equation}
and if $z\in \overline{P}(x)\cap B_{2}\left(\xi(x),Kh\right)$
\begin{equation}
\label{eq: sphere translation 2}
z+C_{2}h^{2}\frac{x-\xi(x)}{||\xi(x)-x||_{2}}\in \overline{M}_{i}.
\end{equation}
\end{lmm}
\begin{proof}Let $B_{1}$ be the Euclidean closed ball centered at $\xi(x)-R\frac{x-\xi(x)}{\left\|x-\xi(x)\right\|_{2}}$ of radius $R$ and $B_{2}$ be the Euclidean closed ball centered at $\xi(x)+R\frac{x-\xi(x)}{\left\|x-\xi(x)\right\|_{2}}$ of radius $R$. By Assumption \textbf{A3}, $B_{1}\subset \overline{M_{j}}$ and $B_{2}\subset \overline{M_{i}}$ (see details in Appendix \ref{sec: detail A3}). Then, the Hausdorff distance between $B_{2}\left(\xi(x),Kh\right)\cap M_{j}$ and $B_{2}\left(\xi(x),Kh\right)\cap\underline{P}(x)$, and between $B_{2}\left(\xi(x),Kh\right)\cap M_{i}$ and $B_{2}\left(\xi(x),Kh\right)\cap\overline{P}(x)$, are both upper bounded by the Hausdorff distance between $\partial B_{1} \cup \partial B_{2}$ intersected with $B_{2}\left(\xi(x),Kh\right)$ and the intersection with $B_{2}\left(\xi(x),Kh\right)$ of the hyperplane:
$$P(x)=\left\{z\in[0,1]^{d}\text{ s.t. }\left\langle z,\frac{x-\xi(x)}{\left\|x-\xi(x)\right\|_{2}}\right\rangle= \left\langle \xi(x) ,\frac{x-\xi(x)}{\left\|x-\xi(x)\right\|_{2}}\right\rangle\right\}.$$
By symmetry, this distance is equal to the Hausdorff distance between $\partial B_{1}\cap B_{2}\left(\xi(x),Kh\right)$ and $P(x)\cap B_{2}\left(\xi(x),Kh\right)$.
\begin{figure}[H]
    \centering
    \tikzset{every picture/.style={line width=0.75pt}} 

\begin{tikzpicture}[x=0.75pt,y=0.75pt,yscale=-1,xscale=1]

\draw    (347.7,107.05) -- (347.7,74.55) -- (347.7,26.1) ;
\draw [shift={(347.7,23.1)}, rotate = 90] [fill={rgb, 255:red, 0; green, 0; blue, 0 }  ][line width=0.08]  [draw opacity=0] (8.93,-4.29) -- (0,0) -- (8.93,4.29) -- cycle    ;
\draw [shift={(347.7,110.05)}, rotate = 270] [fill={rgb, 255:red, 0; green, 0; blue, 0 }  ][line width=0.08]  [draw opacity=0] (8.93,-4.29) -- (0,0) -- (8.93,4.29) -- cycle    ;
\draw [color={rgb, 255:red, 74; green, 144; blue, 226 }  ,draw opacity=1 ][line width=2.25]    (100.4,110.15) -- (574.2,110.15) ;
\draw   (100,2.6) -- (574.5,2.6) -- (574.5,278.6) -- (100,278.6) -- cycle ;
\draw  [dash pattern={on 4.5pt off 4.5pt}]  (100,169.5) .. controls (287,93.5) and (400,85.5) .. (574,169.5) ;
\draw  [draw opacity=0] (213.57,279.08) .. controls (210.43,267.82) and (208.75,255.98) .. (208.75,243.75) .. controls (208.75,169.66) and (270.31,109.6) .. (346.25,109.6) .. controls (422.19,109.6) and (483.75,169.66) .. (483.75,243.75) .. controls (483.75,255.22) and (482.28,266.35) .. (479.5,276.97) -- (346.25,243.75) -- cycle ; \draw  [color={rgb, 255:red, 208; green, 2; blue, 27 }  ,draw opacity=1 ] (213.57,279.08) .. controls (210.43,267.82) and (208.75,255.98) .. (208.75,243.75) .. controls (208.75,169.66) and (270.31,109.6) .. (346.25,109.6) .. controls (422.19,109.6) and (483.75,169.66) .. (483.75,243.75) .. controls (483.75,255.22) and (482.28,266.35) .. (479.5,276.97) ;  
\draw  [draw opacity=0] (480.77,2.78) .. controls (467.59,63.46) and (412.37,109) .. (346.25,109) .. controls (279.71,109) and (224.21,62.89) .. (211.49,1.63) -- (346.25,-25.15) -- cycle ; \draw  [color={rgb, 255:red, 208; green, 2; blue, 27 }  ,draw opacity=1 ] (480.77,2.78) .. controls (467.59,63.46) and (412.37,109) .. (346.25,109) .. controls (279.71,109) and (224.21,62.89) .. (211.49,1.63) ;  
\draw  [dash pattern={on 0.84pt off 2.51pt}] (260.75,110.05) .. controls (260.75,62.03) and (299.68,23.1) .. (347.7,23.1) .. controls (395.72,23.1) and (434.65,62.03) .. (434.65,110.05) .. controls (434.65,158.07) and (395.72,197) .. (347.7,197) .. controls (299.68,197) and (260.75,158.07) .. (260.75,110.05) -- cycle ;
\draw  [fill={rgb, 255:red, 155; green, 155; blue, 155 }  ,fill opacity=1 ] (344.25,110.05) .. controls (344.25,108.14) and (345.79,106.6) .. (347.7,106.6) .. controls (349.61,106.6) and (351.15,108.14) .. (351.15,110.05) .. controls (351.15,111.96) and (349.61,113.5) .. (347.7,113.5) .. controls (345.79,113.5) and (344.25,111.96) .. (344.25,110.05) -- cycle ;
\draw    (349.25,243.75) -- (481,243.7) ;
\draw [shift={(484,243.7)}, rotate = 179.98] [fill={rgb, 255:red, 0; green, 0; blue, 0 }  ][line width=0.08]  [draw opacity=0] (8.93,-4.29) -- (0,0) -- (8.93,4.29) -- cycle    ;
\draw [shift={(346.25,243.75)}, rotate = 359.98] [fill={rgb, 255:red, 0; green, 0; blue, 0 }  ][line width=0.08]  [draw opacity=0] (8.93,-4.29) -- (0,0) -- (8.93,4.29) -- cycle    ;
\draw  [fill={rgb, 255:red, 155; green, 155; blue, 155 }  ,fill opacity=1 ] (344.25,23.1) .. controls (344.25,21.19) and (345.79,19.65) .. (347.7,19.65) .. controls (349.61,19.65) and (351.15,21.19) .. (351.15,23.1) .. controls (351.15,25.01) and (349.61,26.55) .. (347.7,26.55) .. controls (345.79,26.55) and (344.25,25.01) .. (344.25,23.1) -- cycle ;
\draw    (264.85,133.5) -- (265.65,87.1) ;
\draw [shift={(265.7,84.1)}, rotate = 90.98] [fill={rgb, 255:red, 0; green, 0; blue, 0 }  ][line width=0.08]  [draw opacity=0] (8.93,-4.29) -- (0,0) -- (8.93,4.29) -- cycle    ;
\draw [shift={(264.8,136.5)}, rotate = 270.98] [fill={rgb, 255:red, 0; green, 0; blue, 0 }  ][line width=0.08]  [draw opacity=0] (8.93,-4.29) -- (0,0) -- (8.93,4.29) -- cycle    ;

\draw (111,20.4) node [anchor=north west][inner sep=0.75pt]    {$M_{i}$};
\draw (110,245.4) node [anchor=north west][inner sep=0.75pt]    {$M_{j}$};
\draw (411,226.4) node [anchor=north west][inner sep=0.75pt]    {$R$};
\draw (524,88.4) node [anchor=north west][inner sep=0.75pt]    {$\textcolor[rgb]{0.29,0.56,0.89}{P}$};
\draw (471,34.4) node [anchor=north west][inner sep=0.75pt]    {$\textcolor[rgb]{0.82,0.01,0.11}{B_{2}}$};
\draw (472,151.4) node [anchor=north west][inner sep=0.75pt]    {$\textcolor[rgb]{0.82,0.01,0.11}{B}\textcolor[rgb]{0.82,0.01,0.11}{_{1}}$};
\draw (352.33,59.07) node [anchor=north west][inner sep=0.75pt]    {$Kh$};
\draw (355.33,5.73) node [anchor=north west][inner sep=0.75pt]    {$x$};
\draw (349.7,116.9) node [anchor=north west][inner sep=0.75pt]    {$\xi ( x)$};
\draw (228,91) node [anchor=north west][inner sep=0.75pt]    {$2C_{2}h^{2}$};

\end{tikzpicture}
    \caption{The tangent balls $B_{1}$ and $B_{2}$ of radius $R$ delimit the region where $\partial M_{i}\cap \partial M_{j}$ lies in. 2D illustration.}
    \label{fig: reach ball}
\end{figure}
Now, let $z\in\partial B_{1}\setminus\{\xi(x)\}$, and $p(z)$ be its projection on $P(x)$. Let $Q$ be the plane containing $z$, $p(z)$ and $\xi(x)$, $Q$ intersects $\partial B_{1}$ in a circle $\mathcal{C}$ of radius $R$ and intersects $P$ in a line $D$ tangent to $\mathcal{C}$. The problem then reduces to upper-bounding the distance between a circle and a tangent line around the intersection point. Without loss of generality, we can suppose that we are in $\mathbb{R}^{2}$, $\mathcal{C}$ is the circle of radius $R$ centered at $(0,R)$ and $D$ the line $y=0$ (tangent to $\mathcal{C}$ at $(0,0)$). In $B((0,0),Kh)$, as $Kh<R/2$, $\mathcal{C}$ can be described as,
$$\mathcal{C}=\left\{(x,y)\in B((0,0),Kh)\text{ s.t. }y=R-\sqrt{R^{2}-x^{2}}\right\}.$$
Hence the distance between $\mathcal{C}$ and $D$ in $B((0,0),Kh)$ is upper bounded by:
$$R-\sqrt{R^{2}-(Kh)^{2}}=\frac{K^{2}}{2R}h^{2}+O(h^{3}).$$
Assertions (\ref{eq: haussdorf MI}) and (\ref{eq: haussdorf MI 2}) then follow. Now, simply observe that since $Kh<R/2$, for all $z\in \underline{P}(x)$, $z-(R-\sqrt{R^{2}-(Ch)^{2}})(x-\xi(x))/||x-\xi(x)||_{2}\in B_{1}$ and (\ref{eq: sphere translation}) follows. And symmetrically, for all $z\in \overline{P}(x)$, $z+(R-\sqrt{R^{2}+(Ch)^{2}})(x-\xi(x))/||x-\xi(x)||_{2}\in B_{2}$ and (\ref{eq: sphere translation 2}) follows. 
\end{proof}\vspace{0.25cm}
Before turning to the full proof of Lemma~\ref{lemma-histo 2}, we give a high-level outline of the argument. The idea being essentially the same for both \eqref{lemma-histo 2a} and \eqref{lemma-histo 2b}, we present the outline for proving \eqref{lemma-histo 2a}. Let $H_1\in C_{\lambda,h}$ be a hypercube of center $x_1$ and $x\in H_1\cap P_{\lambda+\sqrt{2}\sigma N_{h},\sqrt{d}h}$:
\begin{itemize}
    \item If the entire segment \([x, \xi(x)]\) is contained in \( H_1 \), then we directly have \([x, \xi(x)] \subset \widehat{\mathcal{F}}_{\lambda} \).
    
    \item Otherwise, let \( y \) be the first point along the segment \([x, \xi(x)]\) where it exits \( H_1\). Then \( y \) lies in a hypercube \( H_2 \in C_h \) of center $x_2$ and adjacent to \( H_1 \) (in the sense that they share a common boundary point). Moving along the line segment $[x, \xi(x)]$, we go from a piece \( M_i \), where $f$ (around $x$) is significantly greater than $\lambda$ to the boundaries of another piece \( M_j \), where the signal $f$ (around $\xi(x)$) is smaller than $\lambda+\varepsilon$, with \( \varepsilon \asymp h^\alpha \). Exploiting Lemma~\ref{lemma-histo 1}, we show that the translation of the set \( H_1\cap\overline{M}_{j}\cap G_{1/N}\) by $x_2-x_1$ (which is a subset of $H_2$), remains within $\overline{M}_{j}\cap G_{1/N}$, except for a subset of $H_1\cap G_{1/N}$ that has cardinality of order $O(nh^{d+1})$. Using this and Assumptions \textbf{A0}-\textbf{A2}, we show that \( H' \in \widehat{\mathcal{F}}_{\lambda + \varepsilon} \) for some \( \varepsilon \asymp h^\alpha \).
    
    \item We then repeat the same argument starting from \( y \). Iterating this reasoning along the segment \([x, \xi(x)]\), we establish assertion \eqref{lemma-histo 2a}.
\end{itemize}

\begin{proof}[Proof of Lemma \ref{lemma-histo 2}]
We start by proving Assertion (\ref{lemma-histo 2a}). Let $x\in \overline{M}_{i}\cap\widehat{\mathcal{F}}_{\lambda}\cap S_{\lambda+\sqrt{2}\sigma N_{h}h^{\alpha},\sqrt{d}h}$ and assume that $\xi(x)\in \partial M_{j}$. In particular, we have $||x-\xi(x)||\leq \sqrt{d}h$. Assume that $n$ is sufficiently large such that $3\sqrt{d}h<R/2$. By Assumption \textbf{A3} (see Lemma \ref{lmm: A3 imply} in Appendix \ref{sec: detail A3}), we have:
$$([x,\xi(x)])^{\sqrt{d}h}\subset \overline{M_{i}}\cup \overline{M}_{j}.$$
Furthermore, as $x\in S_{\lambda+\sqrt{2}\sigma N_{h}h^{\alpha},\sqrt{d}h}$, $$x\in\overline{\left(M_j\cap\mathcal{F}_{\lambda+\sqrt{2}\sigma N_{h}h^{\alpha}}\right)^{\sqrt{d}h}}.$$
Let $H_{1}\in C_{h,\lambda}$ be a hypercube containing $x$ and denote $x_{1}$ its center. Suppose that there exists $y\in \left[x,\xi(x)\right]$ such that $y\notin H_{1}$. Denote $H_{2}$ the hypercube of $C_{h}$ containing $y$ and $x_{2}$ its center. Suppose furthermore that $H_{2}$ is adjacent to $H_{1}$ (i.e., $H_{1}\cap H_{2}\ne \emptyset$). If,
$$H_{1}\cap M_{i}\cap \mathcal{F}_{\lambda+(\sqrt{2}\sigma N_{h}+3L\sqrt{d})h^{\alpha}}\ne\emptyset$$
then, as $H_{2}\subset H_1^{\sqrt{d}h}$, $H_{2}\subset \mathcal{F}_{\lambda+(\sqrt{2}\sigma N_{h}+5L\sqrt{d})h^{\alpha}}$ by Assumptions \textbf{A1} and \textbf{A2} (see Lemma \ref{lmm: A1-A2 imply} in Appendix \ref{sec: A1-A2 imply}), and thus, by Lemma \ref{lmm1}, $H_{2}\subset \widehat{\mathcal{F}}_{\lambda+(2\sqrt{2}\sigma N_{h}+5\sqrt{d})h^{\alpha}}$.
\begin{figure}[h]
    \centering
\begin{tikzpicture}[x=0.75pt,y=0.75pt,yscale=-1,xscale=1]

\draw [color={rgb, 255:red, 74; green, 144; blue, 226 }  ,draw opacity=1 ][line width=1.5]    (201.5,119) -- (408.5,466) ;
\draw   (458.32,339.64) -- (586.05,339.64) -- (586.05,467.37) -- (458.32,467.37) -- cycle ;
\draw   (330.83,339.41) -- (458.56,339.41) -- (458.56,467.13) -- (330.83,467.13) -- cycle ;
\draw   (203.11,339.41) -- (330.83,339.41) -- (330.83,467.13) -- (203.11,467.13) -- cycle ;
\draw   (202.87,211.92) -- (330.6,211.92) -- (330.6,339.64) -- (202.87,339.64) -- cycle ;
\draw   (457.85,84.67) -- (585.57,84.67) -- (585.57,212.39) -- (457.85,212.39) -- cycle ;
\draw   (202.87,84.67) -- (330.6,84.67) -- (330.6,212.39) -- (202.87,212.39) -- cycle ;
\draw    (442.88,315.13) -- (356.08,370.76) ;
\draw [shift={(354.4,371.84)}, rotate = 327.34] [color={rgb, 255:red, 0; green, 0; blue, 0 }  ][line width=0.75]    (10.93,-3.29) .. controls (6.95,-1.4) and (3.31,-0.3) .. (0,0) .. controls (3.31,0.3) and (6.95,1.4) .. (10.93,3.29)   ;
\draw  [fill={rgb, 255:red, 208; green, 2; blue, 27 }  ,fill opacity=1 ] (372.94,359.14) .. controls (372.94,356.81) and (374.95,354.91) .. (377.43,354.91) .. controls (379.9,354.91) and (381.91,356.81) .. (381.91,359.14) .. controls (381.91,361.48) and (379.9,363.37) .. (377.43,363.37) .. controls (374.95,363.37) and (372.94,361.48) .. (372.94,359.14) -- cycle ;
\draw  [dash pattern={on 4.5pt off 4.5pt}]  (202.87,84.67) .. controls (281.5,129) and (320.5,181) .. (331.5,249) .. controls (342.5,317) and (333.31,389.38) .. (402.12,419.31) .. controls (470.92,449.24) and (560.71,438.04) .. (586.05,467.37) ;
\draw  [fill={rgb, 255:red, 155; green, 155; blue, 155 }  ,fill opacity=1 ] (438.4,315.13) .. controls (438.4,312.79) and (440.41,310.9) .. (442.88,310.9) .. controls (445.36,310.9) and (447.37,312.79) .. (447.37,315.13) .. controls (447.37,317.46) and (445.36,319.35) .. (442.88,319.35) .. controls (440.41,319.35) and (438.4,317.46) .. (438.4,315.13) -- cycle ;
\draw  [fill={rgb, 255:red, 155; green, 155; blue, 155 }  ,fill opacity=1 ] (349.91,371.84) .. controls (349.91,369.5) and (351.92,367.61) .. (354.4,367.61) .. controls (356.88,367.61) and (358.89,369.5) .. (358.89,371.84) .. controls (358.89,374.17) and (356.88,376.07) .. (354.4,376.07) .. controls (351.92,376.07) and (349.91,374.17) .. (349.91,371.84) -- cycle ;
\draw  [fill={rgb, 255:red, 74; green, 144; blue, 226 }  ,fill opacity=1 ] (389.25,278.63) .. controls (389.25,276.29) and (391.26,274.4) .. (393.73,274.4) .. controls (396.21,274.4) and (398.22,276.29) .. (398.22,278.63) .. controls (398.22,280.96) and (396.21,282.86) .. (393.73,282.86) .. controls (391.26,282.86) and (389.25,280.96) .. (389.25,278.63) -- cycle ;
\draw  [fill={rgb, 255:red, 74; green, 144; blue, 226 }  ,fill opacity=1 ] (389.21,406.97) .. controls (389.21,404.63) and (391.22,402.74) .. (393.69,402.74) .. controls (396.17,402.74) and (398.18,404.63) .. (398.18,406.97) .. controls (398.18,409.3) and (396.17,411.2) .. (393.69,411.2) .. controls (391.22,411.2) and (389.21,409.3) .. (389.21,406.97) -- cycle ;
\draw    (393.31,279.13) -- (393.69,404.97) ;
\draw [shift={(393.69,406.97)}, rotate = 269.83] [color={rgb, 255:red, 0; green, 0; blue, 0 }  ][line width=0.75]    (10.93,-3.29) .. controls (6.95,-1.4) and (3.31,-0.3) .. (0,0) .. controls (3.31,0.3) and (6.95,1.4) .. (10.93,3.29)   ;
\draw   (330.6,84.67) -- (458.32,84.67) -- (458.32,212.39) -- (330.6,212.39) -- cycle ;
\draw   (458.32,212.39) -- (586.05,212.39) -- (586.05,340.12) -- (458.32,340.12) -- cycle ;
\draw   (330.6,211.92) -- (458.32,211.92) -- (458.32,339.64) -- (330.6,339.64) -- cycle ;

\draw (442.38,289.48) node [anchor=north west][inner sep=0.75pt]  [color={rgb, 255:red, 155; green, 155; blue, 155 }  ,opacity=1 ]  {$x$};
\draw (318.56,359.62) node [anchor=north west][inner sep=0.75pt]  [color={rgb, 255:red, 155; green, 155; blue, 155 }  ,opacity=1 ]  {$\xi ( x)$};
\draw (370.29,365.16) node [anchor=north west][inner sep=0.75pt]    {$\textcolor[rgb]{0.82,0.01,0.11}{y}$};
\draw (398.65,259.43) node [anchor=north west][inner sep=0.75pt]    {$\textcolor[rgb]{0.29,0.56,0.89}{x}\textcolor[rgb]{0.29,0.56,0.89}{_{1}}$};
\draw (403.25,380.05) node [anchor=north west][inner sep=0.75pt]    {$\textcolor[rgb]{0.29,0.56,0.89}{x}\textcolor[rgb]{0.29,0.56,0.89}{_{2}}$};
\draw (515,108.54) node [anchor=north west][inner sep=0.75pt]    {$M_{i}$};
\draw (231,414.54) node [anchor=north west][inner sep=0.75pt]    {$M_{j}$};
\draw (435,218.4) node [anchor=north west][inner sep=0.75pt]    {$H_{1}$};
\draw (435,344.4) node [anchor=north west][inner sep=0.75pt]    {$H_{2}$};
\draw (251,236.4) node [anchor=north west][inner sep=0.75pt]    {$\textcolor[rgb]{0.29,0.56,0.89}{P}$};

\end{tikzpicture}

    \caption{A 2D scenario where $\xi(x)$ is not in $H_{1}$.}
    \label{fig:enter-label}
\end{figure}
From now on, we suppose:
\begin{equation}
\label{eq: condition far}
H_{1}\cap M_{i}\cap \mathcal{F}_{\lambda+(\sqrt{2}\sigma N_{h}+3L\sqrt{d})h^{\alpha}}=\emptyset.
\end{equation}
As $y\in H_{2}$, then,
\begin{equation}
\label{eq:scal prod 1}
\left\langle \frac{\xi(x)-x}{||\xi(x)-x||_{2}},x_{2}-x_{1}\right\rangle=\left\langle \frac{y-x}{||y-x||_2},x_{2}-x_{1}\right\rangle \geq 0.
\end{equation}
Let us denote $\underline{P}=\underline{P}(x)$ and $P=P(x)$. Then, it follows that for all $z\in H_{1}\cap \underline{P}$, $z+(x_{2}-x_{1})\in H_{2}\cap \underline{P}$. Note that for all $z\in H_{2}\cup H_{1}$, $||\xi(x)-z||_{2}\leq 3\sqrt{d}h$. Thus, take $K=3\sqrt{d}$ and let $C_{2}$ be the corresponding constant according to Lemma \ref{lemma-histo 1}. For a set $A$ and a vector $u$, we denote $A+u=\{z+u,z\in A\}$. By Lemma \ref{lemma-histo 1}, if $z\in H_{1}\cap \overline{M}_{j}\setminus P^{2C_{2}h^{2}}$ then $z\in \underline{P}$. Thus, $z+(x_2-x_1)\in \underline{P}$. Furthermore, if $z+(x_2-x_1)\notin  P^{2C_{2}h^{2}}$, applying Lemma \ref{lemma-histo 1} again, we have that $z+(x_2-x_1)\in  \overline{M}_{j}$. Therefore, we have:
\begin{equation}
\label{eq: inclusion hyper}
\left((H_{1}\cap \overline{M}_{j})\setminus P^{2C_{2}h^{2}}+(x_2-x_1)\right)\setminus P^{2C_{2}h^{2}}\subset H_{2}\cap \overline{M}_{j}.    
\end{equation}

As $H_{2}\subset H_{1}^{\sqrt{d}h}$, Assumptions \textbf{A1} and \textbf{A2} imply that, for all $z\in H_{2}\cap \overline{M}_{j}$, $z\in\mathcal{F}_{\lambda+(\sqrt{2}\sigma N_{h}+3L\sqrt{d})h^{\alpha}}$ (see Lemma \ref{lmm: A1-A2 imply} in Appendix \ref{sec: A1-A2 imply}). Hence, by (\ref{eq: condition far}),
\begin{equation*}
\label{eq: inf sup}
\sup\limits_{z\in H_{2}\cap \overline{M}_{j}}f(z)\leq \inf\limits_{z\in H_{1}\cap M_{i}}f(z). 
\end{equation*}
Now, observe that:
\begin{itemize}
    \item if $z\in H_{2}\cap \overline{M_{j}}$, either $z-(x_2-x_1)\in M_{i}\cap H_1$, and then (\ref{eq: condition far}) ensures that $f(z)\leq f(z-(x_2-x_1))$, or $z-(x_2-x_1)\in \overline{M}_{j}\cap H_1$, and then Assumptions \textbf{A1} and \textbf{A2} ensure that $f(z)\leq f(z-(x_2-x_1))+L(\sqrt{d}h)^{\alpha}$ (as $||x_2-x_1||_2\leq \sqrt{d}h$ and by Lemma \ref{lmm: A1-A2 imply} in Appendix \ref{sec: A1-A2 imply}). Thus, in both cases,
$$f(z)\leq f(z-(x_2-x_1))+L(\sqrt{d}h)^{\alpha}.$$
\item If $z\in H_{2}\cap M_{i}$ and $z-(x_2-x_1)\in M_{i}\cap H_1$, then, Assumptions \textbf{A1} and \textbf{A2} ensure that $f(z)\leq f(z-(x_2-x_1))+L(\sqrt{d}h)^{\alpha}$ (Lemma \ref{lmm: A1-A2 imply} in Appendix \ref{sec: A1-A2 imply}).
\item If $z\in H_{2}\cap M_{i}$ and $z-(x_2-x_1)\in \overline{M}_{j}\cap H_1$, then, by \eqref{eq: inclusion hyper}, $z\in P^{2C_{2}h^{2}}$ or $z-(x_2-x_1)\in P^{2C_{2}h^{2}}$.
\end{itemize}
Now, let $x_{1},...,x_{n}$ be the points of $G_{1/N}$ (the regular grid on which the signal is observed). Let $A_1=H_2\cap \overline{M}_j$, $A_2=\{z\in H_{2}\cap M_{i}$ \text{ s.t. } $z-(x_2-x_1)\in M_{i}\cap H_1\}$ and $A_3=\{z\in H_{2}\cap M_{i} \text{ s.t. }z-(x_2-x_1)\in \overline{M}_{j}\cap H_1\}$. As we have shown that: 
\begin{align*}
A_3&\subset \{z\in H_2\cap P^{2C_{2}h^{2}}\}\cup\{z\in H_2 \text{ s.t. }z-(x_2-x_1)\in H_1\cap P^{2C_{2}h^{2}}\}\\
&=\left(P^{2C_{2}h^{2}}\cup\left(P^{2C_{2}h^{2}}+(x_2-x_1)\right)\right)\cap H_2
\end{align*}
there exists a constant $\kappa$ (only depending on $d$ and $R$) such that $|A_3\cap G_{1/N}|\leq \kappa nh^{d+1}$. Using Assumption \textbf{A0}, it follows that:
\begin{align*}
\sum_{x_{i}\in H_{2}}f(x_{i})&=\sum_{x_i\in A_1}f(x_{i})+\sum_{x_i\in A_2}f(x_{i})+\sum_{x_i\in A_3}f(x_{i})\\
&\leq \sum_{x_i\in A_1}\left(f(x_{i}-(x_2-x_1))+L(\sqrt{d}h)^{\alpha}\right)\\
&\quad+\sum_{x_i\in A_2}\left(f(x_{i}-(x_2-x_1))+L(\sqrt{d}h)^{\alpha}\right)\\
&\quad+\sum_{x_i\in A_3}\left(f(x_i-(x_2-x_1))-f(x_i-(x_2-x_1)) +f(x_i)\right) \\
&\leq \sum_{x_i\in A_1\cup A_2\cup A_3} \left(f(x_{i}-(x_2-x_1))+L(\sqrt{d}h)^{\alpha}\right)+\sum_{x_i\in A_3}2\sup_{z\in [0,1]^{d}}|f(z)|\\
&\leq \sum_{x_i\in A_1\cup A_2\cup A_3} \left(f(x_{i}-(x_2-x_1))+ L(\sqrt{d}h)^{\alpha}\right)+2nh^{d+1}\kappa M\\
&=\sum_{x_{i}\in H_{1}}f(x_{i})+nh^{d}L(\sqrt{d}h)^{\alpha}+2nh^{d+1}\kappa M.
\end{align*}
Now, as $H_1\in C_{h,\lambda}$, recall that we have:
$$\frac{1}{nh^{d}}\sum_{x_{i}\in H_1}X_i=\frac{1}{nh^{d}}\sum_{x_{i}\in H_1}\left(f(x_{i})+\sigma\varepsilon_i\right)\leq  \lambda$$
Thus, by choice of $h$ (and as $h\leq 1$),
\begin{align*}
\frac{1}{nh^{d}}\sum_{x_{i}\in H_{2}} X_{i}&=\frac{1}{nh^{d}}\left(\sum_{x_{i}\in H_{2}} f(x_{i})+\sigma\sum_{x_{i}\in H_{2}}\varepsilon_i\right)\\
&\leq \frac{1}{nh^{d}}\left(\sum_{x_{i}\in H_{1}}f(x_{i})+nh^{d}L(\sqrt{d}h)^{\alpha}+2nh^{d+1}\kappa M+\sigma\sum_{x_{i}\in H_{2}}\varepsilon_i\right)\\
&\leq \lambda + L(3\sqrt{d}h)^{\alpha}+2\kappa Mh+\frac{\sigma}{nh^{d}}\sum_{x_{i}\in H_{2}}\varepsilon_i-\frac{\sigma}{nh^{d}}\sum_{x_{i}\in H_{1}}\varepsilon_i\\
&\leq \lambda + L(3\sqrt{d}h)^{\alpha}+2\kappa Mh^{\alpha}+ 2N_{h}\sqrt{2}\sigma \sqrt{\frac{\log\left(1/h^{d}\right)}{nh^{d}}}\\
&\leq \lambda + L(3\sqrt{d}h)^{\alpha}+2\kappa Mh^{\alpha}+2\sqrt{2}\sigma N_{h}h^{\alpha}.
\end{align*}
Therefore: 
$$H_{2}\subset \widehat{\mathcal{F}}_{\lambda+(2\sqrt{2}\sigma N_{h}+5L\sqrt{d}+2\kappa M)h^{\alpha}}.$$
Now, if $H_{2}$ is not adjacent to $H_{1}$, there exists a finite sequence $H_{3},H_{4},...,H_{p}$ of cube of $C_{h}$ such that for all $k\in\{2,...,n\}$, $[x,\xi(x)]\cap H_{k}\ne\emptyset$ and $H_{3}$ is adjacent to $H_{2}$, $H_{4}$ adjacent to $H_{3}$, ..., and $H_{p}$ adjacent to $H_{1}$. Applying the previous reasoning iteratively then gives that, for all $k\in\{2,...,p\}$,
$$H_{k}\subset\widehat{\mathcal{F}}_{\lambda+(1+p-k)(2\sqrt{2}\sigma N_{h}+5L\sqrt{d}+2\kappa M)h^{\alpha}}.$$
Note that since $||x-\xi(x)||\leq \sqrt{d}h$, we have:
$$p\leq \left|B_2(x,\sqrt{d}h)\cap C_h\right|\leq  \left|B_2(H_1,\sqrt{d}h)\cap C_h\right|\leq \left|B_\infty(H_1,dh)\cap C_h\right|\leq(2d)^{d}$$. 
Thus, for all $k\in\{2,...,n\}$,
$$H_{k}\subset\widehat{\mathcal{F}}_{\lambda+(1+(2d)^{d})(2\sqrt{2}\sigma N_{h}+5L\sqrt{d}+2\kappa M)h^{\alpha}}.$$
and Assertion (\ref{lemma-histo 2a}) follows.\\\\
Now, we prove assertion (\ref{lemma-histo 2b}). The proof is essentially the same as that for assertion (\ref{lemma-histo 2a}). Suppose that $x\in\widehat{\mathcal{F}}_{\lambda}\cap P_{\lambda,\sqrt{d}h}\cap \overline{M}_i$ and that $\gamma_{\lambda,\sqrt{d}h}(x)\in M_{i}$. In particular, it implies:
$$x\in \overline{\mathcal{F}_{\lambda}\cap M_{i}}\text{ and }||x-\gamma_{\lambda,\sqrt{d}h}(x)||_{2}\leq \sqrt{d}h.$$
Assume that $n$ is sufficiently large such that $3\sqrt{d}h<R/2$. Assumption \textbf{A3} ensures (see Lemma \ref{lmm: A3 imply} in Appendix \ref{sec: detail A3}) that there exists $j\in\{1,...,l\}$ such that
$$\left[x,\gamma_{\lambda,\sqrt{d}h}(x)\right]^{\sqrt{d}h}\subset \overline{M}_{i}\cup \overline{M}_{j}$$
Let $H_{1}\in C_{h,\lambda}$ be a hypercube containing $x$ and denote $x_{1}$ its center. Suppose that there exists $y\in \left[x,\gamma_{\lambda,\sqrt{d}h}(x)\right]$ such that $y\notin H_{1}$. Denote $H_{2}$ the hypercube of $C_{h}$ containing $y$ and $x_{2}$ its center. Suppose furthermore that $H_{2}$ is adjacent to $H_{1}$. If
$$H_{1}\cap M_{j}\cap \mathcal{F}_{\lambda+(\sqrt{2}\sigma N_{h}+3L\sqrt{d})h^{\alpha}}\ne\emptyset$$
as $H_{2}\subset H_{1}^{\sqrt{d}h}$, then $H_{2}\subset \mathcal{F}_{\lambda+(\sqrt{2}\sigma N_{h}+5L\sqrt{d})h^{\alpha}}$ by Assumptions \textbf{A1} and \textbf{A2} (see Lemma \ref{lmm: A1-A2 imply} in Appendix \ref{sec: A1-A2 imply}) and thus, by Lemma \ref{lmm1}, $H_{2}\subset\widehat{\mathcal{F}}_{\lambda+(2\sqrt{2}\sigma N_{h}+5\sqrt{d})h^{\alpha}}$.
From now on, we suppose that:
\begin{equation}
\label{eq: condition far 2}
H_{1}\cap M_{j}\cap \mathcal{F}_{\lambda+(\sqrt{2}\sigma N_{h}+3L\sqrt{d})h^{\alpha}}=\emptyset.
\end{equation}
As $y\in H_{2}$, we have:
\begin{align}
\left\langle \frac{\gamma_{\lambda,\sqrt{d}h}(x)-x}{||\gamma_{\lambda,\sqrt{d}h}(x)-x||_{2}},x_{2}-x_{1}\right\rangle&=\left\langle \frac{y-x}{||y-x||_2},x_{2}-x_{1}\right\rangle \geq 0.\label{eq:scal prod 2}
\end{align}
Let us denote $\overline{P}=\overline{P}(\gamma_{\lambda,\sqrt{d}h}(x))$ and $P=P(\gamma_{\lambda,\sqrt{d}h}(x)).$ Note that, by definition of $\gamma_{\lambda,\sqrt{d}h}$, $\overline{P}$ can be equivalently written as:
$$\overline{P}=\left\{z\in[0,1]^{d}\text{ s.t. }\left\langle z-\xi(x),\frac{\gamma_{\lambda,\sqrt{d}h}(x)-x}{||\gamma_{\lambda,\sqrt{d}h}(x)-x||_{2}}\right\rangle\geq 0\right\}.$$
Consequently, for all $z\in H_{1}\cap \overline{P}$, $z+(x_{2}-x_{1})\in H_{2}\cap \overline{P}$. Note that for all $z\in H_{2}\cup H_{1}$,
$$||\xi(\gamma_{\lambda,\sqrt{d}h}(x))-z||_{2}=||\xi(x)-z||_2\leq 3\sqrt{d}h.$$
Thus, take $K=3\sqrt{d}$ and let $C_{2}$ be the corresponding constant according to Lemma \ref{lemma-histo 1}. By Lemma \ref{lemma-histo 1}, if $z\in H_{1}\cap \overline{M}_{i}\setminus P^{2C_{2}h^{2}}$ then $z\in \overline{P}$. Thus, $z+(x_2-x_1)\in \overline{P}$. Furthermore, if $z+(x_2-x_1)\notin  P^{2C_{2}h^{2}}$, applying Lemma \ref{lemma-histo 1} again, we have that $z+(x_2-x_1)\in  \overline{M}_{i}$. Therefore, we have:
\begin{equation}
\label{eq: inclusion hyper 2}
\left((H_{1}\cap \overline{M}_{i})\setminus P^{2C_{2}h^{2}}+(x_2-x_1)\right)\setminus P^{2C_{2}h^{2}}\subset H_{2}\cap \overline{M}_{i}.    
\end{equation}
 As $H_{2}\subset H_{1}^{\sqrt{d}h}$, Assumptions \textbf{A1} and \textbf{A2} implies that, for all $z\in H_{2}\cap \overline{M}_{i}$, $z\in\mathcal{F}_{\lambda+3L\sqrt{d}h^{\alpha}}$ (see Lemma \ref{lmm: A1-A2 imply} in Appendix \ref{sec: A1-A2 imply}). Hence, by (\ref{eq: condition far 2}),
\begin{equation}
\label{eq: inf sup 2}
\sup\limits_{z\in H_{2}\cap \overline{M}_{i}}f(z)\leq \inf\limits_{z\in H_{1}\cap M_{j}}f(z)
\end{equation}
Now, observe that:
\begin{itemize}
    \item if $z\in H_{2}\cap \overline{M_{i}}$, either $z-(x_2-x_1)\in M_{j}\cap H_1$ and (\ref{eq: inf sup 2}) ensures that $f(z)\leq f(z-(x_2-x_1))$, or $z-(x_2-x_1)\in \overline{M}_{i}\cap H_1$ and, as $||x_1-x_2||_2\leq \sqrt{d}h$, Assumptions \textbf{A1} and \textbf{A2} ensure that $f(z)\leq f(z-(x_2-x_1))+L(\sqrt{d}h)^{\alpha}$ (Lemma \ref{lmm: A1-A2 imply} in Appendix \ref{sec: A1-A2 imply}). Thus, in both cases,
$$f(z)\leq f(z-(x_2-x_1))+L(\sqrt{d}h)^{\alpha}.$$
\item If $z\in H_{2}\cap M_{j}$ and $z-(x_2-x_1)\in M_{j}\cap H_1$, then Assumptions \textbf{A1} and \textbf{A2} ensure that $f(z)\leq f(z-(x_2-x_1))+L(\sqrt{d}h)^{\alpha}$ (Lemma \ref{lmm: A1-A2 imply} in Appendix \ref{sec: A1-A2 imply}).
\item  If $z\in H_{2}\cap M_{j}$ and $z-(x_2-x_1)\in \overline{M}_{i}\cap H_1$, then, by \eqref{eq: inclusion hyper 2}, $z\in P^{2C_{2}h^{2}}$ or $z-(x_2-x_1)\in P^{2C_{2}h^{2}}$ .
\end{itemize}
Now, let $x_{1},...,x_{n}$ be the points of $G_{1/N}$. Let $A_1=H_2\cap \overline{M}_i$, $A_2=\{z\in H_{2}\cap M_{j}$ \text{ s.t. } $z-(x_2-x_1)\in M_{j}\cap H_1\}$ and $A_3=\{z\in H_{2}\cap M_{j} \text{ s.t. }z-(x_2-x_1)\in \overline{M}_{i}\cap H_1\}$. Again, there exists a constant $\kappa$ (only depending on $d$ and $R$) such that $|A_3\cap G_{1/N}|\leq \kappa nh^{d+1}$. Using Assumption \textbf{A0}, it follows that:
\begin{align*}
\sum_{x_{i}\in H_{2}}f(x_{i})&=\sum_{x_i\in A_1}f(x_{i})+\sum_{x_i\in A_2}f(x_{i})+\sum_{x_i\in A_3}f(x_{i})\\
&\leq \sum_{x_i\in A_1}\left(f(x_{i}-(x_2-x_1))+L(\sqrt{d}h)^{\alpha}\right)\\
&\quad+\sum_{x_i\in A_2}\left(f(x_{i}-(x_2-x_1))+L(\sqrt{d}h)^{\alpha}\right)\\
&\quad+\sum_{x_i\in A_3}\left(f(x_i-(x_2-x_1))-f(x_i-(x_2-x_1)) +f(x_i)\right)\\
&\leq \sum_{x_{i}\in H_{1}}f(x_{i})+nh^{d}L(\sqrt{d}h)^{\alpha}+2nh^{d+1}\kappa M.
\end{align*}
Now, as $H_1\in C_{h,\lambda}$, recall that we have:
$$\frac{1}{nh^{d}}\sum_{x_{i}\in H_1}X_i=\frac{1}{nh^{d}}\sum_{x_{i}\in H_1}\left(f(x_{i})+\sigma\varepsilon_i\right)\leq  \lambda$$
Thus, by choice of $h$, for $n$ sufficiently large such that $h<1$,
\begin{align*}
\frac{1}{nh^{d}}\sum_{x_{i}\in H_{2}} X_{i}&=\frac{1}{nh^{d}}\left(\sum_{x_{i}\in H_{2}} f(x_{i})+\sigma\sum_{x_{i}\in H_{2}}\varepsilon_i\right)\\
&\leq \frac{1}{nh^{d}}\left(\sum_{x_{i}\in H_{1}}f(x_{i})+nh^{d}L(\sqrt{d}h)^{\alpha}+2nh^{d+1}\kappa M+\sigma\sum_{x_{i}\in H_{2}}\varepsilon_i\right)\\
&\leq \lambda + L(3\sqrt{d}h)^{\alpha}+2\kappa Mh+\frac{\sigma}{nh^{d}}\sum_{x_{i}\in H_{2}}\varepsilon_i-\frac{\sigma}{nh^{d}}\sum_{x_{i}\in H_{1}}\varepsilon_i\\
&\leq \lambda + L(3\sqrt{d}h)^{\alpha}+2\kappa Mh^{\alpha}+ 2N_{h}\sqrt{2}\sigma \sqrt{\frac{\log\left(1/h^{d}\right)}{nh^{d}}}\\
&\leq \lambda + L(3\sqrt{d}h)^{\alpha}+2\kappa Mh^{\alpha}+2\sqrt{2}\sigma N_{h}h^{\alpha}
\end{align*}
$H_{2}\subset \widehat{\mathcal{F}}_{\lambda+(2\sqrt{2}\sigma N_{h}+5L\sqrt{d}+2\kappa M)h^{\alpha}}$.
Applying again the iterative reasoning used for assertion (\ref{lemma-histo 2a}), we obtain assertion (\ref{lemma-histo 2b}).

\end{proof}
\subsection{Proof of Lemma \ref{lemma noise 2}}
This section is dedicated to the proof of Lemma \ref{lemma noise 2} from Section \ref{Upperbounds section}.\vspace{0.25cm}
\label{proof noise}
\begin{proof}
Let $h>\sfrac{1}{N}$ and $H\subset[0,1]^{d}$ be a closed hypercube of side $h$. As the $(\varepsilon_{i})_{i=1,...,n}$ are i.i.d and standard Gaussian variables, we have, for all $H\in C_{h}$,
$$\mathbb{P}\left(\left| \frac{1}{nh^{d}}\sum\limits_{x_i\in H}\sigma\varepsilon_{i}\right|\geq t\right)\leq 2\exp\left(-\frac{nh^{d}t^{2}}{2\sigma^{2}}\right).$$
And thus, as the number of point in any $H\in C_{h}$ is at least to $nh^{d}$,
$$\mathbb{P}\left(\left| \frac{1}{nh^{d}}\sum\limits_{x_i\in H}\sigma\varepsilon_{i}\right|\geq t\right)\leq 2\exp\left(-\frac{nh^{d}t^{2}}{2\sigma^{2}}\right).$$
Now, by union bound, using $\left|C_{h}\right|\leq (1/h)^{d}$,
$$\mathbb{P}\left(\max\limits_{H\in C_{h}}\left| \frac{1}{nh^{d}}\sum\limits_{x_i\in H}\sigma\varepsilon_{i}\right|\geq t\right)\leq 2\left(\frac{1}{h}\right)^{d}\exp\left(-\frac{nh^{d}t^{2}}{2\sigma^{2}}\right)$$
and thus,
$$\mathbb{P}\left(N_{h}\geq t\right)\leq 2\left(\frac{1}{h}\right)^{d}\exp\left(-t^{2}\log\left(1/h^{d}\right)\right).$$
Now, take $t\geq \sqrt{8}$, then $t^{2}/4+2\leq t^{2}$. Thus, if $h<1$
\begin{align*}
\mathbb{P}\left(N_h\geq t \right)&\leq2\left(\frac{1}{h}\right)^{d}\exp\left(-\frac{1}{2}t^{2}\log\left(1+\frac{1}{h^{d}}\right)\right)\\
&\leq 2\left(\frac{1}{h}\right)^{d}\exp\left(-(t^{2}/8+1)\log\left(1+\frac{1}{h^{d}}\right)\right)\\
&\leq 2\exp(-t^{2}\log(2)/8).
\end{align*}
Observe that, for all $t\leq \sqrt{8}$, 
$$\exp(-t^{2}\log(2)/8) \geq \exp(-\log(2))$$
and thus, for all $t\leq \sqrt{8}$,
$$2\exp(\log(2))\times\exp(-t^{2}\log(2)/8)\geq 1\geq \mathbb{P}\left(N_h\geq t \right).$$
It follows that, for all $t>0$,
$$\mathbb{P}\left(N_h\geq t \right)\leq 2\exp(\log(2))\times\exp(-t^{2}\log(2)/8).$$

\end{proof}
\section{Claims details}
\label{appendix: claim}
\subsection{Lemma \ref{lmm: A3 imply}}
\label{sec: detail A3}
This section provides details supporting a key claim that was repeatedly invoked in the previous proofs: namely, that Assumption~\textbf{A3} implies that, locally, the boundaries of regular pieces separate at most two regions. We formalize this claim in the following lemma.
\begin{lmm}
\label{lmm: A3 imply}
Let $x\in[0,1]^{d}$ and $h<R$. Under Assumption \textbf{A3}, there exists $1\leq i,j\leq l$ such that $B_2(x,h)\subset \overline{M}_i\cap \overline{M}_j$.
\end{lmm}
\begin{proof}
 Suppose, by contradiction, that \( B_2(x,h) \) intersects three or more regular pieces. Then either there exists a multiple point \( z \in B_2(x,h) \), that is, a point lying in the boundaries of at least three distinct regular pieces. Let \( M_{i_1}, M_{i_2}, M_{i_3} \) be three such pieces, so that \( z \in \partial M_{i_1} \cap \partial M_{i_2} \cap \partial M_{i_3} \). By Theorem~4.8 of \cite{Fed59}, and under Assumption~\textbf{A3}, there exist three Euclidean balls of radius \( R/2 \), denoted \( B_1, B_2, B_3 \), all containing \( z \) and such that \( B_j \subset M_{i_j} \cup \{z\} \) for all \( 1 \leq j \leq 3 \). In particular, \( B_1 \cap B_2 \cap B_3 = \{z\} \), which is a contradiction, since no three balls (with strictly positive radii) in \( \mathbb{R}^d \) can intersect at a unique point. This excludes the case of a multiple point. Now, suppose that there is no multiple point within \( B_2(x,h) \). Then there exists a regular piece \( M_{i_1} \) such that
\[
\partial M_{i_1} \cap \left( \bigcup_{j \ne i_1} \partial M_j \setminus \partial M_{i_1} \right) \cap B_2(x,h) = \emptyset,
\]
but both \( \partial M_{i_1} \cap B_2(x,h) \neq \emptyset \) and \( \left( \bigcup_{j \ne i_1} \partial M_j \setminus \partial M_{i_1} \right) \cap B_2(x,h) \neq \emptyset \). Let
\[
(y_1, y_2) \in \operatorname{argmin}_{y_1 \in \partial M_{i_1} \cap B_2(x,h),\, y_2 \in \left( \bigcup_{j \ne i_1} \partial M_j \setminus \partial M_{i_1} \right) \cap B_2(x,h)} \|y_1 - y_2\|_2,
\]
Note that $||y_1-y_2||_2>0$ and  define \( z = (y_1+ y_2)/2 \). Since \( B_2(x,h) \) is convex, \( z \in B_2(x,h) \). 

If \( z \in \bigcup_{j=1}^l \partial M_j \), then 
\[
d_2\left(\partial M_{i_1}, \bigcup_{j \ne i_1} \partial M_j \setminus \partial M_{i_1}\right) \leq \frac{\|y_1 - y_2\|_2}{2}<\|y_1 - y_2\|_2,
\]
which contradicts the minimality of \(\|y_1 - y_2\|_2 \). Therefore, using that \( d_2(z, \partial M_{i_1}) \leq\|y_1 - y_2\|_2/2 \leq h  \), we conclude that
\[
z \in B_2\left( \bigcup_{j=1}^l \partial M_j, h \right) \setminus \bigcup_{j=1}^l \partial M_j.
\]
Now, the point \( z \) admits (at least) two distinct closest points in \( \bigcup_{j=1}^l \partial M_j \), namely \( y_1 \) and \( y_2 \). If this were not the case, there would exist \( y_3 \in \bigcup_{j=1}^l \partial M_j \) such that \(\|z - y_3\|_2 <\|y_1 - y_2\|_2 / 2\), which would imply
\[
d_2\left(\partial M_{i_1}, \bigcup_{j \ne i_1} \partial M_j \setminus \partial M_{i_1}\right) \leq  \frac{\|y_1 - y_2\|_2}{2}+\|z - y_3\|_2<\|y_1 - y_2\|_2 ,
\]
contradicting again the minimality of \(\|y_1 - y_2\|_2 \). Hence, \( z \) has at least two distinct closest points in \( \bigcup_{j=1}^l \partial M_j \), which violates the reach condition from Assumption~\textbf{A3}. This completes the argument. 
\end{proof}
\subsection{Lemma \ref{lmm: A1-A2 imply}}
\label{sec: A1-A2 imply}
This section provides details supporting another key claim that was repeatedly invoked in the previous proofs involving Assumptions \textbf{A1} and \textbf{A2}, formalized as the following lemma.
\begin{lmm}
\label{lmm: A1-A2 imply}
Let $i\in\{1,...,l\}$, $\lambda\in \mathbb{R}$, $x\in \overline{M_{i}}\cap\mathcal{F}_{\lambda}$ and $h>0$. Under Assumptions \textbf{A1} and \textbf{A2}, we have $B_{2}(x,h)\cap\overline{M}_{i}\subset \mathcal{F}_{\lambda+Lh^{\alpha}}.$
\end{lmm}
\begin{proof}
By Assumption \textbf{A1} for all $y\in B_{2}(x,h)\cap M_i$, 
$$|f(x)-f(y)|\leq L||x-y||^{\alpha}_{2}\leq Lh^{\alpha}$$
Thus $B_{2}(x,h)\cap M_i\subset \mathcal{F}_{\lambda+Lh^{\alpha}}$. Moreover, if \( z \in \partial M_i \cap B_{2}(x,h)\), there exists $(y_{n})_{n\in\mathbb{N}}$ a sequence of elements of $B_{2}(x,h)\cap M_i$ converging to $z$. Then, by Assumption \textbf{A2}, we have: $$f(z)\leq \lim_{n\rightarrow+\infty}f(y_{n})\leq \lambda+Lh^{\alpha}$$
and consequently $B_{2}(x,h)\cap \overline{M}_i\subset \mathcal{F}_{\lambda+Lh^{\alpha}}$. 
\end{proof}\vspace{0.25cm}
Following the same reasoning, the result of Lemma \ref{lmm: A1-A2 imply} can also be extended, supposing only that $x\in \overline{M}_i$ and $\lim_{z\in M_i\rightarrow x}\leq \lambda$.
\bibliographystyle{plainnat}
\bibliography{bibliographie}
\end{document}